\documentclass[final,3p,times]{elsarticle}
\usepackage{amssymb,amsmath,amsthm,epsfig,verbatim,url}
\usepackage{tikz,pgfplots,tikz-3dplot}
\usetikzlibrary{fit}
\newcommand\addvmargin[1]{
\node[fit=(current bounding box),inner ysep=#1,inner xsep=0]{};
}
\usepackage[notref,notcite]{showkeys}
\newtheorem{thm}{Theorem}[section]
\newtheorem{lem}[thm]{Lemma}
\newtheorem{rem}[thm]{Remark}
\renewcommand{\theequation}{\arabic{section}.\arabic{equation}}
\newcommand{\bRplus}{{\mathbb R}_{>0}}
\newcommand{\bRgeq}{{\mathbb R}_{\geq 0}}
\newcommand{\RZ}{{\mathbb R} \slash {\mathbb Z}}
\newcommand{\bR}{{\mathbb R}}

\newcommand{\spa}{\operatorname{span}}
\newcommand{\Gauss}{{\mathcal{K}}}
\newcommand{\doctorkappa}{\mathfrak{K}}
\newcommand{\ratio}{{\mathfrak r}}
\newcommand{\dH}[1]{\;{\rm d}{\mathcal{H}}^{#1}} 
\newcommand{\dL}[1]{\;{\rm d}{\mathcal{L}}^{#1}} 

\newcommand{\drho}{\;{\rm d}\rho}
\newcommand{\spont}{{\overline\varkappa}}
\newcommand{\Vh}{\underline{V}^h}

\newcommand{\Vhpartial}{\underline{V}^h_\partial}
\newcommand{\Vhpartialzero}{\underline{V}^h_{\partial_0}}
\newcommand{\Vpartial}{\underline{V}_\partial}
\newcommand{\Vpartialzero}{\underline{V}_{\partial_0}}
\newcommand{\nabS}{\nabla_{\!\mathcal{S}}}

\newcommand{\id}{\rm id}
\newcommand{\deldel}[1]{\frac{\delta}{{\delta}#1}}
\newcommand{\dd}[1]{\frac{\rm d}{{\rm d}#1}}
\newcommand{\ddt}{\dd{t}}

\newcommand{\ek}{e}
\newcommand{\ttau}{\Delta t}
\newcommand{\sliprho}{\widehat\varrho_{\partial\mathcal{S}}}
\newcommand{\BGNsd}{\mathcal{E}}
\newcommand{\BGNsdstab}{\mathcal{F}}
\newcommand{\BGNintstab}{\mathcal{I}}
\newcommand{\BGNwf}{\mathcal{W}}
\newcommand{\normal}{{\rm n}}
\def\epsilon{\varepsilon} 

\newcommand{\errorXx}{\|\Gamma - \Gamma^h\|_{L^\infty}}
\newcommand{\pol}{\vec}  

\begin{document}
\begin{frontmatter}
\title{
Finite element methods for fourth order
axisymmetric geometric evolution equations
}
\author[1]{John W. Barrett}
\address[1]{Department of Mathematics, Imperial College London, 
London, SW7 2AZ, UK}
\ead{j.barrett@imperial.ac.uk}
\author[2]{Harald Garcke}
\address[2]{Fakult{\"a}t f{\"u}r Mathematik, Universit{\"a}t Regensburg, 
93040 Regensburg, Germany}
\ead{harald.garcke@ur.de}
\author[1]{Robert N\"urnberg\corref{cor1}}
\ead{robert.nurnberg@imperial.ac.uk}
\cortext[cor1]{Corresponding author, Telephone +44 207594857}

\begin{abstract}
Fourth order curvature driven interface evolution equations frequently
appear in the natural sciences. Often axisymmetric geometries are of
interest, and in this situation numerical computations are much more
efficient. We will introduce and analyze several new finite element
schemes for fourth order geometric evolution equations in
an axisymmetric setting,
and for selected schemes 
we will show existence, uniqueness and stability results. 
The presented schemes have very good mesh and stability properties, as will be
demonstrated by several numerical examples.
\end{abstract} 


\end{frontmatter}

\setcounter{equation}{0}
\section{Introduction} \label{sec:intro}

The motion of interfaces driven by a law for the normal velocity, 
which involves the surface Laplacian of curvature quantities, 
plays an important role in many applications. 
The resulting differential equations are parabolic and of fourth order.
Prominent examples are the surface diffusion flow, which models phase changes 
due to diffusion along an interface, see \cite{Mullins57,CahnT94}. 
In this evolution
law the normal velocity of the interface is given by the surface Laplacian 
of the mean curvature.

Typical membrane energies involve the curvature of the membrane. 
In the simplest models the Willmore functional, which is just the 
integrated squared mean curvature, is an appropriate energy,
see \cite{Willmore65}.
Recently, in particular, biomembranes have been the focus of research and in this case more complex energies, like the Canham-Helfrich energy, are of interest, see \cite{Canham70,Helfrich73,Seifert97} for details.
Taking the $L^2$--gradient flow of such an energy also leads to a fourth order 
geometric evolution equation involving the surface Laplacian of the mean 
curvature and cubic nonlinearities in the curvature, 
see \cite{Simonett01,KuwertS02}.
In the case of biological membranes also more complex laws, taking volume and
surface constraints or a coupling to fluid flow into account, are of relevance,
see \cite{willmore,nsns} and the references therein. 

In this paper
we introduce new numerical schemes for axisymmetric versions of these flows. 
This is a very relevant issue as in many situations axisymmetric shapes appear
and reducing the computations to a spatially one-dimensional problem greatly
reduces the computational complexity. Schemes for the axisymmetric problem
also have the benefit that mesh degeneracies, which for other schemes
frequently happen during
the evolution, can be avoided. We will also introduce schemes which make use of the tangential degrees of freedom in order to obtain good mesh properties. Some
of these schemes even have the property that mesh points equidistribute during the evolution.

We now specify the interface evolution laws studied in this paper in more 
detail. Let $(\mathcal{S}(t))_{t \geq 0} \subset \bR^3$ 
be a family of smooth, oriented  hypersurfaces, which we
later assume to be axisymmetric.
The mean curvature flow for $\mathcal{S}(t)$ is given by the evolution law
\begin{equation} \label{eq:mcfS}
\mathcal{V}_{\mathcal{S}} = k_m\qquad\text{ on } \mathcal{S}(t)\,,
\end{equation}
and it is the $L^2$--gradient flow for 
for the surface area.
Here $\mathcal{V}_{\mathcal{S}}$ denotes the normal velocity of 
$\mathcal{S}(t)$ in the direction of the normal $\pol\normal_{\mathcal{S}}$.
Moreover, $k_m$ is the mean curvature of $\mathcal{S}(t)$, i.e.\ the sum of the
principal curvatures of $\mathcal{S}(t)$. For the methods derived in this paper the identity
\begin{equation} \label{eq:LBid} 
\Delta_{\mathcal{S}}\,\pol\id = k_m \,
\pol\normal_{\mathcal{S}}  \qquad \text{ on } \mathcal{S}(t)
\end{equation}
will be  crucial, 
where $\Delta_{\mathcal{S}}$ is the Laplace--Beltrami operator on
$\mathcal{S}(t)$ and $\pol\id$ denotes the identity function in $\bR^3$.
A derivation of the identity (\ref{eq:LBid}) can be found in 
e.g.\ \cite{DeckelnickDE05}. 
In this paper we will consider fourth order analogues of the second order
geometric evolution equation (\ref{eq:mcfS}).

The surface diffusion flow for $\mathcal{S}(t)$ is given by the evolution law
\begin{equation} \label{eq:sdS}
\mathcal{V}_{\mathcal{S}} =- \Delta_{\mathcal{S}}\,k_m \qquad\text{on }
\mathcal{S}(t)\,.
\end{equation}
This law was introduced by Mullins, \cite{Mullins57}, 
in order to describe thermal grooving and this evolution law also has 
important applications in epitaxial growth, see e.g.\ 
\cite{GurtinJ02,BanschMN05}.

A flow combining surface diffusion and surface attachment limited
kinetics introduced in \cite{CahnT94}, and analyzed in
\cite{ElliottG97a}, is given by
\begin{equation} \label{eq:SALK}
\mathcal{V}_{\mathcal{S}} = 
- \Delta_{\mathcal{S}}\,\left(\frac1\alpha - \frac1\xi\,\Delta_{\mathcal{S}}
\right)^{-1} k_m
\quad\text{on }\ \mathcal{S}(t)\,,
\end{equation}
where $\alpha,\xi \in \bRplus$ are given parameters. This flow can
be written as
\begin{equation} \label{eq:salknew}
\mathcal{V}_{\mathcal{S}} =-\Delta_{\mathcal{S}}\, y\,, \qquad \left(
-\frac1\xi\, \Delta_{\mathcal{S}}+ \frac1\alpha \right) y =k_m 
\quad\text{on }\ \mathcal{S}(t) \,,
\end{equation}
and in the limit of fast attachment kinetics $\xi\to \infty$
and $\alpha=1$, we recover surface diffusion, (\ref{eq:sdS}). In the
limit of fast  surface diffusion  $\alpha\to\infty$
and $\xi=1$ we recover conserved mean curvature flow,
\begin{equation*} 
\mathcal{V}_{\mathcal{S}} = k_m
- \frac{\int_{\mathcal{S}} k_m \dH{2}}{\int_{\mathcal{S}} 1 \dH{2}}
\quad\text{on }\ \mathcal{S}(t)\,,
\end{equation*}
with $\mathcal{H}^{2}$ being the surface measure. 
A discussion of these limits can be found in \cite{TaylorC94}.
Hence, for general values $\alpha$, $\xi \in \bRplus$, the intermediate flow
(\ref{eq:SALK}) interpolates between surface diffusion and 
conserved mean curvature flow,
see e.g.\ \cite{ElliottG97a} and \cite[p.\ 4282]{gflows3d} for more details.

We now define the generalized Willmore energy of the surface $\mathcal{S}(t)$ as
\begin{equation} \label{eq:W}
\tfrac12\,\int_{\mathcal{S}(t)} (k_m - \spont)^2 \dH{2} \,,
\end{equation}
where $\spont\in\bR$ is a given constant, the so-called spontaneous curvature.
On $\mathcal{S}(t)$, Willmore flow, i.e.\ the 
$L^2$--gradient flow for (\ref{eq:W}), is given by
\begin{equation}\label{eq:Willmore_flow}
\mathcal{V}_{\mathcal{S}} = -\Delta_{\mathcal{S}} \,k_m
- (k_m - \spont)\,|\nabS\,\pol\normal_{\mathcal{S}}|^2
+\tfrac{1}{2}\,(k_m - \spont)^2\,k_m 
= -\Delta_{\mathcal{S}} \,k_m + 2\, (k_m- \spont)\,k_g
-\tfrac{1}{2}\,(k_m^2 - \spont^2)\,k_m \qquad\text{on }\ \mathcal{S}(t)\,.
\end{equation}
Here $\nabS\,\pol\normal_{\mathcal{S}}$ is the Weingarten map
and $k_g$ is the Gaussian curvature of $\mathcal{S}(t)$, i.e.\ it is the 
product of the two principal curvatures. We also consider Helfrich flow,
which is the volume and surface area preserving variant of
(\ref{eq:Willmore_flow}).

In this paper, we consider the case that $\mathcal{S}(t)$ is 
an axisymmetric surface, that is rotationally symmetric with respect to the
$x_2$--axis. We further assume that $\mathcal{S}(t)$ is made up of a single 
connected component, with or without boundary. Clearly, in the latter case 
the boundary $\partial\mathcal{S}(t)$ of $\mathcal{S}(t)$ consists of either
one or two circles that each lie within a hyperplane that is parallel to the
$x_1-x_3$--plane. For the evolving family of surfaces we allow for the
following types of boundary conditions. A boundary circle may assumed to be
fixed, it may be allowed to move vertically along the boundary of a fixed
infinite cylinder that is aligned with the axis of rotation, 
or it may be allowed to expand and shrink within a 
hyperplane that is parallel to the $x_1-x_3$--plane. 
Depending on the postulated free energy, certain
angle conditions will arise where $\mathcal{S}(t)$ meets the external 
boundary. If the free energy is just surface area,
$\mathcal{H}^2(\mathcal{S}(t))$,
then a $90^\circ$ degree contact angle
condition arises. We refer to Section~\ref{sec:1} below for further details,
in particular with regard to more general contact angles.

Numerical analysis of geometric evolution equations has been an
active field in the last thirty years and we refer to 
\cite{DeckelnickDE05} for an overview. 
Approaches using parametric finite element methods
have heavily relied on ideas of Gerd Dziuk, who first used a weak formulation
of (\ref{eq:LBid}) in order to compute the mean curvature, see 
\cite{Dziuk88,Dziuk91}. The present authors have used the tangential degrees 
of freedom to improve the mesh quality during the evolution of discretized 
curvature flows, see \cite{triplej,triplejMC,gflows3d,willmore}. 
There has been interest in numerical schemes for axisymmetric
schemes for geometric evolution equations both for second and for 
fourth order flows, see 
\cite{NicholsM65b,Nichols76,BernoffBW98,ColemanFM96,%
CoxL15,DeckelnickDE03,DeckelnickS10,SudohHK13,Zhao17preprint}.
However, the literature on
numerical analysis of such schemes is sparse. For exceptions we refer to
\cite{DeckelnickDE03,DeckelnickS10} in the context of graph formulations for
surface diffusion and Willmore flow, respectively. Axisymmetric versions
of 
geometric flows have also been treated analytically and questions
regarding stability and singularity formation have been studied, see
\cite{Huisken90,DziukK91,BasaSS94,ColemanFM95,Kohsaka16}. 
We also refer to \cite{DallAcquaS17preprint,DallAcquaS18},
who discuss the relation between the axisymmetric Willmore flow and the elastic
flow in hyperbolic space.

The structure of this work is as follows. In Section~\ref{sec:1} we introduce
weak formulations for fourth order axisymmetric geometric flows,
which all involve a splitting
into two second order equations. The weak formulations are essential for
the discretization with the help of piecewise linear, continuous 
finite elements. Spatially discretized semidiscrete schemes, 
based on these weak formulations, are introduced in Section~\ref{sec:sd}. 
Fully discrete schemes are introduced
in Section~\ref{sec:fd} and for some of the schemes existence, uniqueness 
and stability results are shown. 
Finally, in Section~\ref{sec:nr} numerical results for surface diffusion, 
for the intermediate law (\ref{eq:SALK}), for Willmore flow and for 
Helfrich flow are presented. 
The results demonstrate the stability and good mesh properties discussed 
in the preceding sections and the ideas presented in this paper hence 
have the potential to work also for more complex dynamics like
the evolution of biomembranes in flows, see e.g.\ the setting in \cite{nsns}.

\setcounter{equation}{0}
\section{Weak formulations} \label{sec:1}
\begin{figure}
\center
\newcommand{\AxisRotator}[1][rotate=0]{%
    \tikz [x=0.25cm,y=0.60cm,line width=.2ex,-stealth,#1] \draw (0,0) arc (-150:150:1 and 1);%
}
\begin{tikzpicture}[every plot/.append style={very thick}, scale = 1]
\begin{axis}[axis equal,axis line style=thick,axis lines=center, xtick style ={draw=none}, 
ytick style ={draw=none}, xticklabels = {}, 
yticklabels = {}, 
xmin=-0.2, xmax = 0.8, ymin = -0.4, ymax = 2.55]
after end axis/.code={  
   \node at (axis cs:0.0,1.5) {\AxisRotator[rotate=-90]};
   \draw[blue,->,line width=2pt] (axis cs:0,0) -- (axis cs:0.5,0);
   \draw[blue,->,line width=2pt] (axis cs:0,0) -- (axis cs:0,0.5);
   \node[blue] at (axis cs:0.5,-0.2){$\pol\ek_1$};
   \node[blue] at (axis cs:-0.2,0.5){$\pol\ek_2$};
   \draw[red,very thick] (axis cs: 0,0.7) arc[radius = 70, start angle= -90, end angle= 90];
   \node[red] at (axis cs:0.7,1.9){$\Gamma$};
}
\end{axis}
\end{tikzpicture} \qquad \qquad
\tdplotsetmaincoords{120}{50}
\begin{tikzpicture}[scale=2, tdplot_main_coords,axis/.style={->},thick]
\draw[axis] (-1, 0, 0) -- (1, 0, 0);
\draw[axis] (0, -1, 0) -- (0, 1, 0);
\draw[axis] (0, 0, -0.2) -- (0, 0, 2.7);
\draw[blue,->,line width=2pt] (0,0,0) -- (0,0.5,0) node [below] {$\pol\ek_1$};
\draw[blue,->,line width=2pt] (0,0,0) -- (0,0.0,0.5);
\draw[blue,->,line width=2pt] (0,0,0) -- (0.5,0.0,0);
\node[blue] at (0.2,0.4,0.1){$\pol\ek_3$};
\node[blue] at (0,-0.2,0.3){$\pol\ek_2$};
\node[red] at (0.7,0,1.9){$\mathcal{S}$};
\node at (0.0,0.0,2.4) {\AxisRotator[rotate=-90]};

\tdplottransformmainscreen{0}{0}{1.4}
\shade[tdplot_screen_coords, ball color = red] (\tdplotresx,\tdplotresy) circle (0.7);
\end{tikzpicture}
\caption{Sketch of $\Gamma$ and $\mathcal{S}$, as well as 
the unit vectors $\pol\ek_1$, $\pol\ek_2$ and $\pol\ek_3$.}
\label{fig:sketch}
\end{figure}

Let $\RZ$ be the periodic interval $[0,1]$, and set
\[
I = \RZ\,, \text{ with } \partial I = \emptyset\,,\quad \text{or}\quad
I = (0,1)\,, \text{ with } \partial I = \{0,1\}\,.
\]
We consider the axisymmetric situation, where 
$\pol x(t) : \overline I \to \bR^2$ 
is a parameterization of $\Gamma(t)$. 
Throughout $\Gamma(t)$ represents the generating curve of a
surface $\mathcal{S}(t)$ 
that is axisymmetric with respect to the $x_2$--axis, see
Figure~\ref{fig:sketch}. In particular, on defining
\begin{equation*} 
\pol\Pi_3^3(r, z, \theta) = 
(r\,\cos\theta, z, r\,\sin\theta)^T 
\quad\text{for}\quad r\in \bRgeq\,,\ z \in \bR\,,\ \theta \in [0,2\,\pi]
\end{equation*}
and
\begin{equation*} 
\Pi_2^3(r, z) = \{\pol\Pi_3^3(r, z, \theta) : \theta \in [0,2\,\pi)\}\,,
\end{equation*}
we have that
\begin{equation} \label{eq:SGamma}
\mathcal{S}(t) = 
\bigcup_{(r,z)^T \in \Gamma(t)} \Pi_2^3(r, z)
= \bigcup_{\rho \in \overline I} \Pi_2^3(\pol x(\rho,t))\,.
\end{equation}
Here we allow $\Gamma(t)$ to be either a closed curve, parameterized over
$\RZ$, which corresponds to $\mathcal{S}(t)$ being a genus-1 surface
without boundary.
Or $\Gamma(t)$ may be an open curve, parameterized over $[0,1]$.
Then $\Gamma(t)$ has two endpoints, and each endpoint can either correspond to
an interior point of $\mathcal{S}(t)$, or to a boundary circle of
$\mathcal{S}(t)$. Endpoints of $\Gamma(t)$ that correspond to an interior point
of the surface $\mathcal{S}(t)$ are attached to the $x_2$--axis, 
on which they can freely move up and down. For example, if both endpoints
of $\Gamma(t)$ are attached to the $x_2$--axis,
then $\mathcal{S}(t)$ is a genus-0 surface without boundary.
If only one end of $\Gamma(t)$ is attached to the $x_2$--axis, 
then $\mathcal{S}(t)$ is an open surface with boundary, where the boundary
consists of a single connected component.
If no endpoint of $\Gamma(t)$ is attached to the $x_2$--axis, 
then $\mathcal{S}(t)$ is an open surface with boundary, where the boundary
consists of two connected components. 

In particular, we always assume that, for all $t \in [0,T]$,
\begin{subequations}
\begin{align} 
\pol x(\rho,t) \,.\,\pol\ek_1 & > 0 \quad 
\forall\ \rho \in \overline I\setminus \partial_0 I\,,\label{eq:xpos} \\
\pol x(\rho,t) \,.\,\pol\ek_1 &= 0 \quad 
\forall\ \rho \in \partial_0 I\,,\label{eq:axibc} \\
\pol x_t(\rho,t) \,.\,\pol\ek_i &= 0 \quad 
\forall\ \rho \in \partial_i I \,, \ i =1,2\,,  \label{eq:freeslipbc} \\
\pol x_t(\rho,t) &= \pol 0 \quad 
\forall\ \rho \in \partial_D I \,, \label{eq:noslipbc}
\end{align}
\end{subequations}
where $\partial_D I \cup \bigcup_{i=0}^2 \partial_i I = \partial I$ 
is a disjoint partitioning of $\partial I$, with $\partial_0 I$
denoting the subset of boundary points 
of $I$ that correspond to endpoints of $\Gamma(t)$ attached to the 
$x_2$--axis. Moreover, $\partial_D I \cup \bigcup_{i=1}^2 \partial_i I$ 
denotes the subset of boundary points of $I$ that model components of the
boundary of $\mathcal{S}(t)$. Here endpoints in $\partial_D I$ correspond
to fixed boundary circles of $\mathcal{S}(t)$, 
that lie within a hyperplane parallel to
the $x_1-x_3$--plane $\bR \times \{0\} \times \bR$.
Endpoints in $\partial_1 I$ correspond to boundary circles of $\mathcal{S}(t)$
that can move freely along the boundary of an infinite cylinder
that is aligned with the axis of rotation.
Endpoints in $\partial_2 I$ correspond to boundary circles of $\mathcal{S}(t)$
that can expand/shrink freely within a hyperplane parallel to
the $x_1-x_3$--plane $\bR \times \{0\} \times \bR$.
See Table~\ref{tab:diagram} for a visualization of the different types of 
boundary nodes.
\begin{table}
\center
\caption{The different types of boundary nodes enforced by 
(\ref{eq:axibc})--(\ref{eq:noslipbc}).}
\begin{tabular}{ccc}
\hline
$\partial I$ & $\partial \Gamma$ & $\partial\mathcal{S}$ \\ \hline
$\partial_0 I$ &
\begin{tikzpicture}[scale=0.5,baseline=40]
\begin{axis}[axis equal,axis line style=thick,axis lines=center, 
xtick style ={draw=none}, ytick style ={draw=none}, xticklabels = {}, 
yticklabels = {}, xmin=-0.1, xmax = 2, ymin = -2, ymax = 2]
\addplot[mark=*,color=blue,mark size=6pt] coordinates {(0,1)};
\draw[<->,line width=3pt,color=red] (axis cs:0.3,0.5) -- (axis cs:0.3,1.5);
\node at (axis cs:0.5,-0.3){\Large$\pol\ek_1$};
\node at (axis cs:-0.3,0.5){\Large$\pol\ek_2$};
\end{axis}
\addvmargin{1mm}
\end{tikzpicture} 
& N/A \\ 
$\partial_D I$ &
\begin{tikzpicture}[scale=0.5,baseline=40]
\begin{axis}[axis equal,axis line style=thick,axis lines=center, 
xtick style ={draw=none}, ytick style ={draw=none}, xticklabels = {}, 
yticklabels = {}, xmin=-0.1, xmax = 2, ymin = -2, ymax = 2]
\addplot[mark=*,color=blue,mark size=6pt] coordinates {(2,1)};
\node at (axis cs:0.5,-0.3){\Large$\pol\ek_1$};
\node at (axis cs:-0.3,0.5){\Large$\pol\ek_2$};
\end{axis}
\addvmargin{1mm}
\end{tikzpicture} 
& 
\begin{tikzpicture}[baseline=0]
\draw[color=blue,thick] (0,0) circle [x radius=2cm, y radius=1cm];
\addvmargin{1mm}
\end{tikzpicture} 
\\ 
$\partial_1 I$ &
\begin{tikzpicture}[scale=0.5,baseline=40]
\begin{axis}[axis equal,axis line style=thick,axis lines=center, 
xtick style ={draw=none}, ytick style ={draw=none}, xticklabels = {}, 
yticklabels = {}, xmin=-0.1, xmax = 2, ymin = -2, ymax = 2]
\addplot[mark=*,color=blue,mark size=6pt] coordinates {(2,1)};
\draw[<->,line width=3pt,color=red] (axis cs:2.3,0.5) -- (axis cs:2.3,1.5);
\draw[thick,color=blue] (axis cs:2,-2) -- (axis cs:2,2);
\node at (axis cs:0.5,-0.3){\Large$\pol\ek_1$};
\node at (axis cs:-0.3,0.5){\Large$\pol\ek_2$};
\end{axis}
\addvmargin{1mm}
\end{tikzpicture} 
& 
\begin{tikzpicture}[baseline=0]
\draw[color=blue,thick] (0,0) circle [x radius=2cm, y radius=1cm];
\draw[<->,color=red,line width=2pt] (2.2,-0.5) -- (2.2,0.5);
\draw[color=blue,thin] (2,-1) -- (2,1);
\draw[color=blue,thin] (-2,-1) -- (-2,1);
\addvmargin{1mm}
\end{tikzpicture} 
\\ 
$\partial_2 I$ &
\begin{tikzpicture}[scale=0.5,baseline=40]
\begin{axis}[axis equal,axis line style=thick,axis lines=center, 
xtick style ={draw=none}, ytick style ={draw=none}, xticklabels = {}, 
yticklabels = {}, xmin=-0.1, xmax = 2, ymin = -2, ymax = 2]
\addplot[mark=*,color=blue,mark size=6pt] coordinates {(2,1)};
\draw[<->,line width=3pt,color=red] (axis cs:1.5,0.7) -- (axis cs:2.5,0.7);
\draw[thick,color=blue] (axis cs:0,1) -- (axis cs:4,1);
\node at (axis cs:0.5,-0.3){\Large$\pol\ek_1$};
\node at (axis cs:-0.3,0.5){\Large$\pol\ek_2$};
\end{axis}
\addvmargin{1mm}
\end{tikzpicture} 
& 
\begin{tikzpicture}[baseline=0]
\draw[color=blue,thick] (0,0) circle [x radius=2cm, y radius=1cm];
\draw[color=blue,thin] (0,0) circle [x radius=1.5cm, y radius=0.75cm];
\draw[color=blue,thin] (0,0) circle [x radius=2.5cm, y radius=1.25cm];
\draw[<->,color=red,line width=2pt] (1.5,0) -- (2.5,0);
\draw[<->,color=red,line width=2pt] (-2.5,0) -- (-1.5,0);
\addvmargin{1mm}
\end{tikzpicture} \\ \hline
\end{tabular}
\label{tab:diagram}
\end{table}%

On assuming that
\begin{equation} \label{eq:xrho}
|\pol x_\rho| \geq c_0 > 0 \qquad \forall\ \rho \in \overline I\,,
\end{equation}
we introduce the arclength $s$ of the curve, i.e.\ $\partial_s =
|\pol{x}_\rho|^{-1}\,\partial_\rho$, and set
\begin{equation} \label{eq:tau}
\pol\tau(\rho,t) = \pol x_s(\rho,t) = 
\frac{\pol x_\rho(\rho,t)}{|\pol x_\rho(\rho,t)|} \qquad \mbox{and}
\qquad \pol\nu(\rho,t) = -[\pol\tau(\rho,t)]^\perp,
\end{equation}
where $(\cdot)^\perp$ denotes a clockwise rotation by $\frac{\pi}{2}$.

On recalling (\ref{eq:SGamma}), we observe that the normal
$\pol\normal_{\mathcal{S}}$ on $\mathcal{S}(t)$ is given by
\begin{equation} \label{eq:nuS}
\pol\normal_{\mathcal{S}}(\pol\Pi_3^3(\pol x(\rho,t),\theta)) = 
\pol\nu_{\mathcal{S}}(\rho,\theta,t) = 
\begin{pmatrix}
(\pol\nu(\rho,t)\,.\,\pol\ek_1)\,\cos\theta \\
\pol\nu(\rho,t)\,.\,\pol\ek_2 \\
(\pol\nu(\rho,t)\,.\,\pol\ek_1)\,\sin\theta 
\end{pmatrix}
 \quad\text{for}\quad
\rho \in \overline I\,,\ 
\theta \in [0,2\,\pi)
\end{equation}
and $t\in[0,T]$.
Similarly, the normal velocity $\mathcal{V}_{\mathcal{S}}$ of $\mathcal{S}(t)$ 
in the direction $\pol\normal_{\mathcal{S}}$ is given by
\begin{equation*} 
\mathcal{V}_{\mathcal{S}} = \pol x_t(\rho,t)\,.\,\pol\nu(\rho,t) \quad\text{on }
\Pi_2^3(\pol x(\rho,t)) \subset \mathcal{S}(t)\,,
\quad \forall\ \rho \in \overline I\,,\ t \in [0,T]\,.
\end{equation*}

For the curvature $\varkappa$ of $\Gamma(t)$ it holds that
\begin{equation} \label{eq:varkappa}
\varkappa\,\pol\nu = \pol\varkappa = \pol\tau_s =
\frac1{|\pol x_\rho|} \left[ \frac{\pol x_\rho}{|\pol x_\rho|} \right]_\rho.
\end{equation}

An important role in this paper is played by the surface area of the 
surface $\mathcal{S}(t)$, which is equal to
\begin{equation} \label{eq:A}
\mathcal{H}^2(\mathcal{S}(t)) = A(\pol x(t)) = 
2\,\pi\,\int_I \pol x(\rho,t)\,.\,\pol\ek_1\,|\pol x_\rho(\rho,t)|
\drho\,.
\end{equation}
Often the surface area, $A(\pol x(t))$, 
will play the role of the free energy in our paper. 
But for an open surface $\mathcal{S}(t)$, with boundary 
$\partial\mathcal{S}(t)$, 
we consider contact energy contributions which are discussed in
\cite{Finn86}, see also \cite[(2.21)]{ejam3d}. 
In the axisymmetric setting the relevant energy is given by
\begin{equation} \label{eq:E}
E(\pol x(t)) = A(\pol x(t))
+ 2\,\pi\,\sum_{p\in\partial_1 I} 
\sliprho^{(p)}\,(\pol x(p,t)\,.\,\pol\ek_1)\,\pol x(p,t)\,.\,\pol\ek_2
+ \pi\,\sum_{p\in\partial_2 I} 
\sliprho^{(p)}\,(\pol x(p,t)\,.\,\pol\ek_1)^2\,,
\end{equation}
where we recall from (\ref{eq:freeslipbc}) that, for $i=1,2$, either 
$\partial_i I = \emptyset$, $\{0\}$, $\{1\}$ or $\{0,1\}$.
In the above $\sliprho^{(p)} \in \bR$, for $p\in \{0,1\}$, are given constants.
Here $\sliprho^{(p)}$, for $p \in \partial_1 I$, 
denotes the change in contact energy
density in the direction of $-\pol\ek_2$, that the two phases separated by
the interface $\mathcal{S}(t)$ have with the infinite cylinder at
the boundary circle of $\mathcal{S}(t)$ represented by $\pol x(p,t)$. 
Similarly, $\sliprho^{(p)}$, for $p \in \partial_2 I$,
denotes the change in contact energy
density in the direction of $-\pol\ek_1$, that the two phases separated by
the interface $\mathcal{S}(t)$ have with the hyperplane $\bR\times\{0\}\bR$ 
at the boundary circle of $\mathcal{S}(t)$ represented by
$\pol x(p,t)$. 
These changes in contact energy lead to the contact angle conditions
\begin{subequations} \label{eq:mcbc}
\begin{align} 
(-1)^p\,\pol\tau(p,t)\,.\,\pol\ek_2 &= \sliprho^{(p)} 
\qquad p \in \partial_1 I\,,\label{eq:mcbc1} \\
(-1)^p\,\pol\tau(p,t)\,.\,\pol\ek_1 &= \sliprho^{(p)} 
\qquad p \in \partial_2 I\,,\label{eq:mcbc2}
\end{align}
\end{subequations}
for all $t \in (0,T]$.
In most cases, the contact energies are assumed to be
the same, so that $\sliprho^{(0)}=\sliprho^{(1)}=0$, which leads to
$90^\circ$ contact angle conditions in (\ref{eq:mcbc}), and means that
(\ref{eq:E}) collapses to (\ref{eq:A}).
See \cite{ejam3d} for more details on contact angles and contact energies.
We note that a necessary condition to
admit a solution to (\ref{eq:mcbc1}) or to (\ref{eq:mcbc2}) is that 
$|\sliprho^{(p)}| \leq 1$, but we do allow for more general values in
(\ref{eq:E}). 
In addition, we observe that the energy (\ref{eq:E}) is not bounded from below
if $\sliprho^{(p)} \not=0 $ for $p\in\partial_1 I$ or if
$\sliprho^{(p)}<0$ for $p\in\partial_2 I$.

For later use we note that
\begin{align}
\ddt\, E(\pol x(t)) & = 2\,\pi\,\int_I \left[\pol x_t\,.\,\pol\ek_1
+ \pol x\,.\,\pol\ek_1\,\frac{(\pol x_t)_\rho\,.\,\pol x_\rho}{|\pol x_\rho|^2}
\right] |\pol x_\rho| \drho 
+ 2\,\pi\,\sum_{p\in\partial_1 I} 
\sliprho^{(p)}\left[
(\pol x_t(p,t)\,.\,\pol\ek_1)\,\pol x(p,t)\,.\,\pol\ek_2
+ (\pol x(p,t)\,.\,\pol\ek_1)\,\pol x_t(p,t)\,.\,\pol\ek_2
\right]
\nonumber \\ & \qquad
+ 2\,\pi\,\sum_{p\in\partial_2 I} 
\sliprho^{(p)}\,(\pol x(p,t)\,.\,\pol\ek_1)\,\pol x_t(p,t)\,.\,\pol\ek_1
\,.
\label{eq:dEdt}
\end{align}
Moreover, we recall that expressions for the mean curvature and the 
Gaussian curvature of $\mathcal{S}(t)$ are given by
\begin{equation} \label{eq:meanGaussS}
\varkappa_{\mathcal{S}} = 
\varkappa - 
\frac{\pol\nu\,.\,\pol\ek_1}{\pol x\,.\,\pol\ek_1}
\quad\text{and}\quad
\Gauss_{\mathcal{S}} = -
\varkappa\,\frac{\pol\nu\,.\,\pol\ek_1}{\pol x\,.\,\pol\ek_1}
\quad\text{on }\ \overline{I}\,,
\end{equation}
respectively; see e.g.\ \cite[(6)]{CoxL15}. More precisely, if $k_m$ and $k_g$
denote the mean and Gaussian curvatures of $\mathcal{S}(t)$, then
\begin{equation} \label{eq:kmkg}
k_m = \varkappa_{\mathcal{S}}(\rho,t) 
\ \text{ and }\
k_g = \Gauss_{\mathcal{S}}(\rho,t) 
\quad\text{on }
\Pi_2^3(\pol x(\rho,t)) \subset \mathcal{S}(t)\,,
\quad \forall\ \rho \in \overline I\,,\ t \in [0,T]\,.
\end{equation}
In the literature, the two terms making up $\varkappa_{\mathcal{S}}$
in (\ref{eq:meanGaussS}) 
are often referred to as in-plane and azimuthal curvatures,
respectively, with their sum being equal to the mean curvature.
We note that combining (\ref{eq:meanGaussS}) and (\ref{eq:varkappa}) yields 
that
\begin{equation} \label{eq:kappaS}
\varkappa_{\mathcal{S}}\,\pol\nu = 
\varkappa\,\pol\nu 
- \frac{\pol\nu\,.\,\pol\ek_1}{\pol x\,.\,\pol\ek_1}\,\pol\nu
=
\frac1{|\pol x_\rho|} \left[ \frac{\pol x_\rho}{|\pol x_\rho|} \right]_\rho
- \frac{\pol\nu\,.\,\pol\ek_1}{\pol x\,.\,\pol\ek_1}\,\pol\nu\,,
\end{equation}
see also (\ref{eq:LBSvecy}) in Appendix~\ref{sec:B}.
It follows from (\ref{eq:kappaS}) that
\begin{equation} \label{eq:kappaSnew}
(\pol x \,.\,\pol\ek_1)\,\varkappa_{\mathcal{S}}\,\pol\nu
 = (\pol x \,.\,\pol\ek_1)\,\pol\tau_s
+ (\pol\tau\,.\,\pol\ek_1)\,\pol\tau - \pol\ek_1
= [(\pol x \,.\,\pol\ek_1)\,\pol\tau]_s  - \pol\ek_1 
= [(\pol x \,.\,\pol\ek_1)\,\pol x_s]_s  - \pol\ek_1\,.
\end{equation}
A weak formulation of (\ref{eq:kappaSnew}) will form the basis of our stable
approximations for surface diffusion, (\ref{eq:sdS}), and the intermediate flow
(\ref{eq:SALK}). 
Clearly, for a smooth surface with bounded mean curvature it follows from
(\ref{eq:kappaS}) that 
\begin{equation} \label{eq:bcnu}
\pol\nu(\rho,t) \,.\,\pol\ek_1 = 0
\qquad \forall\ \rho \in \partial_0 I\,,\quad \forall\ t\in[0,T]\,,
\end{equation}
which is clearly equivalent to
\begin{equation} \label{eq:bc}
\pol x_\rho(\rho,t) \,.\,\pol\ek_2 = 0
\qquad \forall\ \rho \in \partial_0 I\,,\quad \forall\ t\in[0,T]\,.
\end{equation}
A precise derivation of (\ref{eq:bc}) in the context of a weak formulation
of (\ref{eq:kappaS}) can be found in \cite[Appendix~A]{aximcf}.

We observe that it follows from (\ref{eq:bcnu}) and (\ref{eq:varkappa}) 
that
\begin{equation} 
\lim_{\rho\to \rho_0}
\frac{\pol\nu(\rho,t)\,.\,\pol\ek_1}{\pol x(\rho,t)\,.\,\pol\ek_1} 
= \lim_{\rho\to \rho_0} 
\frac{\pol\nu_\rho(\rho,t)\,.\,\pol\ek_1}{\pol x_\rho(\rho,t)\,.\,\pol\ek_1}
= \pol\nu_s(\rho_0,t)\,.\,\pol\tau(\rho_0,t) 
= -\varkappa(\rho_0,t)
\qquad \forall\ \rho_0\in\partial_0 I\,,\
\forall\ t \in [0,T]\,.
\label{eq:bclimit}
\end{equation}

\subsection{Surface diffusion}
On recalling (\ref{eq:LBSrad}) from Appendix~\ref{sec:B}, we note that
in the axisymmetric parameterization of $\mathcal{S}(t)$,
the flow (\ref{eq:sdS}) can be written as
\begin{equation} \label{eq:xtsd}
(\pol x\,.\,\pol\ek_1)\,\pol x_t\,.\,\pol\nu = 
- \left[\pol x\,.\,\pol\ek_1\,[\varkappa_{\mathcal{S}} ]_s \right]_s
\qquad\text{on } I\,,
\end{equation}
with, on recalling (\ref{eq:axibc})--(\ref{eq:noslipbc}),
\begin{equation} \label{eq:xtbc}
\pol x_t(\rho,t)\,.\,\pol\ek_1 = 0 \quad
\forall\ \rho\in\partial_0 I\,, \quad 
\pol x_t(\rho,t)\,.\,\pol\ek_i = 0 \quad
\forall\ \rho\in\partial_i I\,,\ i = 1,2\,, \quad 
\pol x_t(\rho,t) = \pol 0 \quad \forall\ \rho\in\partial_D I\,, 
\qquad \forall\ t\in[0,T]\,,
\end{equation}
as well as (\ref{eq:bc}), (\ref{eq:mcbc}) and
\begin{equation} \label{eq:sdbca}
(\varkappa_{\mathcal{S}})_\rho  (\rho,t) = 
\left(\varkappa- \frac{\pol\nu\,.\,\pol\ek_1}
{\pol x\,.\,\pol\ek_1} \right)_\rho (\rho,t) = 0 \qquad 
\forall\ \rho \in \partial I\,, \qquad\forall\ t\in(0,T]\,.
\end{equation}
Here (\ref{eq:sdbca}) for $\rho \in \partial_0 I$ ensures that the radially
symmetric function $k_m$, recall (\ref{eq:kmkg}), on $\mathcal{S}(t)$ 
induced by $\varkappa_{\mathcal{S}}(t)$ is differentiable.
For $\rho \in \partial_1 I \cup \partial_2 I \cup \partial_D I$ 
the condition (\ref{eq:sdbca}) can be interpreted
as a no-flux condition. 
We remark that (\ref{eq:xtsd}) agrees with \cite[(2)]{NicholsM65b}.

Let 
$\Vpartialzero = \{ \pol\eta \in [H^1(I)]^2 : \pol\eta(\rho)\,.\,\pol\ek_1 = 0
\quad \forall\ \rho \in \partial_0 I\}$ and
$\Vpartial = \{ \Vpartialzero : \pol\eta(\rho)\,.\,\pol\ek_i = 0
\quad \forall\ \rho \in \partial_i I\,,\ i=1,2, \
\pol\eta(\rho) = \pol 0 \quad \forall\ \rho \in \partial_D I\}$.
Then we consider the following weak formulation of
(\ref{eq:xtsd}) and (\ref{eq:varkappa}), on recalling (\ref{eq:meanGaussS}).

$(\BGNsd)$:
Let $\pol x(0) \in \Vpartialzero$. For $t \in (0,T]$
find $\pol x(t) \in [H^1(I)]^2$, with $\pol x_t(t) \in \Vpartial$, 
and $\varkappa(t)\in H^1(I)$ such that
\begin{subequations} \label{eq:sdweak}
\begin{align}
& \int_I  (\pol x\,.\,\pol\ek_1)\,\pol x_t\,.\,\pol\nu\,\chi\,|\pol x_\rho|\drho
= \int_I  \pol x\,.\,\pol\ek_1 
\left(\varkappa - \frac{\pol\nu\,.\,\pol\ek_1}{\pol x\,.\,\pol\ek_1}
\right)_\rho\, \chi_\rho\,|\pol x_\rho|^{-1} \drho 
\qquad \forall\ \chi \in H^1(I)\,, \label{eq:sdweaka} \\
& \int_I \varkappa\,\pol\nu\,.\,\pol\eta\, |\pol x_\rho| \drho
+ \int_I (\pol x_\rho\,.\,\pol\eta_\rho)\, |\pol x_\rho|^{-1} \drho
= - \sum_{i=1}^2 
\sum_{p \in \partial_i I} \sliprho^{(p)}\,\pol\eta(p)\,.\,\pol\ek_{3-i}
\qquad \forall\ \pol\eta \in \Vpartial\,.
\label{eq:sdweakb}
\end{align}
\end{subequations}
We note that (\ref{eq:sdweakb}) weakly imposes (\ref{eq:bc})
and (\ref{eq:mcbc}), while it is immediately clear that 
(\ref{eq:sdweaka}) weakly imposes (\ref{eq:sdbca}) on 
$\partial I \setminus \partial_0 I$.
The degenerate weight $\pol x\,.\,\pol\ek_1$ on the right hand side in
(\ref{eq:sdweaka}) means that it is not obvious 
that (\ref{eq:sdweaka}) weakly imposes (\ref{eq:sdbca}) on $\partial_0 I$.
Hence we rigorously derive in Appendix~\ref{sec:A}
that (\ref{eq:sdweaka}) does indeed weakly impose
(\ref{eq:sdbca}) on $\partial_0 I$.

Let $\mathcal{L}^3$ denote the Lebesgue measure in $\bR^3$. Then
choosing $\chi = 2\,\pi$ in (\ref{eq:sdweaka}) yields
\begin{equation} \label{eq:dVdtsd}
\pm \ddt\,\mathcal{L}^3(\Omega(t)) = 
\int_{\mathcal{S}(t)} \mathcal{V}_{\mathcal{S}} \dH{2}
= 2\,\pi\,\int_I (\pol x\,.\,\pol\ek_1)
\,\pol x_t\,.\,\pol\nu\,|\pol x_\rho|\drho = 0\,,
\end{equation}
where 
$\mathcal{S}(t) = \partial\Omega(t)$, 
and where the sign in (\ref{eq:dVdtsd}) depends on whether 
$\pol\normal_{\mathcal{S}}$ is the outer or inner normal
to $\Omega(t)$ on $\mathcal{S}(t)$, recall (\ref{eq:nuS}). 
Moreover, choosing $\chi = 
\varkappa - \frac{\pol\nu\,.\,\pol\ek_1}{\pol x\,.\,\pol\ek_1}$ in 
(\ref{eq:sdweaka}) and $\pol\eta = \pol x_t$ in (\ref{eq:sdweakb}) yields,
on recalling (\ref{eq:dEdt}) and (\ref{eq:xpos}), that
\begin{equation} \label{eq:sdweakstab}
\frac1{2\,\pi}\,\ddt\, E(\pol x(t)) 
= -\int_I \pol x \,.\,\pol\ek_1
\left|\left[\varkappa - \frac{\pol\nu\,.\,\pol\ek_1}{\pol x\,.\,\pol\ek_1}
\right]_\rho\right|^2 |\pol x_\rho|^{-1} \drho \leq 0\,.
\end{equation}
It does not appear possible to mimic the proof of (\ref{eq:sdweakstab}) on the
discrete level. Hence we also introduce the following 
alternative formulation for surface diffusion, which treats the mean
curvature $\varkappa_{\mathcal{S}}(t)$ of $\mathcal{S}(t)$ as an unknown. 

$(\BGNsdstab)$:
Let $\pol x(0) \in \Vpartialzero$. For $t \in (0,T]$
find $\pol x(t) \in [H^1(I)]^2$, with $\pol x_t(t) \in \Vpartial$, 
and $\varkappa_{\mathcal{S}}(t)\in H^1(I)$ such that
\begin{subequations} \label{eq:sdstabab}
\begin{align}
& \int_I  (\pol x\,.\,\pol\ek_1)\,\pol x_t\,.\,\pol\nu\,\chi\,|\pol x_\rho|\drho
= \int_I  \pol x\,.\,\pol\ek_1\,
(\varkappa_{\mathcal{S}})_\rho\, \chi_\rho\,|\pol x_\rho|^{-1} \drho 
\qquad \forall\ \chi \in H^1(I)\,, \label{eq:sdstaba} \\
& \int_I \pol x\,.\,\pol\ek_1\,
\varkappa_{\mathcal{S}}\,\pol\nu\,.\,\pol\eta\, |\pol x_\rho| \drho
+\int_I \left[\pol\eta \,.\,\pol\ek_1
+ \pol x\,.\,\pol\ek_1\,\frac{\pol x_\rho\,.\,\pol\eta_\rho}{|\pol x_\rho|^2}
\right] |\pol x_\rho| \drho 
= - \sum_{i=1}^2 \sum_{p \in \partial_i I} \sliprho^{(p)}\,
(\pol x(p,t)\,.\,\pol\ek_1)\,\pol\eta(p)\,.\,\pol\ek_{3-i}
\qquad \forall\ \pol\eta \in \Vpartial\,.
\label{eq:sdstabb}
\end{align}
\end{subequations}
We note that (\ref{eq:sdstabb}) 
weakly imposes (\ref{eq:bc}) and (\ref{eq:mcbc}), 
while (\ref{eq:sdstaba}) weakly imposes 
(\ref{eq:sdbca}), recall (\ref{eq:meanGaussS}), 
where for the case $\partial_0 I \not= \emptyset$ we refer to
Appendix~\ref{sec:A}.

Choosing $\chi = 2\,\pi$ in (\ref{eq:sdstaba}) yields (\ref{eq:dVdtsd}), as
before. Moreover, choosing $\chi = \varkappa_{\mathcal{S}}$ in 
(\ref{eq:sdstaba}) and $\pol\eta = \pol x_t$ in (\ref{eq:sdstabb}) yields,
on recalling (\ref{eq:dEdt}), that
\begin{equation} \label{eq:sdstab}
\frac1{2\,\pi}\,\ddt\, E(\pol x(t))
= -\int_I \pol x \,.\,\pol\ek_1\,
|(\varkappa_{\mathcal{S}})_\rho|^2\,|\pol x_\rho|^{-1} \drho \leq 0\,.
\end{equation}
In contrast to (\ref{eq:sdweakstab}), it will be possible to mimic the proof
of (\ref{eq:sdstab}) on the discrete level.

\subsection{Intermediate evolution law}
In the axisymmetric parameterization of $\mathcal{S}(t)$,
the flow (\ref{eq:salknew}) can be written, similarly to (\ref{eq:xtsd}), as
\begin{equation} \label{eq:xtSALK}
(\pol x\,.\,\pol\ek_1)\,\pol x_t\,.\,\pol\nu = 
- \left[\pol x\,.\,\pol\ek_1\,y_s \right]_s\,,\quad
-\tfrac1\xi\,[\pol x\,.\,\pol\ek_1\,y_s]_s + \tfrac1\alpha\,\pol
x\,.\,\pol\ek_1\,y = x\,.\,\pol\ek_1\,\varkappa_{\mathcal{S}}
\qquad\text{on } I\,,
\end{equation}
with (\ref{eq:xtbc}), as well as (\ref{eq:bc}), (\ref{eq:mcbc}) and
\begin{equation} \label{eq:intbca}
y_\rho  (\rho,t) = 0\qquad
\forall\ \rho \in \partial I\,, \qquad\forall\ t\in(0,T]\,.
\end{equation}

It is straightforward to adapt the formulations $(\BGNsd)$ and $(\BGNsdstab)$
to (\ref{eq:xtSALK}). For example, generalizing $(\BGNsdstab)$ to 
(\ref{eq:xtSALK}) yields the following weak formulation.

$(\BGNintstab)$:
Let $\pol x(0) \in \Vpartialzero$. For $t \in (0,T]$
find $\pol x(t) \in [H^1(I)]^2$, with $\pol x_t(t) \in \Vpartial$, 
and $(y(t), \varkappa_{\mathcal{S}}(t))\in [H^1(I)]^2$ such that
\begin{subequations} \label{eq:intstababc}
\begin{align}
& \int_I (\pol x\,.\,\pol\ek_1)\,\pol x_t\,.\,\pol\nu\,\chi\,|\pol x_\rho|\drho
= \int_I \pol x\,.\,\pol\ek_1\, y_\rho\, \chi_\rho\,|\pol x_\rho|^{-1} \drho 
\qquad \forall\ \chi \in H^1(I)\,, \label{eq:intstaba} \\
&
\frac1\xi\, \int_I \pol x\,.\,\pol\ek_1 \, y_\rho\, \zeta_\rho\,
|\pol x_\rho|^{-1} \drho
+ \int_I \pol x\,.\,\pol\ek_1 \, \left[ \alpha^{-1}\,y -
\varkappa_{\mathcal{S}} \right] \zeta\,|\pol x_\rho| \drho = 0
\qquad \forall\ \zeta \in H^1(I)\,, \label{eq:intstabb}\\
& \int_I \pol x\,.\,\pol\ek_1\,
\varkappa_{\mathcal{S}}\,\pol\nu\,.\,\pol\eta\, |\pol x_\rho| \drho
+\int_I \left[\pol\eta \,.\,\pol\ek_1
+ \pol x\,.\,\pol\ek_1\,\frac{\pol x_\rho\,.\,\pol\eta_\rho}{|\pol x_\rho|^2}
\right] |\pol x_\rho| \drho 
= - \sum_{i=1}^2 \sum_{p \in \partial_i I} \sliprho^{(p)}\,
(\pol x(p,t)\,.\,\pol\ek_1)\,\pol\eta(p)\,.\,\pol\ek_{3-i}
\qquad \forall\ \pol\eta \in \Vpartial\,.
\label{eq:intstabc}
\end{align}
\end{subequations}
The weak formulation of (\ref{eq:xtSALK}) corresponding to $(\BGNsd)$ is given 
by (\ref{eq:intstaba}), (\ref{eq:sdweakb}) and (\ref{eq:intstabb}) with
$\varkappa_{\mathcal{S}}$ replaced by the expression in (\ref{eq:meanGaussS}). 
We note that (\ref{eq:intstabc}) 
weakly imposes (\ref{eq:bc}) and (\ref{eq:mcbc}), 
while (\ref{eq:intstaba}) and (\ref{eq:intstabb}) weakly impose
(\ref{eq:intbca}), 
where for the case $\partial_0 I \not= \emptyset$ we refer once again to 
Appendix~\ref{sec:A}.

Choosing $\chi = 2\,\pi$ in (\ref{eq:intstaba}) yields (\ref{eq:dVdtsd}), as
before. Moreover, choosing $\chi = \frac\alpha\xi\,\varkappa_{\mathcal{S}}$ in 
(\ref{eq:intstaba}), $\zeta = \alpha\,\varkappa_{\mathcal{S}} - y_s$ in
(\ref{eq:intstabb}) and $\pol\eta = \frac\alpha\xi\,\pol x_t$ in 
(\ref{eq:intstabc}) yields, similarly to (\ref{eq:sdstab}), that
\begin{equation} \label{eq:intstab}
\frac\alpha\xi\,\frac1{2\,\pi}\,\ddt\, E(\pol x(t)) 
= - \frac\alpha\xi\,\int_I \pol x \,.\,\pol\ek_1\,
y_\rho\,(\varkappa_{\mathcal{S}})_\rho\,|\pol x_\rho|^{-1} \drho 
= - \frac1\xi\,\int_I \pol x \,.\,\pol\ek_1\,
|y_\rho|^2\,|\pol x_\rho|^{-1} \drho 
- \alpha\,\int_I \pol x \,.\,\pol\ek_1\,
|\varkappa_{\mathcal{S}} - \tfrac1\alpha\,y|^2\,|\pol x_\rho| \drho 
\leq 0\,.
\end{equation}

\subsection{Willmore flow}
It holds that the Willmore energy of the surface $\mathcal{S}(t)$,
recall (\ref{eq:W}), can be written as
\begin{equation*} 
W(\pol x(t)) = \tfrac12\,\int_{\mathcal{S}(t)} (k_m - \spont)^2 \dH{2} 
= \pi\,\int_I \pol x\,.\,\pol\ek_1
\,(\varkappa_S - \spont)^2 \, |\pol x_\rho| \drho\,, 
\end{equation*}
see also \cite[(6),(7)]{CoxL15}. 
Noting once more (\ref{eq:LBSrad}) from Appendix~\ref{sec:B},
a strong formulation for the flow (\ref{eq:Willmore_flow}) on $I$ is given by
\begin{equation} \label{eq:xtbgn}
(\pol x\,.\,\pol\ek_1)\,\pol x_t\,.\,\pol\nu = 
- \left[\pol x\,.\,\pol\ek_1\,[\varkappa_{\mathcal{S}}]_s \right]_s
+ 2\,\pol x\,.\,\pol\ek_1\,[\varkappa_{\mathcal{S}} - \spont
]\,\Gauss_{\mathcal{S}}
- \tfrac12\,\pol x\,.\,\pol\ek_1
\left(\varkappa_{\mathcal{S}}^2 -\spont^2\right)
\varkappa_{\mathcal{S}}
\quad\text{on }\ I\,,
\end{equation}
with (\ref{eq:bc}), 
(\ref{eq:sdbca}) and $\pol x_t(\rho,t)\,.\,\pol\ek_1 = 0$ for 
$\rho\in\partial_0 I = \partial I$, $t\in[0,T]$.
Here we stress that for
Willmore flow we always assume that $\partial_0 I = \partial I$.
That is because it does not appear possible to model Willmore flow 
for open surfaces in the weak formulation (\ref{eq:bgnweak}), below.
The reason is that the relevant boundary conditions,
i.e.\ clamped, Navier, semi-free or free, 
see e.g.\ \cite[p.\ 1706]{pwfopen}, that would need to be
enforced for $\pol x_t$, cannot be enforced through this weak formulation
in the open curve case. Instead, techniques 
as in \cite{pwfopen} are needed here, and we will consider the details 
in the forthcoming paper \cite{axipwf}.

Then we consider the following weak formulation of
(\ref{eq:xtbgn}) and (\ref{eq:varkappa}), on recalling (\ref{eq:meanGaussS}).

$(\BGNwf)$:
Let $\pol x(0) \in \Vpartialzero$. For $t \in (0,T]$
find $\pol x(t) \in [H^1(I)]^2$, with $\pol x_t(t) \in \Vpartial$, 
and $\varkappa(t)\in H^1(I)$ such that
\begin{subequations} \label{eq:bgnweak}
\begin{align}
& \int_I (\pol x\,.\,\pol\ek_1)\,\pol x_t\,.\,\pol\nu\,\chi 
\,|\pol x_\rho|\drho
 = \int_I \pol x\,.\,\pol\ek_1 
\left[\varkappa - \frac{\pol\nu\,.\,\pol\ek_1}{\pol x\,.\,\pol\ek_1}
\right]_\rho \chi_\rho\, |\pol x_\rho|^{-1} \drho
- 2 \int_I \left[\varkappa - 
\frac{\pol\nu\,.\,\pol\ek_1}{\pol x\,.\,\pol\ek_1} - \spont\right]
\varkappa\,\pol\nu\,.\,\pol\ek_1\, \chi\,|\pol x_\rho| \drho
\nonumber \\ & \hspace{4cm}
- \tfrac12 \int_I
\pol x\,.\,\pol\ek_1
\left(
\left[\varkappa - \frac{\pol\nu\,.\,\pol\ek_1}{\pol x\,.\,\pol\ek_1}\right]^2
-\spont^2\right)
\left[\varkappa - \frac{\pol\nu\,.\,\pol\ek_1}{\pol x\,.\,\pol\ek_1}\right]
 \chi\,|\pol x_\rho| \drho
\qquad \forall\ \chi \in H^1(I)\,,
\label{eq:bgnweaka} \\
& \int_I \varkappa\,\pol\nu\,.\,\pol\eta\,|\pol x_\rho| \drho
+  \int_I \pol x_\rho\,.\,\pol\eta_\rho\,|\pol x_\rho|^{-1} \drho
 = 0
\qquad \forall\ \pol\eta \in \Vpartial\,.
\label{eq:bgnweakb}
\end{align}
\end{subequations}
We note that the two last terms on the right hand side of (\ref{eq:bgnweaka}) 
give no contribution at the boundary $\partial I = \partial_0 I$, 
since $\pol\nu\,.\,\pol\ek_1 = \pol x\,.\,\pol\ek_1 = 0$ there.
We also note that (\ref{eq:bgnweakb}) weakly imposes (\ref{eq:bc}). 
Similarly to (\ref{eq:sdweaka}), we note that
(\ref{eq:bgnweaka}) weakly imposes (\ref{eq:sdbca}), see 
\cite[Appendix~A]{aximcf} for details in the case $\rho \in \partial_0 I$.

We note that in contrast to surface diffusion, a weak formulation 
for Willmore flow based on 
$\varkappa_S$, i.e.\ (\ref{eq:sdstabb}), has no benefits over the presented
formulation (\ref{eq:bgnweak}). Due to the presence of Gaussian curvature,
recall (\ref{eq:Willmore_flow}) and (\ref{eq:meanGaussS}), 
a weak formulation based on (\ref{eq:sdstabb})
would still involve the singular fraction 
$\frac{\pol\nu\,.\,\pol\ek_1}{\pol x\,.\,\pol\ek_1}$,
since $\pol x\,.\,\pol\ek_1\,\Gauss_{\mathcal{S}} = 
-(\varkappa_{\mathcal{S}}+\frac{\pol\nu\,.\,\pol\ek_1}{\pol x\,.\,\pol\ek_1}) 
\, \pol\nu\,.\,\pol\ek_1$.
Moreover, and in contrast to a formulation with (\ref{eq:bgnweakb}),
discretizations based on such a formulation would exhibit tangential motion
of vertices that does not lead to equidistribution, and which for linear fully
discrete schemes may lead to a breakdown of the scheme.

\subsubsection{Helfrich flow}
Helfrich flow is given as the surface area and volume preserving variant of 
(\ref{eq:Willmore_flow}). Its strong formulation can be written as
\begin{equation} \label{eq:Helfrich_flow}
\mathcal{V}_{\mathcal{S}} = -\Delta_{\mathcal{S}} \,k_m + 2\,
(k_m- \spont)\,k_g
-\tfrac{1}{2}\,(k_m^2 - \spont^2)\,k_m
+ \lambda_A\,k_m + \lambda_V
\quad\text{on }\ \mathcal{S}(t)\,,
\end{equation}
where $(\lambda_A(t),\lambda_V(t))^T \in \bR^2$ are chosen such that 
\begin{equation} \label{eq:sideSAV}
\mathcal{H}^2(\mathcal{S}(t)) = \mathcal{H}^2(\mathcal{S}(0))\,,\qquad
\mathcal{L}^3(\Omega(t)) = \mathcal{L}^3(\Omega(0))\,.
\end{equation}
On writing (\ref{eq:bgnweaka}) as
\[
\int_I (\pol x\,.\,\pol\ek_1)\,\pol x_t\,.\,\pol\nu\,\chi\,|\pol x_\rho| 
\drho
 - \int_I \pol x\,.\,\pol\ek_1 
\left[\varkappa - \frac{\pol\nu\,.\,\pol\ek_1}{\pol x\,.\,\pol\ek_1}
\right]_\rho \chi_\rho \,|\pol x_\rho|^{-1} \drho
= \int_I f\, \chi \,|\pol x_\rho|\drho
\]
a weak formulation of Helfrich flow is given as follows.

$(\BGNwf^{A,V})$:
Let $\pol x(0) \in \Vpartialzero$. For $t \in (0,T]$
find $\pol x(t) \in [H^1(I)]^2$, with $\pol x_t(t) \in \Vpartial$, 
and $\varkappa(t)\in H^1(I)$ such that
\begin{align} \label{eq:wfweaka}
& \int_I (\pol x\,.\,\pol\ek_1)\,\pol x_t\,.\,\pol\nu\,\chi\,|\pol x_\rho|
\drho
 - \int_I \pol x\,.\,\pol\ek_1 
\left[\varkappa - \frac{\pol\nu\,.\,\pol\ek_1}{\pol x\,.\,\pol\ek_1}
\right]_\rho \chi_\rho \,|\pol x_\rho|^{-1} \drho
\nonumber \\ & \hspace{1cm}
 = \int_I f\, \chi \,|\pol x_\rho| \drho
+ \lambda_A \int_I \pol x\,.\,\pol\ek_1
\left[\varkappa - \frac{\pol\nu\,.\,\pol\ek_1}{\pol x\,.\,\pol\ek_1}
\right] \chi \,|\pol x_\rho| \drho
+ \lambda_V\int_I\pol x\,.\,\pol\ek_1\, \chi \,|\pol x_\rho| \drho
\qquad \forall\ \chi \in H^1(I)
\end{align}
and (\ref{eq:bgnweakb}) hold, with $(\lambda_A(t),\lambda_V(t))^T \in \bR^2$ 
chosen such that (\ref{eq:sideSAV}) hold.

\setcounter{equation}{0}
\section{Semidiscrete schemes} \label{sec:sd}

Let $[0,1]=\cup_{j=1}^J I_j$, $J\geq3$, be a
decomposition of $[0,1]$ into intervals given by the nodes $q_j$,
$I_j=[q_{j-1},q_j]$. 
For simplicity, and without loss of generality,
we assume that the subintervals form an equipartitioning of $[0,1]$,
i.e.\ that 
\begin{equation} \label{eq:Jequi}
q_j = j\,h\,,\quad \mbox{with}\quad h = J^{-1}\,,\qquad j=0,\ldots, J\,.
\end{equation}
Clearly, if $I=\RZ$ we identify $0=q_0 = q_J=1$.

The necessary finite element spaces are defined as follows:
$V^h = \{\chi \in C(\overline I) : \chi\!\mid_{I_j} \
\text{is linear}\ \forall\ j=1\to J\}$ and 
$\Vh = [V^h]^2$,  
$\Vhpartialzero = \Vh \cap \Vpartialzero$,
$\Vhpartial = \Vh \cap \Vpartial$.
We also define
$W^h = V^h$,
$W^h_{\partial_0} = \{ \chi \in V^h : \chi(\rho) = 0
\quad \forall\ \rho \in \partial_0 I\}$,
$\underline W^h = \underline V^h$, 
$\underline W^h_{\partial_0} = [W^h_{\partial_0}]^2$.
Let $\{\chi_j\}_{j=j_0}^J$ denote the standard basis of $V^h$,
where $j_0 = 0$ if $I = (0,1)$ and $j_0 = 1$ if $I=\RZ$.
For later use, we let $\pi^h:C(\overline I)\to V^h$ 
be the standard interpolation operator at the nodes $\{q_j\}_{j=0}^J$.

Let $(\cdot,\cdot)$ denote the $L^2$--inner product on $I$, and 
define the mass lumped $L^2$--inner product $(f,g)^h$,
for two piecewise continuous functions, with possible jumps at the 
nodes $\{q_j\}_{j=1}^J$, via
\begin{equation}
( f, g )^h = \tfrac12\sum_{j=1}^J h_j\,
\left[(f\,g)(q_j^-) + (f\,g)(q_{j-1}^+)\right],
\label{eq:ip0}
\end{equation}
where we define
$f(q_j^\pm)=\underset{\delta\searrow 0}{\lim}\ f(q_j\pm\delta)$.
The definition (\ref{eq:ip0}) naturally extends to vector valued functions.

Let $(\pol X^h(t))_{t\in[0,T]}$, with $\pol X^h(t)\in \Vhpartialzero$, 
be an approximation to $(\pol x(t))_{t\in[0,T]}$ and define
$\Gamma^h(t) = \pol X^h(t)(\overline I)$. Throughout this section we assume
that
\begin{equation*} 
\pol X^h(\rho,t) \,.\,\pol\ek_1 > 0 \quad 
\forall\ \rho \in \overline I\setminus \partial_0 I\,,
\qquad \forall\ t \in [0,T]\,.
\end{equation*}
Assuming that $|\pol{X}^h_\rho| > 0$ almost everywhere on $I$,
and similarly to (\ref{eq:tau}), we set
\begin{equation} \label{eq:tauh}
\pol\tau^h = \pol X^h_s = \frac{\pol X^h_\rho}{|\pol X^h_\rho|} 
\qquad \mbox{and} \qquad \pol\nu^h = -(\pol\tau^h)^\perp\,.
\end{equation}
For later use, we let $\pol\omega^h \in \underline V^h$ be the mass-lumped 
$L^2$--projection of $\pol\nu^h$ onto $\underline V^h$, i.e.\
\begin{equation} \label{eq:omegah}
\left(\pol\omega^h, \pol\varphi \, |\pol X^h_\rho| \right)^h 
= \left( \pol\nu^h, \pol\varphi \, |\pol X^h_\rho| \right)
= \left( \pol\nu^h, \pol\varphi \, |\pol X^h_\rho| \right)^h
\qquad \forall\ \pol\varphi\in\underline V^h\,.
\end{equation}

Recall that
\begin{equation} \label{eq:Ah}
A(\pol Z^h) = 2\,\pi\left(\pol Z^h\,.\,\pol\ek_1 ,
|\pol Z^h_\rho|\right) \quad \pol Z^h \in \Vhpartialzero
\end{equation}
and
\begin{equation} \label{eq:Eh}
E(\pol X^h(t)) = A(\pol X^h(t)) 
+ 2\,\pi \sum_{p\in \partial_1 I} 
\sliprho^{(p)}\,(\pol X^h(p,t)\,.\,\pol\ek_1)\,\pol X^h(p,t)\,.\,\pol\ek_2
+ \pi \sum_{p\in \partial_2 I} 
\sliprho^{(p)}\,(\pol X^h(p,t)\,.\,\pol\ek_1)^2\,.
\end{equation}
We have, similarly to (\ref{eq:dEdt}), that
\begin{align}
\ddt\, E(\pol X^h(t)) & = 2\,\pi \left( \left[\pol X^h_t\,.\,\pol\ek_1
+ \pol X^h\,.\,\pol\ek_1 
\,\frac{(\pol X^h_t)_\rho\,.\,\pol X^h_\rho}
{|\pol X^h_\rho|^2} \right], |\pol X^h_\rho| \right)
\nonumber \\ & \qquad
+ 2\,\pi\,\sum_{p\in \partial_1 I} 
\sliprho^{(p)}\left[(\pol X^h_t(p,t)\,.\,\pol\ek_1)\,
\pol X^h(p,t)\,.\,\pol\ek_2 + 
(\pol X^h(p,t)\,.\,\pol\ek_1)\,
\pol X^h_t(p,t)\,.\,\pol\ek_2 \right]
\nonumber \\ & \qquad
+ 2\,\pi\,\sum_{p\in \partial_2 I} 
\sliprho^{(p)}\,(\pol X^h(p,t)\,.\,\pol\ek_1)\,
\pol X^h_t(p,t)\,.\,\pol\ek_1\,.
\label{eq:dEhdt}
\end{align}

In view of the degeneracy on the right hand side of (\ref{eq:kappaS}), and on
recalling (\ref{eq:bclimit}) and (\ref{eq:omegah}), we introduce,
given a $\kappa^h(t) \in V^h$, the function 
$\doctorkappa^h(\kappa^h(t),t) \in V^h$ such that
\begin{equation} \label{eq:calKh}
[\doctorkappa^h (\kappa^h(t),t)](q_j) = \begin{cases}
\dfrac{\pol\omega^h(q_j,t)\,.\,\pol\ek_1}{\pol X^h(q_j,t)\,.\,\pol\ek_1}
& q_j \in \overline I \setminus \partial_0 I\,, \\
- \kappa^h(q_j,t) & q_j \in \partial_0 I\,.
\end{cases}
\end{equation}

\subsection{Surface diffusion}

Our semidiscrete finite element approximation of $(\BGNsd)$,
(\ref{eq:sdweak}), is given as follows.

$(\BGNsd_h)^{(h)}$:
Let $\pol X^h(0) \in \Vhpartialzero$. For $t \in (0,T]$
find $\pol X^h(t) \in \Vh$, with $\pol X^h_t(t) \in \Vhpartial$, and
$\kappa^h(t) \in V^h$ such that
\begin{subequations} \label{eq:sdsd}
\begin{align}
&
\left((\pol X^h\,.\,\pol\ek_1)\,
\pol X^h_t, \chi\,\pol\nu^h\,|\pol X^h_\rho|\right)^{(h)}
= \left(\pol X^h\,.\,\pol\ek_1 \left[
\kappa^h - \doctorkappa^h(\kappa^h)\right]_\rho, 
\chi_\rho\,|\pol X^h_\rho|^{-1}\right)
\qquad \forall\ \chi \in V^h\,, \label{eq:sdsda}\\
&
\left(\kappa^h\,\pol\nu^h, \pol\eta\,|\pol X^h_\rho|\right)^{(h)}
+ \left(\pol X^h_\rho, \pol\eta_\rho\,|\pol X^h_\rho|^{-1}\right) 
= - \sum_{i=1}^2
\sum_{p \in \partial_i I} \sliprho^{(p)}\,\pol\eta(p)\,.\,\pol\ek_{3-i}
\qquad \forall\ \pol\eta \in \Vhpartial\,.
\label{eq:sdsdb}
\end{align}
\end{subequations}
Here, and throughout, we use the notation $\cdot^{(h)}$ to denote an 
expression with or without the superscript $h$. I.e.\ the scheme
$(\BGNsd_h)^h$ employs mass lumping on some terms, recall (\ref{eq:ip0}), while 
the scheme $(\BGNsd_h)$ employs true integration throughout.
We stress that the side condition (\ref{eq:sdsdb}), for $(\BGNsd_h)^{h}$,
leads to an equidistribution property; see Remark~\ref{rem:equid} below.

For later use we observe that
\begin{align} \label{eq:Vh}
\mathcal{L}^3(\Omega^h(t)) & = 2\,\pi\,\int_{A^h(t)} \pol\id\,.\,\pol\ek_1
\dL{2} = \pi\,\int_{A^h(t)} \nabla\,.\,
\left[(\pol\id\,.\,\pol\ek_1)^2\,\pol\ek_1\right] \dL{2}
\nonumber \\ &
= \pi\,\int_{\Gamma^h(t)} (\pol\id\,.\,\pol\ek_1)^2\,
\pol\nu^h\,.\,\pol\ek_1 \dH{1} 
= \pi\,\int_I (\pol X^h\,.\,\pol\ek_1)^2\,\pol\nu^h\,.\,\pol\ek_1 \,
|\pol X^h_\rho| \drho\,,
\end{align}
where $A^h(t) \subset \bR^2$ denotes the domain enclosed by 
$\Gamma^h(t) = \pol X^h(\overline I)$, and where $\pol\nu^h(t)$ denotes the 
outer normal to $A^h(t)$ on $\partial A^h(t) = \Gamma^h(t)$. 
Of course, $\Omega^h(t) \subset \bR^3$ denotes 
the domain that is enclosed by the three-dimensional axisymmetric surface 
$\mathcal{S}^h(t)$ that is generated by the curve $\Gamma^h(t)$, i.e.\
$\mathcal{S}^h(t) = \partial\Omega^h(t)$.
Moreover, on recalling (\ref{eq:dVdtsd}), we note that
\begin{equation} \label{eq:dVhdt}
\ddt\,\mathcal{L}^3(\Omega^h(t)) = 
\int_{\mathcal{S}^h(t)} \mathcal{V}^h_{\mathcal{S}^h} \dH{2}
= 2\,\pi\left( \pol X^h\,.\,\pol\ek_1, 
\pol X^h_t\,.\,\pol\nu^h\,|\pol X^h_\rho| \right),
\end{equation}
where $\mathcal{V}^h_{\mathcal{S}^h}(t)$ denotes the normal velocity of
$\mathcal{S}^h(t)$ in the direction of $\pol\nu^h_{\mathcal{S}^h}(t)$,
the outer normal to $\Omega^h(t)$ on $\mathcal{S}^h(t)$.

Choosing $\chi = 1$ in (\ref{eq:sdsda}) yields that
\begin{equation} \label{eq:constchi}
\left(\pol X^h\,.\,\pol\ek_1,
\pol X^h_t\,.\, \pol\nu^h\,|\pol X^h_\rho|\right)^{(h)} = 0\,.
\end{equation}
Comparing (\ref{eq:dVhdt}) and (\ref{eq:constchi}), we observe that due to
mass lumping being employed in (\ref{eq:sdsda}) for
$(\BGNsd_h)^{h}$, it is not possible to
prove exact volume conservation for $(\BGNsd_h)^{h}$.
On the other hand, for the semidiscrete scheme $(\BGNsd_h)$
we obtain exact volume preservation.
We note that in practice the fully discrete variants of 
both $(\BGNsd_h)^{h}$ and $(\BGNsd_h)$, for reasonable
meshes, have excellent volume conserving properties. 

Our semidiscrete finite element approximation of $(\BGNsdstab)$,
(\ref{eq:sdstabab}), is given as follows.

$(\BGNsdstab_h)^{(h)}$:
Let $\pol X^h(0) \in \Vhpartialzero$. For $t \in (0,T]$
find $\pol X^h(t) \in \Vh$, with $\pol X^h_t(t) \in \Vhpartial$, and
$\kappa_{\mathcal{S}}^h(t) \in V^h$ such that
\begin{subequations} \label{eq:sdsds}
\begin{align}
&
\left((\pol X^h\,.\,\pol\ek_1)\,
\pol X^h_t, \chi\,\pol\nu^h\,|\pol X^h_\rho|\right)^{(h)}
= \left(\pol X^h\,.\,\pol\ek_1 \left[
\kappa_{\mathcal{S}}^h \right]_\rho, 
\chi_\rho\,|\pol X^h_\rho|^{-1}\right)
\qquad \forall\ \chi \in V^h\,, \label{eq:sdsdsa}\\
&
\left(\pol X^h\,.\,\pol\ek_1\,
\kappa_{\mathcal{S}}^h\,\pol\nu^h, \pol\eta\,|\pol X^h_\rho|\right)^{(h)}
+ \left( \pol\eta \,.\,\pol\ek_1, |\pol X^h_\rho|\right)
+ \left( (\pol X^h\,.\,\pol\ek_1)\,
\pol X^h_\rho,\pol\eta_\rho\, |\pol X^h_\rho|^{-1} \right) 
= - \sum_{i=1}^2
\sum_{p \in \partial_i I} \sliprho^{(p)}\,
(\pol X^h(p,t)\,.\,\pol\ek_1)\,\pol\eta(p)\,.\,\pol\ek_{3-i}
\qquad \forall\ \pol\eta \in \Vhpartial\,.
\label{eq:sdsdsb}
\end{align}
\end{subequations}
Choosing $\chi = 1$ in (\ref{eq:sdsdsa}), on recalling (\ref{eq:dVhdt}),
yields exact volume conservation for the scheme $(\BGNsdstab_h)$.
Moreover, in contrast to $(\BGNsd_h)^{(h)}$,
it is possible to prove a stability bound for 
$(\BGNsdstab_h)^{(h)}$. 
To this end, choose $\chi = \kappa_{\mathcal{S}}^h$ in
(\ref{eq:sdsdsa}) and $\pol\eta = \pol X^h_t$ in (\ref{eq:sdsdsb}) to obtain,
on recalling (\ref{eq:dEhdt}), that
\begin{equation*} 
\ddt\, E(\pol X^h(t)) 
= - 2\,\pi \left(
\pol X^h\,.\,\pol\ek_1\,|(\kappa_{\mathcal{S}}^h)_\rho|^2, 
|\pol X^h_\rho|^{-1} \right) \leq 0\,. 
\end{equation*}

\begin{rem} \label{rem:equid}
Let $\pol{h}_j(t) = \pol{X}^h(q_j,t) - \pol{X}^h(q_{j-1},t)$ 
for $j = 1,\ldots, J$, and set $\pol h_0 = \pol h_J$ if 
$\partial I = \emptyset$. Then, if 
$(\pol X^h(t), \kappa^h(t)) \in \Vh \times V^h$ satisfies
{\rm (\ref{eq:sdsdb})}, for $(\BGNsd_h)^{h}$, 
it holds that
\begin{equation}
|\pol{h}_j(t)| = |\pol{h}_{j - 1}(t)| \quad \mbox{if} \quad 
\pol{h}_j(t) \nparallel \pol{h}_{j - 1}(t) \quad 
\begin{cases} 
j = 1 ,\ldots, J & \partial I = \emptyset\,, \\
j = 2 ,\ldots, J & \partial I \not= \emptyset\,.
\end{cases}
\label{eq:equid}
\end{equation}
The equidistribution property {\rm (\ref{eq:equid})} can be shown by choosing
$\pol\eta = \chi_{j-1}\,[\pol\omega^h(q_{j-1},t)]^\perp \in \Vhpartial$ 
in {\rm (\ref{eq:sdsdb})}, recall {\rm (\ref{eq:omegah})}. 
See also \cite[Remark~2.4]{triplej} for more details.
We stress that {\rm (\ref{eq:sdsdsb})}, even for $(\BGNsdstab_h)^{h}$, 
does not lead to an equidistribution property for $\Gamma^h(t)$.
\end{rem}

For the reader's convenience, Table~\ref{tab:schemesSD} summarises the
main properties of all the schemes introduced in Section~\ref{sec:sd}.
\begin{table}
\center
\caption{Properties of the different semidiscrete schemes for the evolution
laws (\ref{eq:sdS}), (\ref{eq:SALK}), (\ref{eq:Willmore_flow}) and
(\ref{eq:Helfrich_flow}).
Note that subscripts refer to
semidiscretization, whereas superscripts indicate numerical integration, recall
(\ref{eq:ip0}).
}
\begin{tabular}{llccc}
\hline
\multicolumn{2}{c}{scheme} & flow
 & stability proof & equidistribution \\ \hline
$(\BGNsd_h)^h$ / $(\BGNsd_h)$  & (\ref{eq:sdsd}) & (\ref{eq:sdS}) & no &  yes / no \\
$(\BGNsdstab_h)^h$ / $(\BGNsdstab_h)$ & (\ref{eq:sdsds}) & (\ref{eq:sdS}) &  yes & no \\
$(\BGNintstab_h)^h$ / $(\BGNintstab_h)$ & (\ref{eq:sdint}) & (\ref{eq:SALK}) 
 & yes & no \\
$(\BGNwf_h)^h$ & (\ref{eq:bgnsd}) & (\ref{eq:Willmore_flow}) & no & yes \\
$(\BGNwf^{A,V}_h)^h$ &  (\ref{eq:wfsda}), (\ref{eq:lambdamuh}) 
& (\ref{eq:Helfrich_flow}) & no & yes \\
\end{tabular}
\label{tab:schemesSD}
\end{table}%

\subsection{Intermediate evolution law}

It is straightforward to adapt the semidiscrete schemes
$(\BGNsd)^h$ and $(\BGNsdstab)^{(h)}$ to the flow (\ref{eq:salknew}). 
For example, a semidiscrete finite element approximation of $(\BGNintstab)$,
(\ref{eq:intstababc}), that is based on $(\BGNsdstab_h)^{(h)}$, 
is given as follows.

$(\BGNintstab_h)^{(h)}$:
Let $\pol X^h(0) \in \Vhpartialzero$. For $t \in (0,T]$
find $\pol X^h(t) \in \Vh$, with $\pol X^h_t(t) \in \Vhpartial$, and
$(Y^h(t),\kappa_{\mathcal{S}}^h(t)) \in [V^h]^2$ such that
\begin{subequations} \label{eq:sdint}
\begin{align}
&
\left((\pol X^h\,.\,\pol\ek_1)\,
\pol X^h_t, \chi\,\pol\nu^h\,|\pol X^h_\rho|\right)^{(h)}
= \left(\pol X^h\,.\,\pol\ek_1\, Y^h_\rho, \chi_\rho\,|\pol X^h_\rho|^{-1}\right)
\qquad \forall\ \chi \in V^h\,, \label{eq:sdinta}\\
&
\tfrac1\xi \left(\pol X^h\,.\,\pol\ek_1 \, Y^h_\rho, 
\zeta_\rho\,|\pol X^h_\rho|^{-1}\right)
+ \left(\pol X^h\,.\,\pol\ek_1 \left[ \alpha^{-1}\,Y^h -
\kappa_{\mathcal{S}}^h \right], 
\zeta\,|\pol X^h_\rho|\right)^{(h)}  = 0
\qquad \forall\ \zeta \in V^h\,, \label{eq:sdintb}\\
& \left(\pol X^h\,.\,\pol\ek_1\,
\kappa_{\mathcal{S}}^h\,\pol\nu^h, \pol\eta\,|\pol X^h_\rho|\right)^{(h)}
+ \left( \pol\eta \,.\,\pol\ek_1, |\pol X^h_\rho|\right)
+ \left( (\pol X^h\,.\,\pol\ek_1)\,
\pol X^h_\rho,\pol\eta_\rho\, |\pol X^h_\rho|^{-1} \right) 
= - \sum_{i=1}^2
\sum_{p \in \partial_i I} \sliprho^{(p)}\,
(\pol X^h(p,t)\,.\,\pol\ek_1)\,\pol\eta(p)\,.\,\pol\ek_{3-i}
\qquad \forall\ \pol\eta \in \Vhpartial\,.
\label{eq:sdintc}
\end{align}
\end{subequations}
Choosing $\chi = 1$ in (\ref{eq:sdinta}), on recalling (\ref{eq:dVhdt}),
yields exact volume conservation for the scheme $(\BGNintstab_h)$.
Moreover, it is possible to prove a stability bound for 
$(\BGNintstab_h)^{(h)}$. 
To this end, choose $\chi = \tfrac{\alpha}\xi\,\kappa_{\mathcal{S}}^h$ in
(\ref{eq:sdinta}),
$\zeta=\alpha\,\kappa_{\mathcal{S}}^h-Y^h$ in (\ref{eq:sdintb}) 
and $\pol\eta = \pol X^h_t$ in (\ref{eq:sdintc}) to obtain,
on recalling (\ref{eq:dEhdt}), that
\begin{equation*} 
\frac1{2\,\pi}\,\ddt\, E(\pol X^h(t)) 
= - \frac1\alpha \left(
\pol X^h\,.\,\pol\ek_1\,|Y^h_\rho|^2, |\pol X^h_\rho|^{-1} \right) 
 - \xi \left( \pol X^h\,.\,\pol\ek_1\,
|\kappa^h_{\mathcal{S}} - \tfrac1\alpha\,Y^h|^2 , |\pol X^h_\rho| 
\right)^{(h)} \leq 0\,,
\end{equation*}
which is a discrete analogue of (\ref{eq:intstab}). 

\subsection{Willmore flow}
Our semidiscrete finite element approximation of $(\BGNwf)$,
(\ref{eq:bgnweak}), is given as follows, where we recall that 
$\partial I = \partial_0 I$, and so $\pol X^h(t) \in \Vhpartial$ for all
$t\in[0,T]$.

$(\BGNwf_h)^h$:
Let $\pol X^h(0) \in \Vhpartial$. For $t \in (0,T]$
find $\pol X^h(t) \in \Vh$, with $\pol X^h_t(t) \in \Vhpartial$, and
$\kappa^h(t) \in V^h$ such that
\begin{subequations} \label{eq:bgnsd}
\begin{align}
&
\left( (\pol X^h\,.\,\pol\ek_1)\,\pol X^h_t, \chi\,\pol\nu^h \, |\pol X^h_\rho|
\right)^h
- \left(\pol X^h\,.\,\pol\ek_1 \left[
\kappa^h - \doctorkappa^h(\kappa^h) \right]_\rho, 
\chi_\rho\, |\pol X^h_\rho|^{-1} \right)
= - 2 \left( \left[\kappa^h - 
\frac{\pol\omega^h\,.\,\pol\ek_1}{\pol X^h\,.\,\pol\ek_1} - \spont\right]
\kappa^h\,\pol\omega^h\,.\,\pol\ek_1, \chi \, |\pol X^h_\rho| \right)^h
\nonumber \\ & \hspace{4cm}
- \tfrac12 \left( \pol X^h\,.\,\pol\ek_1
\left(
\left[\kappa^h - \frac{\pol\omega^h\,.\,\pol\ek_1}{\pol X^h\,.\,\pol\ek_1}
\right]^2 - \spont^2\right)
\left[\kappa^h - \frac{\pol\omega^h\,.\,\pol\ek_1}{\pol X^h\,.\,\pol\ek_1}
\right]
, \chi \, |\pol X^h_\rho| \right)^h
\qquad \forall\ \chi \in V^h\,, \label{eq:bgnsda}\\
&
\left( \kappa^h\,\pol\nu^h, \pol\eta \, |\pol X^h_\rho| \right)^h
+ \left( \pol X^h_\rho, \pol\eta_\rho \, |\pol X^h_\rho|^{-1}\right) = 0 
\qquad \forall\ \pol\eta \in \Vhpartial\,.
\label{eq:bgnsdb}
\end{align}
\end{subequations}
We recall from Remark~\ref{rem:equid} that (\ref{eq:bgnsdb}) leads to the 
equidistribution property (\ref{eq:equid}). For this reason we only consider
the variant $(\BGNwf_h)^h$ with mass lumping.

\subsubsection{Helfrich flow}
On re-writing (\ref{eq:bgnsda}) as
\[
\left( (\pol X^h\,.\,\pol\ek_1)\,\pol X^h_t, \chi\,\pol\nu^h\,|\pol X^h_\rho|
\right)^h
- \left( \pol X^h\,.\,\pol\ek_1  \left[
\kappa^h - \doctorkappa^h(\kappa^h) \right]_\rho, 
\chi_\rho \, |\pol X^h_\rho|^{-1} \right)
 = \left( f^h, \chi \, |\pol X^h_\rho| \right)^h ,
\]
we consider the following semidiscrete finite element approximation of
$(\BGNwf^{A,V})$, (\ref{eq:wfweaka}), (\ref{eq:bgnweakb}).

$(\BGNwf_h^{A,V})^h$:
Let $\pol X^h(0) \in \Vhpartial$. For $t \in (0,T]$
find $\pol X^h(t) \in \Vh$, with $\pol X^h_t(t) \in \Vhpartial$, and
$(\kappa^h(t), \lambda_A^h(t), \lambda_V^h(t)) 
\in V^h \times \bR^2$ such that
\begin{align}
&\left( (\pol X^h\,.\,\pol\ek_1)\,\pol X^h_t, \chi\,\pol\nu^h \,|\pol X^h_\rho|
\right)^h
- \left( \pol X^h\,.\,\pol\ek_1 \left[
\kappa^h - \doctorkappa^{h}(\kappa^h)\right]_\rho, 
\chi_\rho \, |\pol X^h_\rho|^{-1}\right)
\nonumber \\ & \qquad
 = \left( f^h, \chi \, |\pol X^h_\rho| \right)
+ \lambda_A^h
\left( \pol X^h\,.\,\pol\ek_1\left[
\kappa^h - \doctorkappa^{h}(\kappa^h) \right],
\chi \, |\pol X^h_\rho| \right)^h
+ \lambda_V^h\left(\pol X^h\,.\,\pol\ek_1,\chi \, |\pol X^h_\rho| 
\right)^h
\qquad \forall\ \chi \in V^h\,,
\label{eq:wfsda}
\end{align}
where $(\lambda_A^h,\lambda_V^h)^T \in \bR^2$ are such that
\begin{equation} \label{eq:sideSAVh}
\mathcal{H}^2(\mathcal{S}^h(t)) = \mathcal{H}^2(\mathcal{S}^h(0))\,, \qquad
\mathcal{L}^3(\Omega^h(t)) = \mathcal{L}^3(\Omega^h(0))\,.
\end{equation}
Here we note that (\ref{eq:sideSAVh}) can be equivalently formulated as
\begin{subequations} \label{eq:lambdamuh}
\begin{align}
A(\pol X^h(t)) &= A(\pol X^h(0))\,,
\\
V(\pol X^h(t)) & = V(\pol X^h(0))\,,\quad
V(\pol Z^h) = -\pi \left( (\pol Z^h\,.\,\pol\ek_1)^2, 
[\pol Z^h_\rho]^\perp\,.\,\pol\ek_1\right) \quad \pol Z^h \in \Vhpartial
\,, 
\end{align}
\end{subequations}
where we have recalled (\ref{eq:Ah}), (\ref{eq:tauh}) and (\ref{eq:Vh}).

\setcounter{equation}{0}
\section{Fully discrete schemes} \label{sec:fd}

Let $0= t_0 < t_1 < \ldots < t_{M-1} < t_M = T$ be a
partitioning of $[0,T]$ into possibly variable time steps 
$\ttau_m = t_{m+1} - t_{m}$, $m=0\to M-1$. 
We set $\ttau = \max_{m=0\to M-1}\ttau_m$.
For a given $\pol{X}^m\in \Vhpartialzero$ we set
$\pol\nu^m = - \frac{[\pol X^m_\rho]^\perp}{|\pol X^m_\rho|}$.
Let $\pol\omega^m \in \Vh$ be the natural fully discrete analogue of
$\pol\omega^h \in \Vh$, recall (\ref{eq:omegah}). 

Similarly to (\ref{eq:calKh}), and given a $\kappa^{m+1} \in V^h$,
we introduce $\doctorkappa^{m}(\kappa^{m+1}) \in V^h$ such that
\begin{equation*} 
[\doctorkappa^{m}(\kappa^{m+1})](q_j) = \begin{cases}
\dfrac{\pol\omega^m(q_j)\,.\,\pol\ek_1}{\pol X^m(q_j)\,.\,\pol\ek_1}
& q_j \in \overline I \setminus \partial_0 I\,, \\
- \kappa^{m+1}(q_j) & q_j \in \partial_0 I\,.
\end{cases}
\end{equation*}

\subsection{Surface diffusion}

Our fully discrete analogue of the scheme
$(\BGNsd_h)^{(h)}$, (\ref{eq:sdsd}), is given as follows.

$(\BGNsd_m)^{(h)}$:
Let $\pol X^0 \in \Vhpartialzero$. For $m=0,\ldots,M-1$, 
find $(\delta\pol X^{m+1}, \kappa^{m+1}) \in \Vhpartial \times V^h$,
where $\pol X^{m+1} = \pol X^m + \delta \pol X^{m+1}$, 
such that
\begin{subequations} \label{eq:sdfd}
\begin{align}
&
\left(\pol X^m\,.\,\pol\ek_1\,
\frac{\pol X^{m+1} - \pol X^m}{\ttau_m}, \chi\,\pol\nu^m\,|\pol
X^m_\rho|\right)^{(h)}
= \left( \pol X^m\,.\,\pol\ek_1 \left[
\kappa^{m+1} - \doctorkappa^{m}(\kappa^{m+1})\right]_\rho ,
\chi_\rho\,|\pol X^m_\rho|^{-1}\right)
\qquad \forall\ \chi \in V^h\,, 
\\
&
\left(\kappa^{m+1}\,\pol\nu^m, \pol\eta\,|\pol X^m_\rho|\right)^{(h)}
+ \left(\pol X^{m+1}_\rho, \pol\eta_\rho\,|\pol X^m_\rho|^{-1}\right) 
= - \sum_{i=1}^2 
\sum_{p \in \partial_i I} \sliprho^{(p)}\,\pol\eta(p)\,.\,\pol\ek_{3-i}
\qquad \forall\ \pol\eta \in \Vhpartial\,.
\end{align}
\end{subequations}
We note that it does not appear possible to prove the existence of a
unique solution to $(\BGNsd_m)^{(h)}$. 
However, despite the lack of a mathematical proof,
in practice the linear system (\ref{eq:sdfd}) is always invertible.

Our fully discrete analogues of the scheme
$(\BGNsdstab_h)^{(h)}$, (\ref{eq:sdsds}), are given as follows.

$(\BGNsdstab_m)^{(h)}$:
Let $\pol X^0 \in \Vhpartialzero$. For $m=0,\ldots,M-1$, 
find $(\delta\pol X^{m+1}, \kappa_{\mathcal{S}}^{m+1}) 
\in \Vhpartial \times V^h$, 
where $\pol X^{m+1} = \pol X^m + \delta \pol X^{m+1}$, such that
\begin{subequations} \label{eq:fdsds}
\begin{align}
&
\left(\pol X^m\,.\,\pol\ek_1\,\frac{\pol X^{m+1} - \pol X^m}{\ttau_m},
\chi\,\pol\nu^m\,|\pol X^m_\rho|\right)^{(h)}
= \left(\pol X^m\,.\,\pol\ek_1 \left[
\kappa_{\mathcal{S}}^{m+1} \right]_\rho, 
\chi_\rho\,|\pol X^m_\rho|^{-1}\right)
\qquad \forall\ \chi \in V^h\,, 
\\
&
\left(\pol X^m\,.\,\pol\ek_1\,
\kappa_{\mathcal{S}}^{m+1}\,\pol\nu^m, \pol\eta\,|\pol X^m_\rho|\right)^{(h)}
+ \left( \pol\eta \,.\,\pol\ek_1, |\pol X^m_\rho|\right)
+ \left( (\pol X^m\,.\,\pol\ek_1)\,
\pol X^{m+1}_\rho,\pol\eta_\rho\, |\pol X^m_\rho|^{-1} \right) 
\nonumber \\ & \hspace{4cm}
= - \sum_{i=1}^2 \sum_{p \in \partial_i I} \sliprho^{(p)}\,
(\pol X^m(p)\,.\,\pol\ek_1)\,\pol\eta(p)\,.\,\pol\ek_{3-i}
\qquad \forall\ \pol\eta \in \Vhpartial\,.
\end{align}
\end{subequations}

For the second variant, which is going to lead to systems of nonlinear
equations and for which a stability result can be shown, we introduce
the notation 
$[ r ]_\pm = \pm \max \{ \pm r, 0 \}$ for $r \in \bR$.

$(\BGNsdstab_{m,\star})^{(h)}$:
Let $\pol X^0 \in \Vhpartialzero$. For $m=0,\ldots,M-1$, 
find $(\delta\pol X^{m+1}, \kappa_{\mathcal{S}}^{m+1}) 
\in \Vhpartial \times V^h$, 
where $\pol X^{m+1} = \pol X^m + \delta \pol X^{m+1}$, such that
\begin{subequations} 
\begin{align}
&
\left(\pol X^m\,.\,\pol\ek_1\,\frac{\pol X^{m+1} - \pol X^m}{\ttau_m},
\chi\,\pol\nu^m\,|\pol X^m_\rho|\right)^{(h)}
= \left(\pol X^m\,.\,\pol\ek_1 \left[
\kappa_{\mathcal{S}}^{m+1} \right]_\rho, 
\chi_\rho\,|\pol X^m_\rho|^{-1}\right)
\qquad \forall\ \chi \in V^h\,, \label{eq:fdsdsnonlineara}\\
&
\left(\pol X^m\,.\,\pol\ek_1\,
\kappa_{\mathcal{S}}^{m+1}\,\pol\nu^m, \pol\eta\,|\pol X^m_\rho|\right)^{(h)}
+ \left( \pol\eta \,.\,\pol\ek_1, |\pol X^{m+1}_\rho|\right)
+ \left( (\pol X^m\,.\,\pol\ek_1)\,
\pol X^{m+1}_\rho,\pol\eta_\rho\, |\pol X^m_\rho|^{-1} \right) 
\nonumber \\ & \qquad
= - \sum_{p \in \partial_1 I} \sliprho^{(p)}\,
(\pol X^m(p)\,.\,\pol\ek_1)\,\pol\eta(p)\,.\,\pol\ek_2
 - \sum_{p \in \partial_2 I} ( ( [\sliprho^{(p)}]_+ \,
\,\pol X^{m+1}(p)
+ [\sliprho^{(p)}]_- \,
\,\pol X^{m}(p))
\,.\,\pol\ek_1)\,\pol\eta(p)\,.\,\pol\ek_1
\qquad \forall\ \pol\eta \in \Vhpartial\,.
\label{eq:fdsdsnonlinearb}
\end{align}
\end{subequations}

We state the following mild assumptions. 
\begin{tabbing}
$(\mathfrak A)$ \qquad\quad \= Let
$|\pol{X}^m_\rho| > 0$ for almost all $\rho\in I$, and let
$\pol{X}^m \,.\,\pol\ek_1 > 0$ for all $\rho\in \overline I \setminus
\partial_0 I$.\\
$(\mathfrak B)^{(h)}$ \quad \>
Let $\mathcal Z^{(h)} = 
\left\{ \left( (\pol X^m\,.\,\pol\ek_1)\,\pol\nu^m,\chi\, 
|\pol X^m_\rho| \right)^{(h)} : \chi \in V^h \right
\} \subset \bR^2$ and assume that \\ \>
$\dim \spa \mathcal Z^{(h)} = 2$. 
\end{tabbing}
Note that the assumption $(\mathfrak B)^{h}$, on recalling
(\ref{eq:omegah}), is equivalent to assuming that 
\linebreak
$\dim \spa\{\pol{\omega}^m(q_j)\}_{j = 1\,\ldots,J}= 2$.

\begin{lem} \label{lem:exsd}
Let $\pol X^m \in\Vhpartialzero$ satisfy the assumptions $(\mathfrak A)$ and
$(\mathfrak B)^{(h)}$.
Then there exists a unique solution
$(\delta\pol X^{m+1},$ $\kappa_{\mathcal{S}}^{m+1}) \in \Vhpartial \times V^h$ 
to $(\BGNsdstab_m)^{(h)}$.
\end{lem}
\begin{proof}
As (\ref{eq:fdsds}) is linear, existence follows from uniqueness. 
To investigate the latter, we consider the system: 
Find $(\delta\pol X, \kappa_{\mathcal{S}}) \in \Vhpartial\times V^h$ 
such that
\begin{subequations} 
\begin{align}
&
\left(\pol X^m\,.\,\pol\ek_1\,\frac{\delta\pol X}{\ttau_m}, 
\chi\,\pol\nu^m\,|\pol X^m_\rho|\right)^{(h)}
= \left(\pol X^m\,.\,\pol\ek_1\,[\kappa_{\mathcal{S}}]_\rho,
\chi_\rho\,|\pol X^m_\rho|^{-1}\right)
\qquad \forall\ \chi \in V^h\,,
\label{eq:proofsda} \\
& \left(\pol X^m\,.\,\pol\ek_1\,\kappa_{\mathcal{S}}\,\pol\nu^m,
\pol\eta\,|\pol X^m_\rho|\right)^{(h)}
+ \left((\pol X^m\,.\,\pol\ek_1)\,
(\delta\pol X)_\rho, \pol\eta_\rho\,|\pol X^m_\rho|^{-1}\right) 
= 0 \qquad \forall\ \pol\eta \in \Vhpartial\,.
\label{eq:proofsdb}
\end{align}
\end{subequations}
Choosing $\chi = \kappa_{\mathcal{S}} \in V^h$ in (\ref{eq:proofsda}) and 
$\pol\eta= \delta\pol X \in \Vhpartial$ in (\ref{eq:proofsdb}) yields that
\begin{equation} \label{eq:uniquesd0}
\ttau_m
\left( \pol X^m\,.\,\pol\ek_1 \,
|(\delta\pol X)_\rho|^2, |\pol X^m_\rho|^{-1}\right)
+ \left(\pol X^m \,.\,\pol\ek_1 \,
|[\kappa_{\mathcal{S}}]_\rho|^2, |\pol X^m_\rho|^{-1} \right) = 0\,.
\end{equation}
It follows from (\ref{eq:uniquesd0}) and the assumption $(\mathfrak A)$
that $\kappa_{\mathcal{S}} = \kappa^c \in \bR$ and 
$\delta\pol X \equiv \pol X^c\in\bR^2$.
Hence it follows from 
(\ref{eq:proofsda}) that
$\pol X^c\,.\,\pol z = 0$ for all $\pol z \in \mathcal{Z}^{(h)}$,
and so assumption $(\mathfrak B)^{(h)}$ yields that $\pol X^c = \pol 0$.
Similarly, it follows from (\ref{eq:proofsdb}) and the fact that 
$\mathcal{Z}^{(h)}$ must contain a nonzero vector that $\kappa^c=0$.
Hence we have shown that (\ref{eq:fdsds}) has a unique solution
$(\delta\pol X^{m+1},\kappa_{\mathcal{S}}^{m+1}) \in \Vhpartial \times
V^h$.
\end{proof}

For the scheme $(\BGNsdstab_{m,\star})^{(h)}$ it does not appear possible to
prove existence of a solution. However, 
despite the lack of a mathematical proof,
in practice we are always able to find
a solution with the help of a Newton method.

\begin{thm} \label{thm:stabsd}
Let $\pol X^m \in\Vhpartialzero$ satisfy the assumption $(\mathfrak A)$, and
let $(\pol X^{m+1},\kappa_{\mathcal{S}}^{m+1})$ be a solution to 
$(\BGNsdstab_{m,\star})^{(h)}$. 
Then it holds that
\begin{equation} \label{eq:stabsd}
E(\pol X^{m+1}) + 
2\,\pi\,\ttau_m
\left(\pol X^m\,.\,\pol\ek_1 
\,|[\kappa_{\mathcal{S}}^{m+1}]_\rho|^2, |\pol X^m_\rho|^{-1}\right)
 \leq E(\pol X^m)\,.
\end{equation}
\end{thm}
\begin{proof}
Choosing $\chi = \ttau_m\,\kappa_{\mathcal{S}}^{m+1}$ in 
(\ref{eq:fdsdsnonlineara}) and
$\pol\eta = \pol X^{m+1} - \pol X^m\in \Vhpartial$ in 
(\ref{eq:fdsdsnonlinearb}) 
yields, on noting that $\pol X^{m}(p)\,.\,\pol\ek_1 = \pol
X^{m+1}(p)\,.\,\pol\ek_1$ for $p \in \partial_1 I$, that
\begin{align*}
 - \ttau_m \left(\pol X^m\,.\,\pol\ek_1 \,
|[\kappa_{\mathcal{S}}^{m+1}]_\rho|^2,
|\pol X^m_\rho|^{-1} \right) & 
= \left( \pol X^{m+1} - \pol X^m  ,\pol\ek_1\, |\pol X^{m+1}_\rho|\right)
+ \left( (\pol X^m\,.\,\pol\ek_1)\,
(\pol X^{m+1} - \pol X^m)_\rho,\pol X^{m+1}_\rho\, |\pol X^m_\rho|^{-1} \right)
\nonumber \\ & \qquad 
+ \sum_{p \in \partial_1 I} \sliprho^{(p)}\,
(\pol X^{m}(p)\,.\,\pol\ek_1)\,(\pol X^{m+1}(p) - \pol X^m(p))\,.\,\pol\ek_2
\nonumber \\ & \qquad
+ \sum_{p \in \partial_2 I} ([\sliprho^{(p)}]_+\,\pol X^{m+1}(p) + 
[\sliprho^{(p)}]_-\,\pol X^{m}(p)]\,.\,\pol\ek_1)\,
(\pol X^{m+1}(p) - \pol X^m(p))\,.\,\pol\ek_1
\nonumber \\ & 
\geq \left( \pol X^{m+1} - \pol X^m  ,\pol\ek_1\, |\pol X^{m+1}_\rho|\right)
+ \left( \pol X^m\,.\,\pol\ek_1,
|\pol X^{m+1}_\rho| - |\pol X^m_\rho| \right)
\nonumber \\ & \qquad 
+ \sum_{p \in \partial_1 I} \sliprho^{(p)}\,
(\pol X^{m}(p)\,.\,\pol\ek_1)\,\pol X^{m+1}(p) \,.\,\pol\ek_2
- \sum_{p \in \partial_1 I} \sliprho^{(p)}\,
(\pol X^{m}(p)\,.\,\pol\ek_1)\,\pol X^{m}(p) \,.\,\pol\ek_2
\nonumber \\ & \qquad 
+ \tfrac12 \sum_{p \in \partial_2 I} [\sliprho^{(p)}]_+\,
(\pol X^{m+1}(p)\,.\,\pol\ek_1)^2\,
- \tfrac12 \sum_{p \in \partial_2 I} [\sliprho^{(p)}]_+\,
(\pol X^{m}(p)\,.\,\pol\ek_1)^2\,
\nonumber \\ & \qquad 
+ \tfrac12 \sum_{p \in \partial_2 I} [\sliprho^{(p)}]_-\,
(\pol X^{m+1}(p)\,.\,\pol\ek_1)^2\,
- \tfrac12 \sum_{p \in \partial_2 I} [\sliprho^{(p)}]_-\,
(\pol X^{m}(p)\,.\,\pol\ek_1)^2\,
\nonumber \\ & 
= \left( \pol X^{m+1} \,.\,\pol\ek_1, |\pol X^{m+1}_\rho|\right)
- \left( \pol X^m\,.\,\pol\ek_1, |\pol X^m_\rho| \right)
\nonumber \\ & \qquad
+ \sum_{p \in \partial_1 I} \sliprho^{(p)}\,
(\pol X^{m+1}(p)\,.\,\pol\ek_1)\,\pol X^{m+1}(p) \,.\,\pol\ek_2
- \sum_{p \in \partial_1 I} \sliprho^{(p)}\,
(\pol X^{m}(p)\,.\,\pol\ek_1)\,\pol X^{m}(p) \,.\,\pol\ek_2
\nonumber \\ & \qquad
+ \tfrac12 \sum_{p \in \partial_2 I} \sliprho^{(p)}\,
(\pol X^{m+1}(p)\,.\,\pol\ek_1)^2\,
- \tfrac12 \sum_{p \in \partial_2 I} \sliprho^{(p)}\,
(\pol X^{m}(p)\,.\,\pol\ek_1)^2\,
\nonumber \\ & 
= \frac1{2\,\pi}\,E(\pol X^{m+1}) 
- \frac1{2\,\pi}\,E(\pol X^m)\,, 
\end{align*}
where we have used 
the two inequalities
$\pol a\,.\,(\pol a - \pol b) \geq |\pol b|\,(|\pol a| - |\pol b|)$
for $\pol a$, $\pol b \in \bR^2$,
and $2\,\beta\,(\beta - \alpha) \geq \beta^2 - \alpha^2$ for
$\alpha,\beta\in\bR$. This proves the desired result (\ref{eq:stabsd}).
\end{proof}

\subsection{Intermediate evolution law}

It is straightforward to adapt the schemes
$(\BGNsd_m)^h$, $(\BGNsdstab_m)^{(h)}$ and $(\BGNsdstab_{m,\star})^{(h)}$ 
to the flow (\ref{eq:salknew}). 
For example, $(\BGNsdstab_{m,\star})^{(h)}$ can be adapted to yield the
following fully discrete approximation of $(\BGNintstab_h)^{(h)}$, 
(\ref{eq:sdint}).

$(\BGNintstab_{m,\star})^{(h)}$:
Let $\pol X^0 \in \Vhpartialzero$. For $m=0,\ldots,M-1$, 
find $(\delta\pol X^{m+1}, Y^{m+1}, \kappa_{\mathcal{S}}^{m+1}) 
\in \Vhpartial \times [V^h]^2$, 
where $\pol X^{m+1} = \pol X^m + \delta \pol X^{m+1}$, such that
\begin{subequations} 
\begin{align}
&
\left(\pol X^m\,.\,\pol\ek_1\,\frac{\pol X^{m+1} - \pol X^m}{\ttau_m},
\chi\,\pol\nu^m\,|\pol X^m_\rho|\right)^{(h)} \!\!\!
= \left(\pol X^m\,.\,\pol\ek_1 \, Y^{m+1}_\rho, 
\chi_\rho\,|\pol X^m_\rho|^{-1}\right)
\qquad \forall\ \chi \in V^h\,, \label{eq:SALKa}\\
&
\tfrac1\xi \left(\pol X^m\,.\,\pol\ek_1 \, Y^{m+1}_\rho, 
\zeta_\rho\,|\pol X^m_\rho|^{-1}\right)
+ \left(\pol X^m\,.\,\pol\ek_1 \left[ \alpha^{-1}\,Y^{m+1} -
\kappa_{\mathcal{S}}^{m+1} \right], 
\zeta\,|\pol X^m_\rho|\right)^{(h)}  = 0
\qquad \forall\ \zeta \in V^h\,, \label{eq:SALKb}\\
&
\left(\pol X^m\,.\,\pol\ek_1\,
\kappa_{\mathcal{S}}^{m+1}\,\pol\nu^m, \pol\eta\,|\pol X^m_\rho|\right)^{(h)}
+ \left( \pol\eta \,.\,\pol\ek_1, |\pol X^{m+1}_\rho|\right)
+ \left( (\pol X^m\,.\,\pol\ek_1)\,
\pol X^{m+1}_\rho,\pol\eta_\rho\, |\pol X^m_\rho|^{-1} \right) 
\nonumber \\ & \qquad
= - \sum_{p \in \partial_1 I} \sliprho^{(p)}\,
(\pol X^m(p)\,.\,\pol\ek_1)\,\pol\eta(p)\,.\,\pol\ek_2
 - \sum_{p \in \partial_2 I} ( ( [\sliprho^{(p)}]_+ \,
\,\pol X^{m+1}(p)
+ [\sliprho^{(p)}]_- \,
\,\pol X^{m}(p))
\,.\,\pol\ek_1)\,\pol\eta(p)\,.\,\pol\ek_1
\qquad \forall\ \pol\eta \in \Vhpartial\,.
\label{eq:SALKc}
\end{align}
\end{subequations}

\begin{thm} \label{thm:stabint}
Let $\pol X^m \in\Vhpartialzero$ satisfy the assumption $(\mathfrak A)$, and
let $(\pol X^{m+1},Y^{m+1},\kappa_{\mathcal{S}}^{m+1})$ be a solution to 
$(\BGNintstab_{m,\star})^{(h)}$. 
Then it holds that
\begin{equation} 
 E(\pol X^{m+1}) + 
\frac{2\,\pi\,\ttau_m}{\alpha}
\left(\pol X^m\,.\,\pol\ek_1 
\,|[Y^{m+1}]_\rho|^2, |\pol X^m_\rho|^{-1}\right)
+ 2\,\pi\,\ttau_m\,\xi \left( \pol X^m\,.\,\pol\ek_1\,
|\kappa^{m+1}_{\mathcal{S}} - \tfrac1\alpha\,Y^{m+1}|^2 , |\pol X^m_\rho| 
\right)^{(h)} \leq E(\pol X^m)\,.
\label{eq:stabint}
\end{equation}
\end{thm}
\begin{proof}
The proof is a simple adaptation of the proof of Theorem~\ref{thm:stabsd}.
In particular, choosing
$\chi = \ttau_m\,\tfrac{\alpha}\xi\,\kappa_{\mathcal{S}}^{m+1}$ in
(\ref{eq:SALKa}),
$\zeta=\ttau_m\,\alpha\,\kappa_{\mathcal{S}}^{m+1}-Y^{m+1}$ in (\ref{eq:SALKb}) 
and $\pol\eta = \pol X^{m+1} - \pol X^m \in \Vhpartial$ 
in (\ref{eq:SALKc}) yields (\ref{eq:stabint}).
\end{proof}

\subsection{Willmore flow}
Our fully discrete analogue of the scheme
$(\BGNwf_h)^h$, (\ref{eq:bgnsd}), is given as follows.

$(\BGNwf_m)^h$:
Let $\pol X^0 \in \Vhpartial$ and $\kappa^0 \in V^h$. 
For $m=0,\ldots,M-1$, 
find $(\pol X^{m+1}, \kappa^{m+1}) \in \Vhpartial \times V^h$ such that
\begin{subequations} \label{eq:bgnfd}
\begin{align}
&
\left( \pol X^m\,.\,\pol\ek_1\,
\frac{\pol X^{m+1} - \pol X^m}{\ttau_m}, \chi\,\pol\nu^m \, |\pol X^m_\rho|
\right)^h
- \left( \pol X^m\,.\,\pol\ek_1 \left[
\kappa^{m+1} - \doctorkappa^{m}(\kappa^{m+1})
\right]_\rho, \chi_\rho\, |\pol X^m_\rho|^{-1} \right)
\nonumber \\ & \quad
= - 2 \left( \left[\kappa^m - 
\frac{\pol\omega^m\,.\,\pol\ek_1}{\pol X^m\,.\,\pol\ek_1} - \spont\right]
\kappa^m\,\pol\omega^m\,.\,\pol\ek_1, \chi \, |\pol X^m_\rho| \right)^h
- \tfrac12 \left( \pol X^m\,.\,\pol\ek_1
\left(
\left[\kappa^m - \frac{\pol\omega^m\,.\,\pol\ek_1}{\pol X^m\,.\,\pol\ek_1}
\right]^2 - \spont^2\right)
\left[\kappa^m - \frac{\pol\omega^m\,.\,\pol\ek_1}{\pol X^m\,.\,\pol\ek_1}
\right]
, \chi \, |\pol X^m_\rho| \right)^h
\nonumber \\ & \hspace{13cm}
\qquad \forall\ \chi \in V^h\,, \label{eq:bgnfda}\\
&
\left( \kappa^{m+1}\,\pol\nu^m, \pol\eta \, |\pol X^m_\rho| \right)^h
+ \left( \pol X^{m+1}_\rho, \pol\eta_\rho \, |\pol X^m_\rho|^{-1}\right) = 0 
\qquad \forall\ \pol\eta \in \Vhpartial\,.
\label{eq:bgnfdb}
\end{align}
\end{subequations}
We note that, similarly to $(\BGNsd_m)^{h}$, it does not appear
possible to prove existence and
uniqueness of a solution to $(\BGNwf_m)^{h}$. 
However, 
despite the lack of a mathematical proof,
in practice the linear systems (\ref{eq:bgnfd}) are always invertible.

\subsubsection{Helfrich flow}
We re-write (\ref{eq:bgnfda}) as
\[
 \left( \pol X^m\,.\,\pol\ek_1\,
\frac{\pol X^{m+1} - \pol X^m}{\ttau_m}, \chi\,\pol\nu^m\,|\pol X^m_\rho|
\right)^h
- \left( \pol X^m\,.\,\pol\ek_1 \left[
\kappa^{m+1} - \doctorkappa^{m}(\kappa^{m+1}) \right]_\rho, 
\chi_\rho \, |\pol X^m_\rho|^{-1} \right)
 = \left( f^m, \chi \, |\pol X^m_\rho| \right)^h .
\]
Then our fully discrete analogue of the scheme
$(\BGNwf_h^{A,V})^h$, (\ref{eq:wfsda}), (\ref{eq:lambdamuh}), 
is given as follows.

$(\BGNwf_m^{A,V})^h$:
Let $\pol X^0 \in \Vhpartial$ and $\kappa^0 \in V^h$. 
For $m=0,\ldots,M-1$, 
find $(\pol X^{m+1}, \kappa^{m+1}, \lambda_A^{m+1},$ $ \lambda_V^{m+1}) 
\in \Vhpartial \times V^h \times \bR^2$ such that (\ref{eq:bgnfdb}) and
\begin{subequations} \label{eq:bgnwffd}
\begin{align}
&\left( \pol X^m\,.\,\pol\ek_1\,
\frac{\pol X^{m+1} - \pol X^m}{\ttau_m}, \chi\,\pol\nu^m \, |\pol X^m_\rho|
\right)^h
- \left( \pol X^m\,.\,\pol\ek_1 \left[
\kappa^{m+1} - \doctorkappa^{m}(\kappa^{m+1})\right]_\rho, 
\chi_\rho \, |\pol X^m_\rho|^{-1}\right)
\nonumber \\ & \qquad
 = \left( f^m, \chi \, |\pol X^m_\rho| \right)
+ \lambda_A^{m+1}
\left( \pol X^m\,.\,\pol\ek_1\left[
\kappa^{m} - \doctorkappa^{m}(\kappa^{m}) \right],
\chi \, |\pol X^m_\rho| \right)^h
+ \lambda_V^{m+1} \left(\pol X^m\,.\,\pol\ek_1,\chi \, |\pol X^m_\rho| 
\right)^h
\qquad \forall\ \chi \in V^h\,, \label{eq:bgnwffda}  \\
&A(\pol X^{m+1}) = A(\pol X^0)\,,\quad
V(\pol X^{m+1}) = V(\pol X^0)\,, 
\end{align}
\end{subequations}
hold, where we have recalled (\ref{eq:lambdamuh}).
The system (\ref{eq:bgnwffd}) can be solved with a suitable nonlinear
solution method, see below.
In the simpler case of surface area conserving Willmore flow, we need to find
\linebreak
$(\pol X^{m+1}, \kappa^{m+1}, \lambda_A^{m+1}, \lambda_V^{m+1}) 
\in \Vhpartial \times V^h \times \bR \times \{0\}$
such that (\ref{eq:bgnwffd}) hold. 
Similarly, for volume conserving Willmore flow, we need to 
find $(\pol X^{m+1}, \kappa^{m+1}, \lambda_A^{m+1}, \lambda_V^{m+1}) 
\in \Vhpartial \times V^h \times \{0\} \times \bR$
such that (\ref{eq:bgnwffd}) hold.

Adapting the strategy in \cite{ElliottS10}, 
we now describe a Newton method for
solving the nonlinear system (\ref{eq:bgnwffd}).
The linear system (\ref{eq:bgnwffda})
and (\ref{eq:bgnfdb}), with $(\lambda_A^{m+1}, \lambda_V^{m+1})$ in
(\ref{eq:bgnwffda}) replaced by $(\lambda_A, \lambda_V)$, can be written as:
Find $(\pol X^{m+1}(\lambda_A,\lambda_V), \linebreak
\kappa^{m+1}(\lambda_A,\lambda_V)) \in \Vhpartial\times V^h$ such that
\begin{equation} \label{eq:lmsys}
\mathbb{T}^m\,\begin{pmatrix}
\kappa^{m+1}(\lambda_A,\lambda_V)\\[1mm]  
\pol X^{m+1}(\lambda_A,\lambda_V)
\end{pmatrix}
= \begin{pmatrix} \underline{\mathfrak g}^m \\[1mm] \pol 0
\end{pmatrix}
+ \lambda_A\, \begin{pmatrix} \underline{\mathfrak K}^m \\[1mm] \pol 0
\end{pmatrix}
+ \lambda_V\, \begin{pmatrix} \underline{\mathfrak N}^m \\[1mm] \pol 0
\end{pmatrix}.
\end{equation}
Assuming the linear operator $\mathbb{T}^m$ is invertible, we obtain that
\begin{equation}
\begin{pmatrix}
\kappa^{m+1}(\lambda_A,\lambda_V) \\[1mm]
\pol X^{m+1}(\lambda_A,\lambda_V)
\end{pmatrix}
= (\mathbb{T}^m)^{-1}\left[\begin{pmatrix} \underline{\mathfrak g}^m \\[1mm]
 \pol 0
\end{pmatrix}
+ \lambda_A \begin{pmatrix} \underline{\mathfrak K}^m \\[1mm] \pol 0
\end{pmatrix}
+ \lambda_V\, \begin{pmatrix} \underline{\mathfrak N}^m \\[1mm] \pol 0
\end{pmatrix}\right] 
=: (\mathbb{T}^m)^{-1} \begin{pmatrix} \underline{\mathfrak g}^m \\[1mm] \pol 0
\end{pmatrix}
+ \lambda_A \begin{pmatrix} {\underline s}^m_1 \\[1mm] \pol{\underline s}^m_2
\end{pmatrix}
+ \lambda_V\, \begin{pmatrix} {\underline q}^m_1 \\[1mm] \pol{\underline q}^m_2
\end{pmatrix} .
\label{eq:lmsysinverse}
\end{equation}
It immediately follows from (\ref{eq:lmsysinverse}) that
\begin{equation*}
\partial_{\lambda_A} \pol X^{m+1}(\lambda_A,\lambda_V) 
= \pol{\underline s}^m_2\,,\quad
\partial_{\lambda_V} \pol X^{m+1}(\lambda_A,\lambda_V)
= \pol{\underline q}^m_2\,.
\end{equation*}
Hence
\[
\partial_{\lambda_A} A(\pol X^{m+1}(\lambda_A,\lambda_V))
  = 
\left[\deldel{\pol X^{m+1}}\, A(\pol X^{m+1}(\lambda_A,\lambda_V))\right]
(\pol s^m_2)\,, \qquad
\partial_{\lambda_A} V(\pol X^{m+1}(\lambda_A,\lambda_V))
  = 
\left[\deldel{\pol X^{m+1}}\, V(\pol X^{m+1}(\lambda_A,\lambda_V))\right]
(\pol s^m_2)\,,
\]
and similarly for $\partial_{\lambda_V} A(\pol X^{m+1}(\lambda_A,\lambda_V))$
and $\partial_{\lambda_V} V(\pol X^{m+1}(\lambda_A,\lambda_V))$.
Here $\pol s^m_2 \in \Vhpartial$ 
is the finite element function corresponding to the
coefficients in $\pol{\underline s}^m_2$ for the standard basis of $\Vh$.
Moreover, we have defined the first variation of $A(\pol Z^h)$,
for any $\pol Z^h \in \Vhpartial$, as
\[
\left[\deldel{\pol Z^h}\, A(\pol Z^h)\right](\pol\eta)
 = \lim_{\epsilon\to0} \frac1\epsilon\left(
A(\pol Z^h + \epsilon\,\pol\eta) - A(\pol Z^h)\right)
= 2\,\pi
\left(\pol\eta\,.\,\pol\ek_1,|\pol Z^h_\rho|\right) + 2\,\pi \left(
(\pol Z^h\,.\,\pol\ek_1)\,
\pol\eta_\rho,\pol Z^h_\rho\, |\pol Z^h_\rho|^{-1} \right)
\qquad \forall\ \pol\eta \in \Vhpartial \, ,
\]
and similarly 
\[
\left[\deldel{\pol Z^h}\, V(\pol Z^h)\right](\pol\eta)
 = \lim_{\epsilon\to0} \frac1\epsilon\left(
V(\pol Z^h + \epsilon\,\pol\eta) - V(\pol Z^h)\right) 
= 2\,\pi
\left( \pol Z^h\,.\,\pol\ek_1, \pol\eta\,.\,[\pol Z^h_\rho]^\perp\right) 
\qquad \forall\ \pol\eta \in \Vhpartial \,.
\]

For a given iterate $(\lambda_A^{k},\lambda_V^{k})$, with corresponding
$\pol X^{m+1,k} = \pol X^{m+1}(\lambda_A^{k},\lambda_V^{k})$ and
$\kappa^{m+1,k} = \kappa^{m+1}(\lambda_A^{k},\lambda_V^{k})$, we now
define the following quantities.
\[
[\pol {\underline k}^{m+1,k}]_i=
\left(
\left[\deldel{\pol X^{m+1,k}}\, A(\pol X^{m+1,k})
\right](\chi_i\,\pol\ek_\ell) \right)_{\ell=1}^2 , \qquad
[\pol {\underline n}^{m+1,k}]_i =
\left(
\left[\deldel{\pol X^{m+1,k}}\, V(\pol X^{m+1,k})
\right](\chi_i\,\pol\ek_\ell) \right)_{\ell=1}^2 . 
\]
Then the Newton update is given by
\begin{equation} \label{eq:Newton}
\begin{pmatrix} \lambda_A^{k+1}\\[1mm] \lambda_V^{k+1} \end{pmatrix}
= 
\begin{pmatrix} \lambda_A^{k}\\[1mm] \lambda_V^{k} \end{pmatrix}
- \begin{pmatrix}
\pol {\underline k}^{m+1,k} \,.\,\pol {\underline s}^m_2
& \pol {\underline k}^{m+1,k} \,.\,\pol {\underline q}^m_2 \\[1mm]
\pol {\underline n}^{m+1,k} \,.\,\pol {\underline s}^m_2
& \pol {\underline n}^{m+1,k} \,.\,\pol {\underline q}^m_2
\end{pmatrix}^{-1}\,
\begin{pmatrix}
A(\pol X^{m+1,k}) - A(\pol X^0) \\[1mm]
V(\pol X^{m+1,k}) - V(\pol X^0)
\end{pmatrix}.
\end{equation}
In practice, the linear systems (\ref{eq:lmsys}) are always invertible, and the
Newton iteration (\ref{eq:Newton}) converges within a couple of iterations.

\setcounter{equation}{0}
\section{Numerical results} \label{sec:nr}

As the fully discrete energy, we consider $E(\pol X^m)$, recall 
(\ref{eq:Eh}). 
Unless otherwise stated, we choose $\sliprho^{(0)} = \sliprho^{(1)} = 0$.
We always employ uniform time steps, $\ttau_m = \ttau$,
$m=0,\ldots,M-1$.

We also consider the ratio
\begin{equation} \label{eq:ratio}
\ratio^m = \dfrac{\max_{j=1\to J} |\pol{X}^m(q_j) - \pol{X}^m(q_{j-1})|}
{\min_{j=1\to J} |\pol{X}^m(q_j) - \pol{X}^m(q_{j-1})|}
\end{equation}
between the longest and shortest element of $\Gamma^m$, and are often
interested in the evolution of this ratio over time.

In practice, we stop the computation when $\pol X^m < 0$ for some 
$\rho \in \overline I$, as the computed results would then no longer be
physical. However, for sufficiently small discretization parameters
this happens only once the computation reaches a singularity for
the underlying flow. 

\subsection{Numerical results for surface diffusion} \label{sec:sdnr}

\subsubsection{Sphere}

Clearly, a sphere is a stationary solution for surface diffusion. Hence,
setting $\partial_0 I = \partial I = \{0,1\}$ and
choosing as initial data $\pol X^0$ the approximations of a semicircle 
displayed in Figure~\ref{fig:semicircle}, we now investigate the different
tangential motions exhibited by the six schemes 
$(\BGNsd_m)^h$, $(\BGNsd_m)$, 
$(\BGNsdstab_m)^h$, $(\BGNsdstab_m)$, $(\BGNsdstab_{m,\star})^h$ and
$(\BGNsdstab_{m,\star})$.
\begin{figure}
\center
\includegraphics[angle=-90,width=0.2\textwidth]{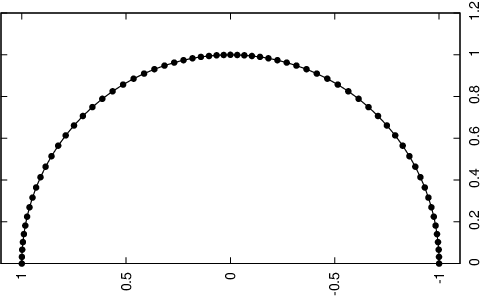} \qquad
\includegraphics[angle=-90,width=0.2\textwidth]{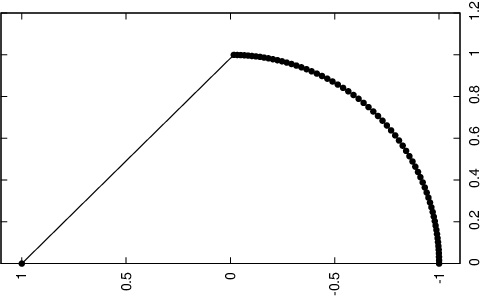}
\caption{
Initial data $\pol X^0$ approximating a semicircle with $J=64$. The initial
ratios (\ref{eq:ratio}) are $\ratio^0= 1.94$ and $\ratio^0= 89.81$, respectively.
}
\label{fig:semicircle}
\end{figure}%
We set $\ttau = 10^{-4}$ and integrate the evolution for the initial data on
the left of Figure~\ref{fig:semicircle} until time $T=1$,
see Figure~\ref{fig:TM}. Of the
six schemes, only $(\BGNsdstab_m)^h$ breaks down before reaching the
final time. When $(\BGNsdstab_m)^h$ breaks
down due to vertices moving to the left of the $x_2$--axis, the element ratio
$\ratio^m$ has reached a value of 6058. Hence it appears that 
$(\BGNsdstab_m)^h$ exhibits an implicit tangential motion towards
the $x_2$--axis, which can lead to coalescence of vertices or to vertices on
the left of the $x_2$--axis. For this reason we do not consider the
scheme $(\BGNsdstab_m)^h$ any further.
For the remaining five schemes 
$(\BGNsd_m)^h$, $(\BGNsd_m)$,
$(\BGNsdstab_m)$, $(\BGNsdstab_{m,\star})^h$,
$(\BGNsdstab_{m,\star})$ the element ratios $\ratio^m$ at time $T=1$ are
$1.00, 1.00, 3.04, 62.21, 3.05$, 
and the enclosed volume is preserved almost exactly by all 
the schemes. We show the final distributions of vertices, and plots of $\ratio^m$
over time in Figure~\ref{fig:TM}.
\begin{figure}
\center
\hspace*{-5mm}
\mbox{
\includegraphics[angle=-90,width=0.19\textwidth]{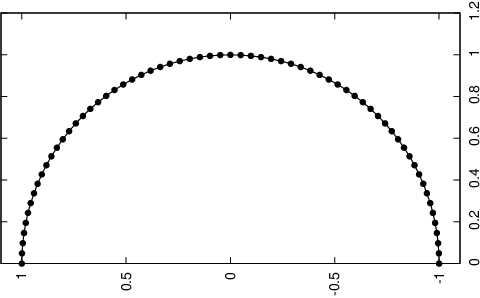}
\includegraphics[angle=-90,width=0.19\textwidth]{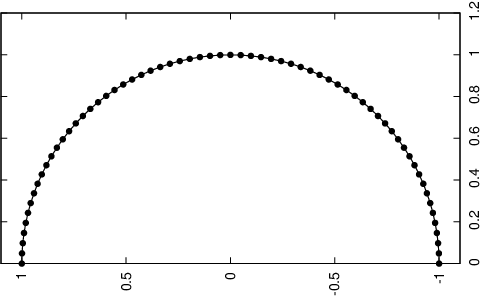}
\includegraphics[angle=-90,width=0.19\textwidth]{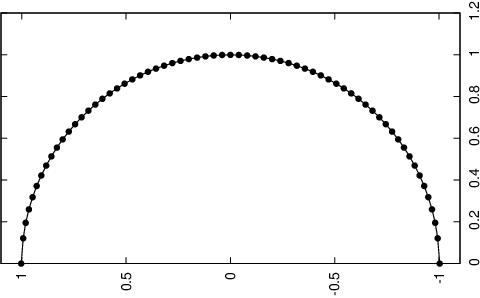}
\includegraphics[angle=-90,width=0.19\textwidth]{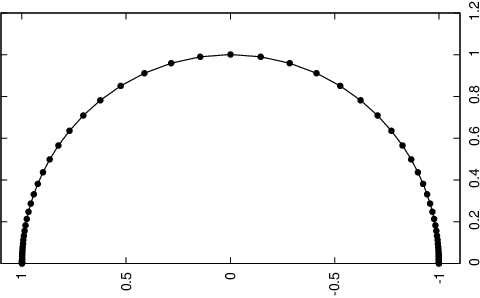}
\includegraphics[angle=-90,width=0.19\textwidth]{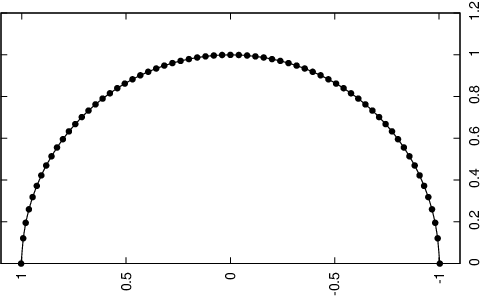}}
\hspace*{-5mm}
\mbox{
\includegraphics[angle=-90,width=0.19\textwidth]{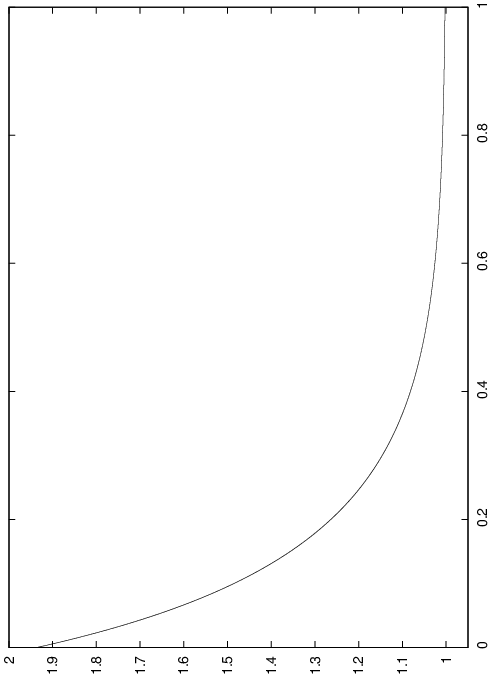}
\includegraphics[angle=-90,width=0.19\textwidth]{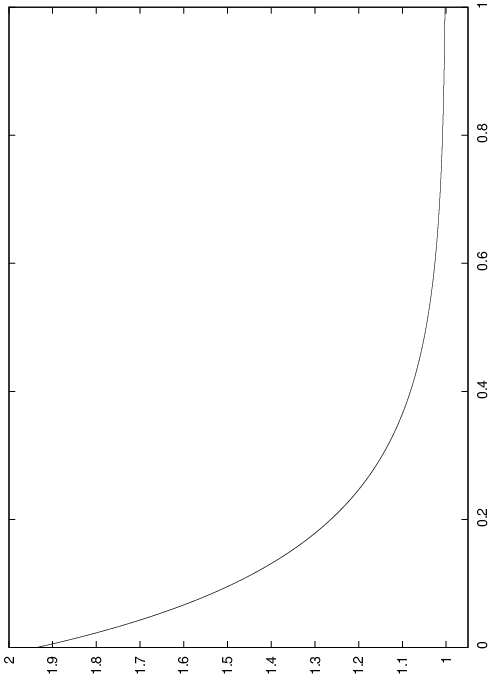}
\includegraphics[angle=-90,width=0.19\textwidth]{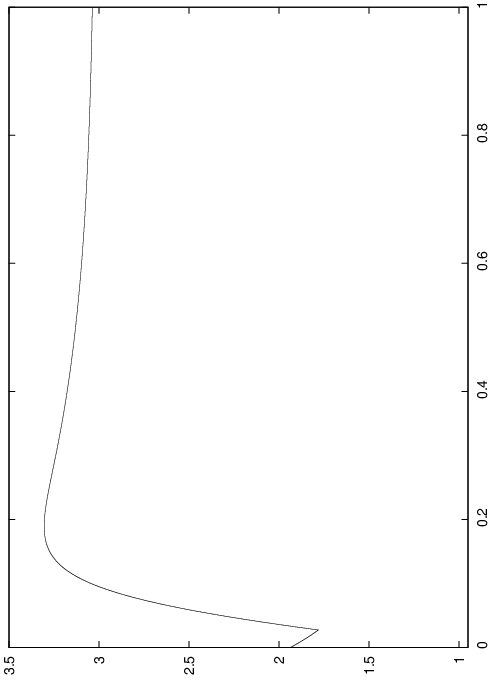}
\includegraphics[angle=-90,width=0.19\textwidth]{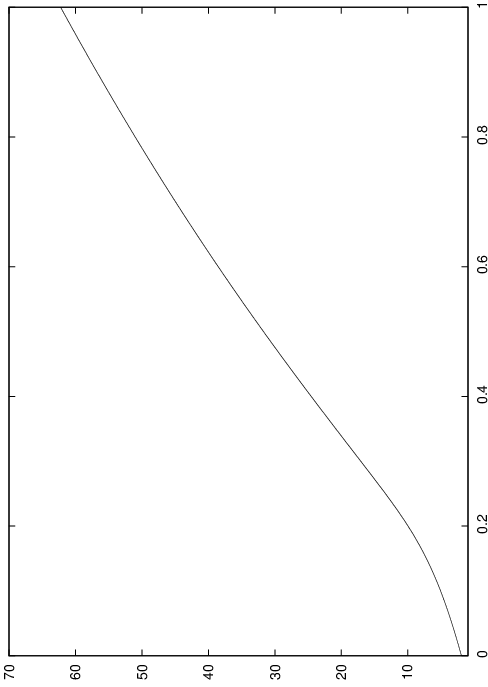}
\includegraphics[angle=-90,width=0.19\textwidth]{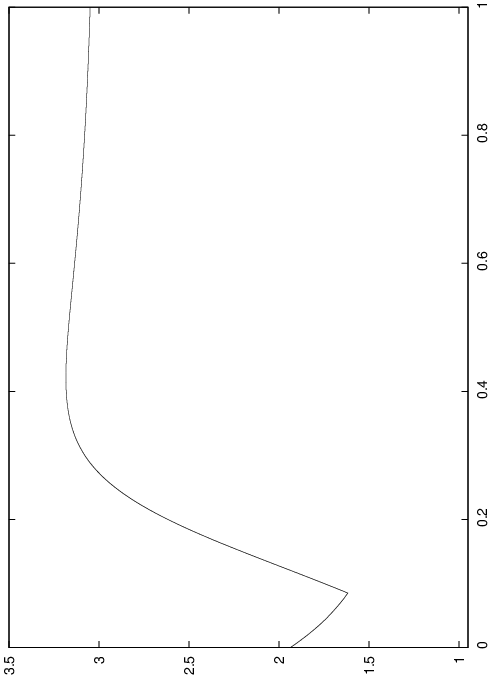}}
\caption{
Comparison of the different schemes for surface diffusion of the unit sphere.
Left to right: $(\BGNsd_m)^h$, $(\BGNsd_m)$,
$(\BGNsdstab_m)$, $(\BGNsdstab_{m,\star})^h$,
$(\BGNsdstab_{m,\star})$.
Plots are for $\pol X^m$ at time $t=1$ and for the ratio $\ratio^m$ over time.
The element ratios $\ratio^m$ at time $t=1$ are
$1.00$, $1.00$, $3.04$, $62.21$, $3.05$.
}
\label{fig:TM}
\end{figure}%
In addition, we show plots of the $\ratio^m$ for the scheme 
$(\BGNsdstab_{m,\star})$ for different time step sizes in
Figure~\ref{fig:TM3456}. In these experiments it appears that the element
ratio asymptotically approaches a value close to $3$. The same plots
for the scheme $(\BGNsd_m)^h$ show $\ratio^m$ monotonically decreasing to
the value $1$ by virtue of the equidistribution property (\ref{eq:equid}),
with the decrease faster for smaller time step sizes $\ttau$. 
\begin{figure}
\center
\includegraphics[angle=-90,width=0.24\textwidth]{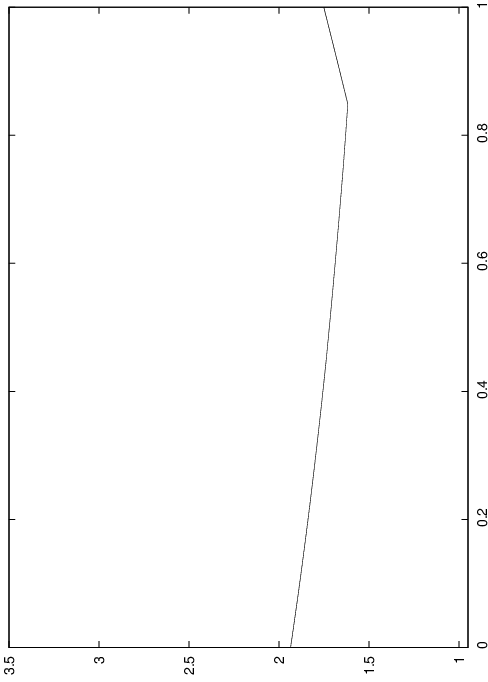}
\includegraphics[angle=-90,width=0.24\textwidth]{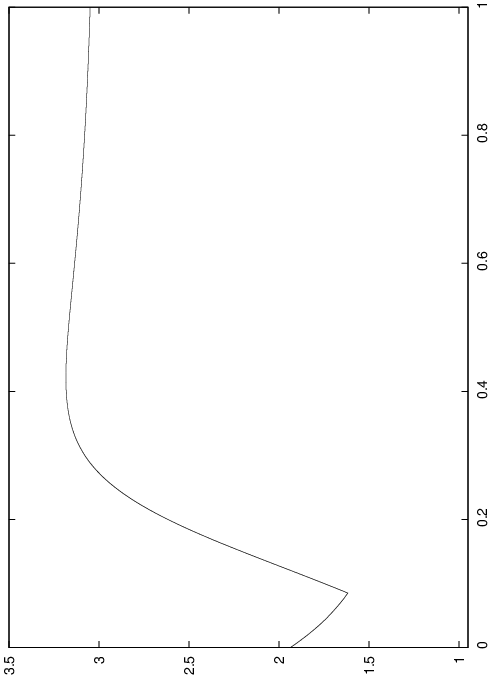}
\includegraphics[angle=-90,width=0.24\textwidth]{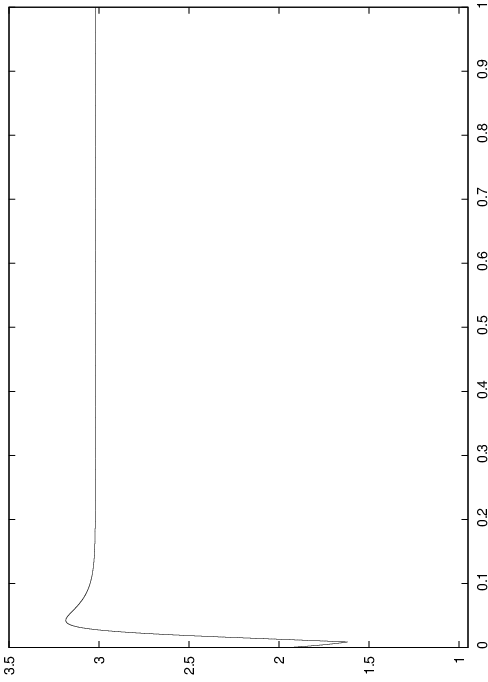}
\includegraphics[angle=-90,width=0.24\textwidth]{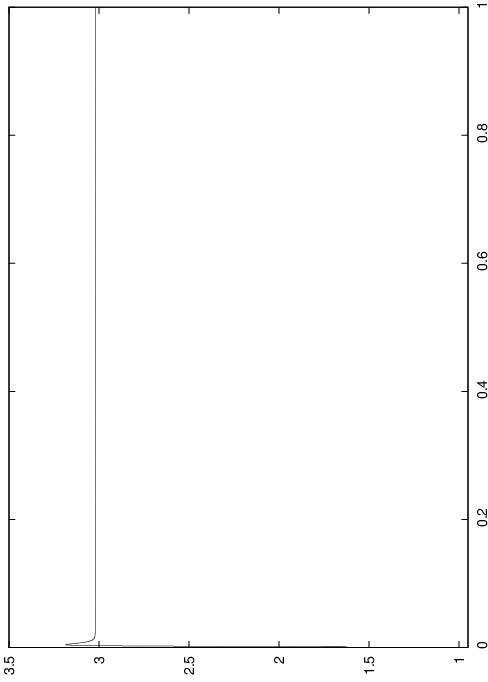}
\caption{$(\BGNsdstab_{m,\star})$
Plots of the ratio $\ratio^m$ for $\ttau = 10^{-k}$, $k=3,\ldots6$.}
\label{fig:TM3456}
\end{figure}%

In a second set of experiments to investigate the different tangential motions
induced by the individual schemes, we repeat the simulations in
Figure~\ref{fig:TM} now for the initial data displayed on the right of
Figure~\ref{fig:semicircle}.
We again use $J=64$ and $\ttau = 10^{-4}$, and show the relevant results in
Figure~\ref{fig:TMusp}.
Once again the scheme $(\BGNsdstab_m)^h$ breaks down
due to vertices moving to the left of the $x_2$--axis. 
For the remaining five schemes 
$(\BGNsd_m)^h$, $(\BGNsd_m)$,
$(\BGNsdstab_m)$, $(\BGNsdstab_{m,\star})^h$,
$(\BGNsdstab_{m,\star})$ the element ratios $\ratio^m$ at time $T=1$ are
$1.06$, $1.06$, $3.02$, $113.13$, $3.07$. 
Due to the very nonuniform initial data, the enclosed volume is only 
preserved well for the three 
schemes without numerical integration. In particular, the relative enclosed
volume losses for the five schemes 
are $20.6\%$, $-1.0\%$, $-0.9\%$, $35.9\%$, $-0.7\%$.
\begin{figure}
\center
\hspace*{-5mm}
\mbox{
\includegraphics[angle=-90,width=0.19\textwidth]{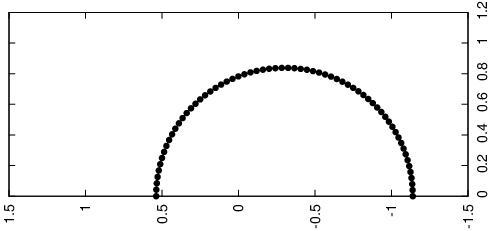}
\includegraphics[angle=-90,width=0.19\textwidth]{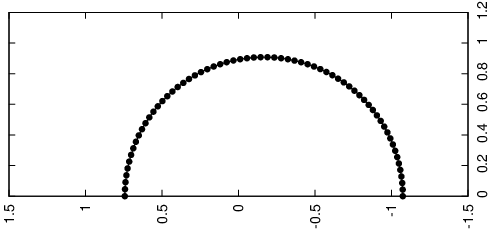}
\includegraphics[angle=-90,width=0.19\textwidth]{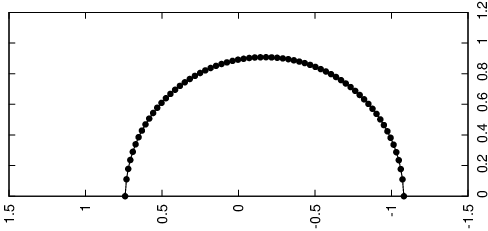}
\includegraphics[angle=-90,width=0.19\textwidth]{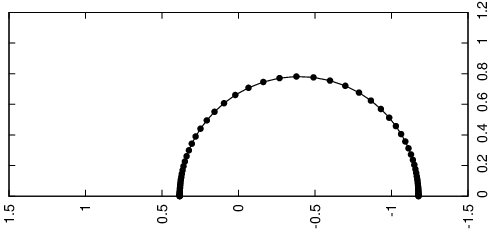}
\includegraphics[angle=-90,width=0.19\textwidth]{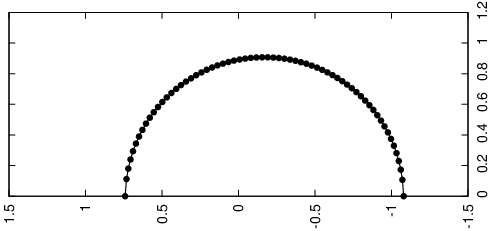}
}
\hspace*{-5mm}
\mbox{
\includegraphics[angle=-90,width=0.19\textwidth]{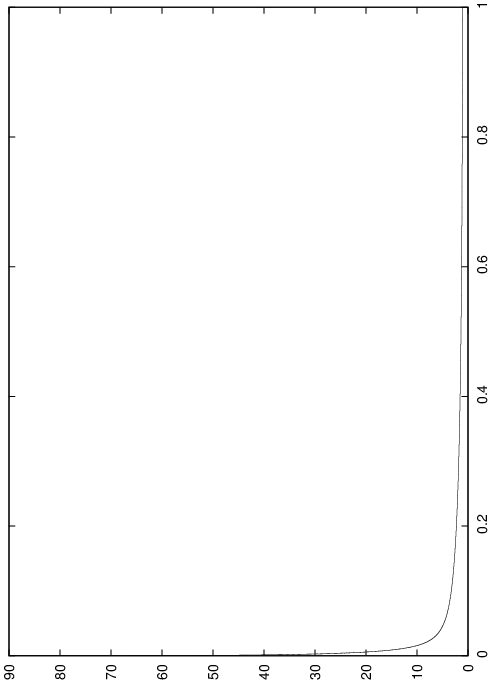}
\includegraphics[angle=-90,width=0.19\textwidth]{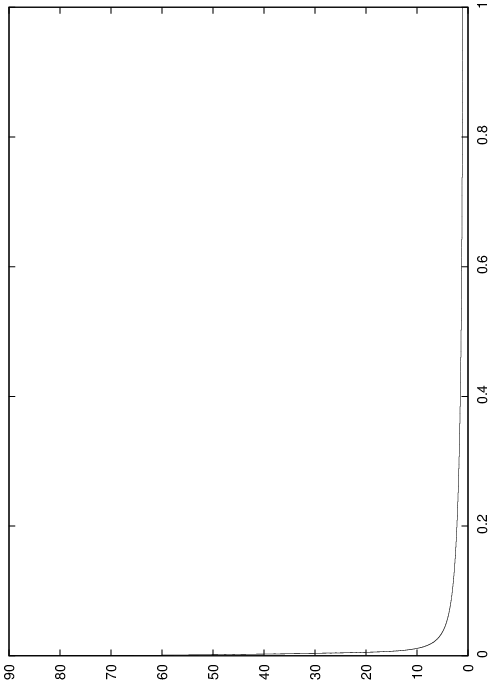}
\includegraphics[angle=-90,width=0.19\textwidth]{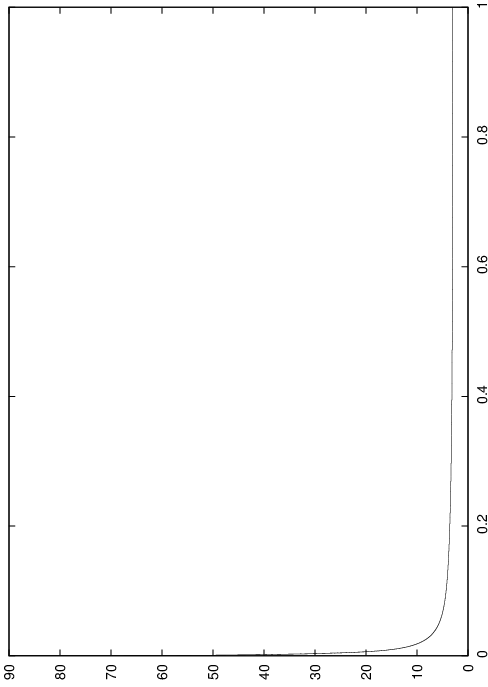}
\includegraphics[angle=-90,width=0.19\textwidth]{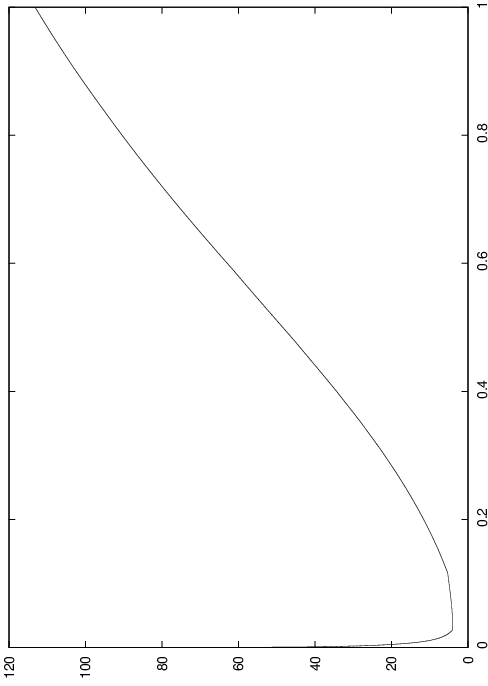}
\includegraphics[angle=-90,width=0.19\textwidth]{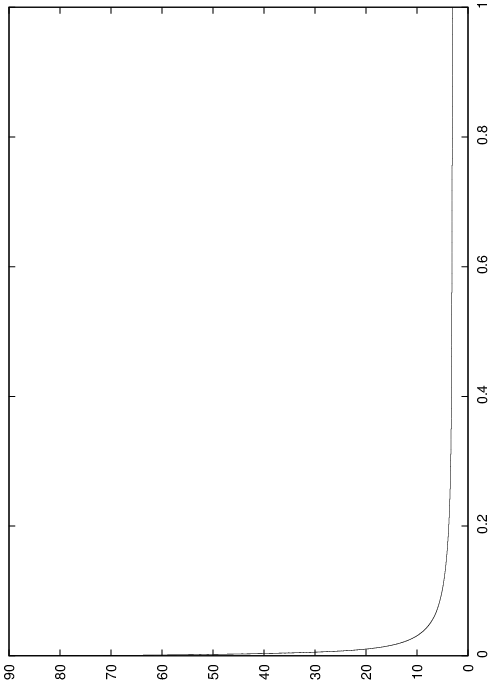}
}
\caption{
Comparison of the different schemes for surface diffusion towards a sphere.
Left to right: $(\BGNsd_m)^h$, $(\BGNsd_m)$,
$(\BGNsdstab_m)$, $(\BGNsdstab_{m,\star})^h$,
$(\BGNsdstab_{m,\star})$.
Plots are for $\pol X^m$ at time $t=1$ and for the ratio $\ratio^m$ over time.
The element ratios $\ratio^m$ at time $t=1$ are
$1.06$, $1.06$, $3.02$, $113.13$, $3.07$. The relative enclosed
volume losses are $20.6\%$, $-1.0\%$, $-0.9\%$, $35.9\%$, $-0.7\%$.
}
\label{fig:TMusp}
\end{figure}%

For the remainder of this subsection, we will only present numerical results
for the two schemes $(\BGNsd_m)^h$ and 
$(\BGNsdstab_{m,\star})$. Note that the former is a linear fully
discrete approximation of $(\BGNsd_h)^h$, for which the
equidistribution property (\ref{eq:equid}) holds. The latter, on the other
hand, is a nonlinear scheme that is unconditionally stable, recall
Theorem~\ref{thm:stabsd}, and, the semidiscrete scheme 
$(\BGNsdstab_h)$ that it is based on preserves the enclosed volume
exactly.
As the results for $(\BGNsd_m)^h$ and 
$(\BGNsdstab_{m,\star})$ are often indistinguishable, we only 
visualize the numerical results for the former, and will do so from now on 
in this section.

\subsubsection{Genus 0 surface}
An experiment for a rounded cylinder of total dimension $1\times7\times1$ 
can be seen in Figure~\ref{fig:sdcigar711}. Here we have once again that
$\partial_0 I = \partial I = \{0,1\}$.
The discretization parameters are $J=128$ and $\ttau = 10^{-4}$.
The relative volume loss for this experiment for $(\BGNsd_m)^h$ 
is $0.05\%$, while for $(\BGNsdstab_{m,\star})$ it is $0.00\%$.
\begin{figure}
\center
\begin{minipage}{0.4\textwidth}
\includegraphics[angle=-90,width=0.38\textwidth]{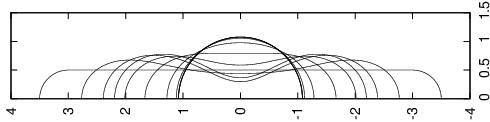} \quad
\includegraphics[angle=-90,width=0.45\textwidth]{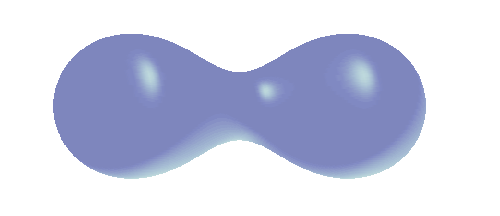}
\end{minipage} \qquad
\begin{minipage}{0.3\textwidth}
\includegraphics[angle=-90,width=0.95\textwidth]{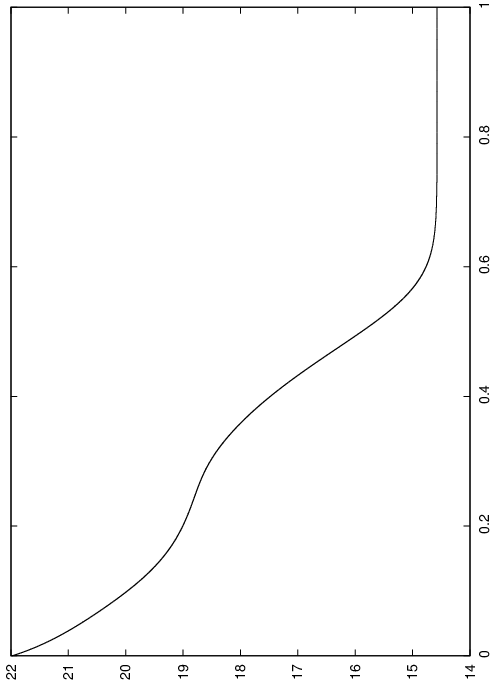} \\
\includegraphics[angle=-90,width=0.95\textwidth]{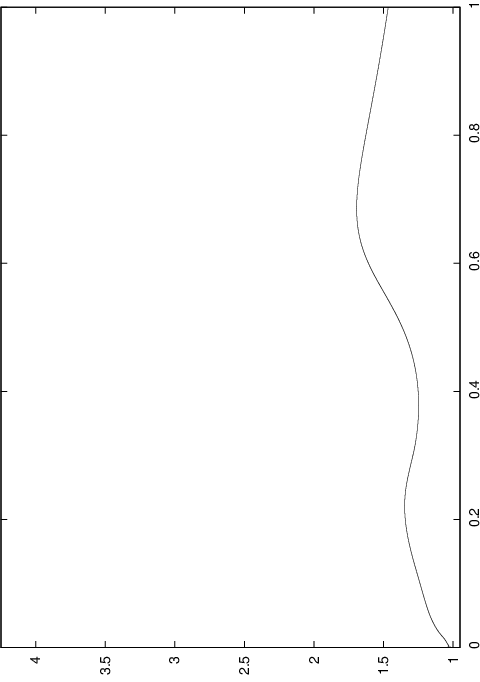} \\
\includegraphics[angle=-90,width=0.95\textwidth]{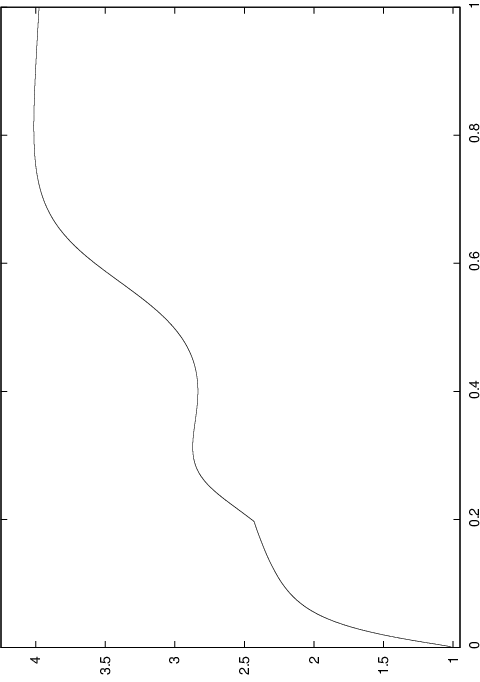}
\end{minipage}
\caption{
$(\BGNsd_m)^h$
Evolution for a rounded cylinder of dimension $1\times7\times1$. 
Plots are at times $t=0,0.1,\ldots,1$.
We also visualize the axisymmetric surface $\mathcal{S}^m$ generated by
$\Gamma^m$ at time $t=0.3$.
On the right are plots of the discrete energy and the ratio $\ratio^m$ and,
as a comparison, a plot of the ratio $\ratio^m$ for the scheme 
$(\BGNsdstab_{m,\star})$.
}
\label{fig:sdcigar711}
\end{figure}%
If we increase the aspect ratio of the initial data, then pinch-off can occur
during the evolution. We visualize this effect in Figure~\ref{fig:sdcigar811},
where as initial data we choose a rounded cylinder of total dimension
$1\times8\times1$. The discretization parameters are as before,
and the relative volume loss for $(\BGNsd_m)^h$ 
is $0.02\%$, while for $(\BGNsdstab_{m,\star})$ it is $0.00\%$.
\begin{figure}
\center
\begin{minipage}{0.4\textwidth}
\includegraphics[angle=-90,width=0.28\textwidth]{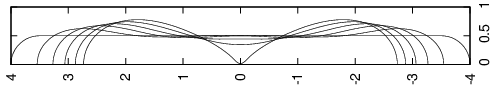} \quad
\includegraphics[angle=-90,width=0.5\textwidth]{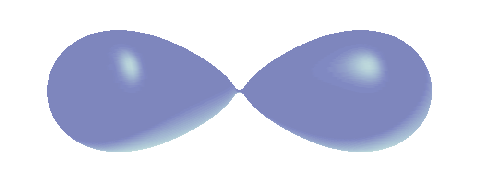}
\end{minipage} \qquad
\begin{minipage}{0.3\textwidth}
\includegraphics[angle=-90,width=0.95\textwidth]{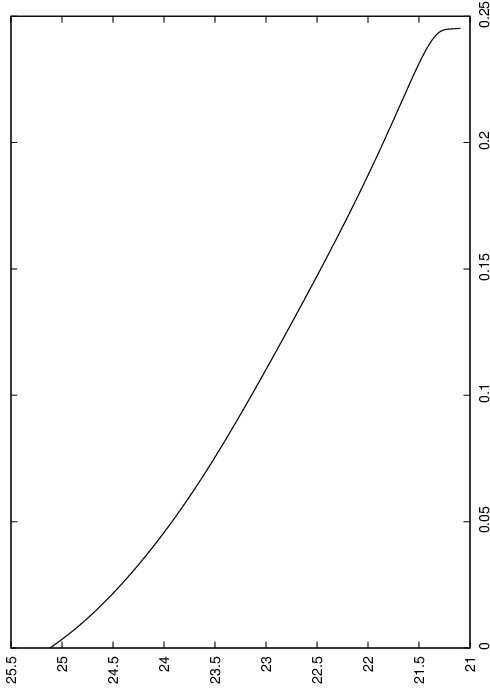} \\
\includegraphics[angle=-90,width=0.95\textwidth]{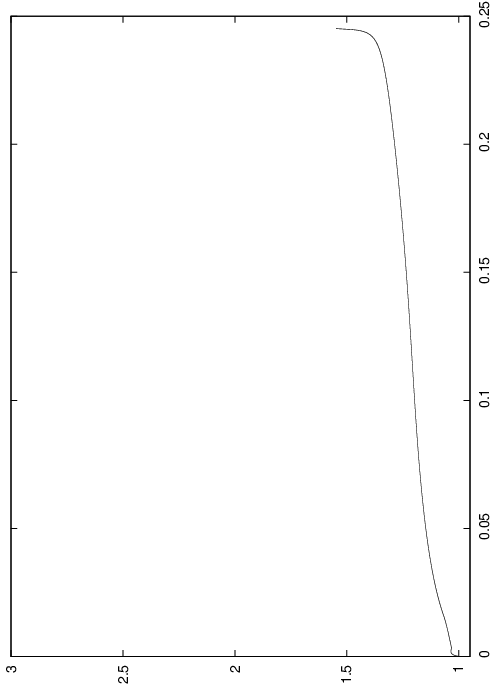} \\
\includegraphics[angle=-90,width=0.95\textwidth]{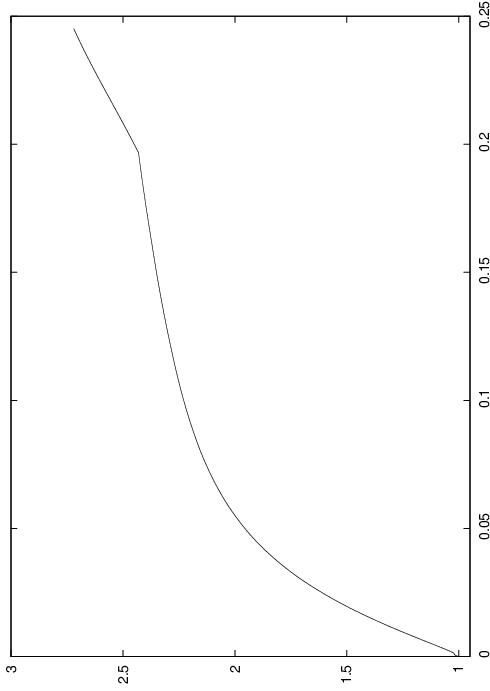}
\end{minipage}
\caption{
$(\BGNsd_m)^h$
Evolution for a rounded cylinder of dimension $1\times8\times1$. 
Plots are at times $t=0,0.05,\ldots,0.2,0.2452$.
We also visualize the axisymmetric surface $\mathcal{S}^m$ generated by
$\Gamma^m$ at time $t=0.2452$.
On the right are plots of the discrete energy and the ratio $\ratio^m$ and,
as a comparison, a plot of the ratio $\ratio^m$ for the scheme 
$(\BGNsdstab_{m,\star})$.
}
\label{fig:sdcigar811}
\end{figure}%
An experiment for a disc shape of total dimension $9 \times 1 \times 9$
is shown in Figure~\ref{fig:sdflatcigar}.
The discretization parameters are $J=128$ and $\ttau = 10^{-3}$.
The relative volume loss for this experiment for $(\BGNsd_m)^h$ 
is $0.03\%$, while for $(\BGNsdstab_{m,\star})$ it is $0.04\%$.
\begin{figure}
\center
\begin{minipage}{0.35\textwidth}
\includegraphics[angle=-90,width=0.95\textwidth]{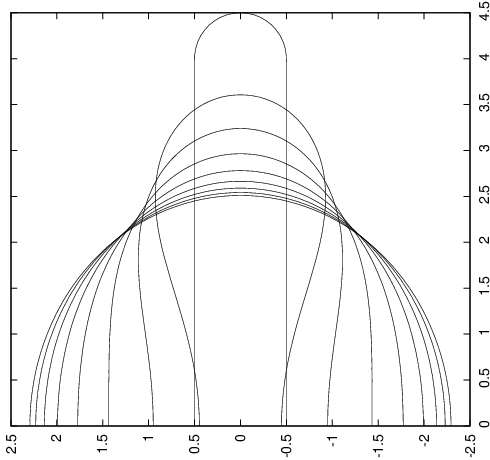} \\
\includegraphics[angle=-90,width=0.9\textwidth]{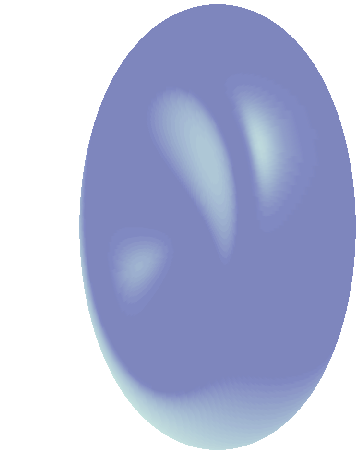}
\end{minipage} \qquad
\begin{minipage}{0.3\textwidth}
\includegraphics[angle=-90,width=0.95\textwidth]{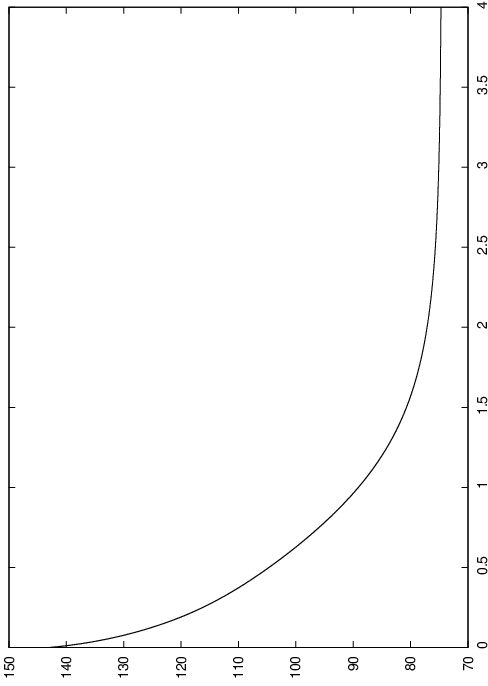} \\
\includegraphics[angle=-90,width=0.95\textwidth]{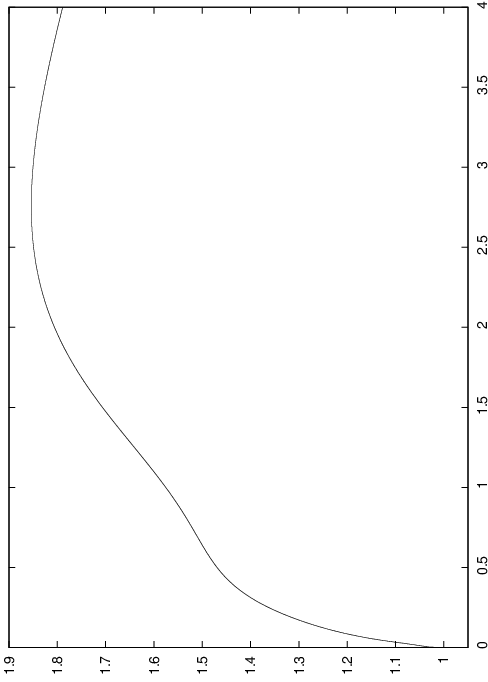} \\
\includegraphics[angle=-90,width=0.95\textwidth]{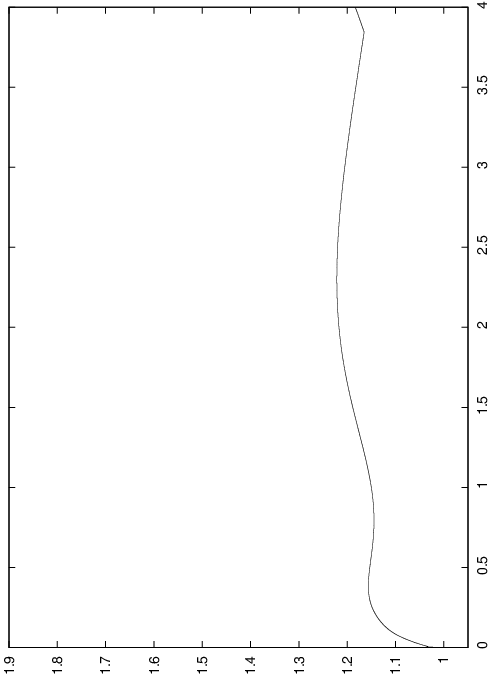}
\end{minipage}
\caption{
$(\BGNsd_m)^h$
Evolution for a disc of dimension $9\times1\times9$. 
Plots are at times $t=0,0.5,\ldots,4$. 
We also visualize the axisymmetric surface $\mathcal{S}^m$ generated by
$\Gamma^m$ at time $t=0.5$.
On the right are plots of the discrete energy and the ratio $\ratio^m$ and,
as a comparison, a plot of the ratio $\ratio^m$ for the scheme 
$(\BGNsdstab_{m,\star})$.
}
\label{fig:sdflatcigar}
\end{figure}%
We notice that although for the time step size $\ttau = 10^{-3}$, the element
ratio for the scheme $(\BGNsdstab_{m,\star})$ is smaller than for
$(\BGNsd_m)^h$, this is no longer the case for smaller time step
sizes. For smaller time step sizes, the ratio approaches the value $1$ very
quickly for the scheme $(\BGNsd_m)^h$, while for 
$(\BGNsdstab_{m,\star})$ it can reach much larger values, before
eventually approaching a value closer to $4$. 
See Figure~\ref{fig:sdflatcigarTM}
for some ratio plots for $(\BGNsdstab_{m,\star})$ when
$\ttau = 10^{-k}$, $k=4,5,6$.
\begin{figure}
\includegraphics[angle=-90,width=0.3\textwidth]{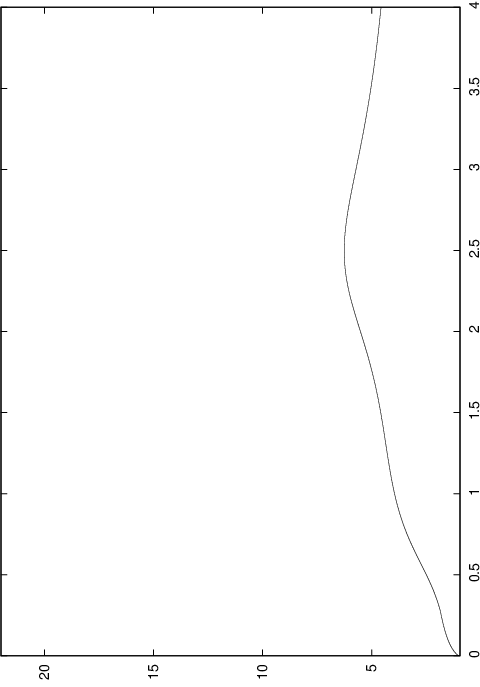}
\includegraphics[angle=-90,width=0.3\textwidth]{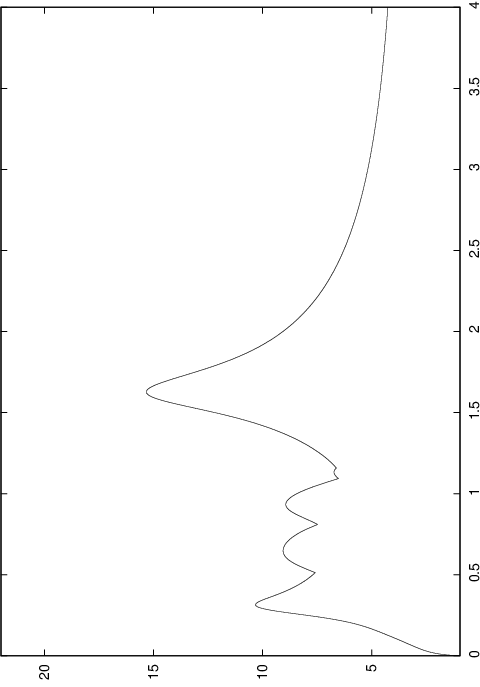}
\includegraphics[angle=-90,width=0.3\textwidth]{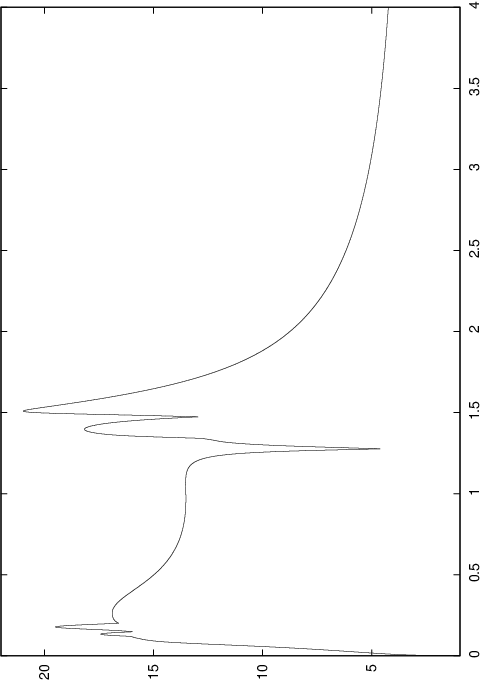}
\caption{
$(\BGNsdstab_{m,\star})$
Plot of the ratio $\ratio^m$ for $\ttau = \ttau = 10^{-k}$, $k=4,5,6$.
}
\label{fig:sdflatcigarTM}
\end{figure}%
We note that this behaviour appears to be generic for all our numerical
experiments for surface diffusion.

\subsubsection{Torus}
In order to model the evolution of a torus, we set $I = \RZ$, so that
$\partial I = \emptyset$.
For a torus with $R=1$, $r=0.25$, we obtain a surface that closes
up towards a genus-0 surface, as in \cite[Fig.\ 14]{gflows3d}.
See Figure~\ref{fig:sdtorusR1r025} for the simulation results, where we note
that the surface closing up represents a singularity for the parametric
approach. In particular, some vertices of $\pol X^m$ are approaching 
the $x_2$--axis, which leads to a moderate increase in the element ratio 
(\ref{eq:ratio}).
The discretization parameters for this experiment 
are $J=256$ and $\ttau = 10^{-5}$.
The observed relative volume loss is $0.02\%$ for both 
the schemes $(\BGNsd_m)^h$ and $(\BGNsdstab_{m,\star})$.
\begin{figure}
\center
\begin{minipage}{0.35\textwidth}
\includegraphics[angle=-90,width=0.95\textwidth]{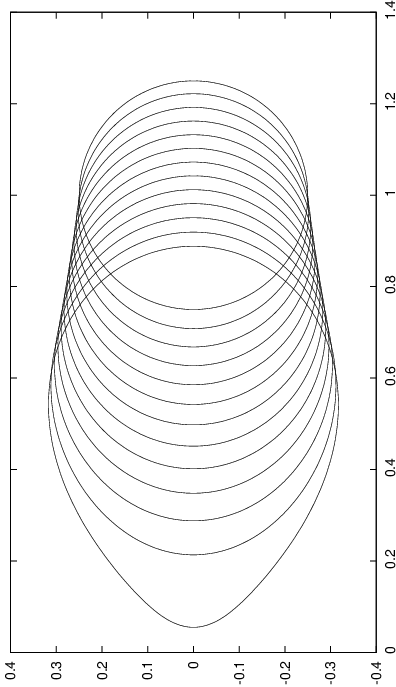} \\
\includegraphics[angle=-90,width=0.7\textwidth]{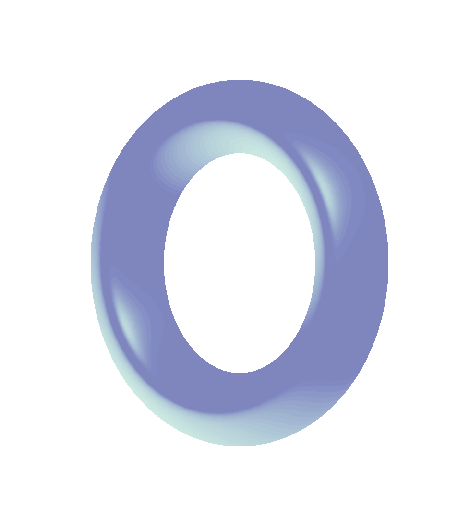}
\includegraphics[angle=-90,width=0.7\textwidth]{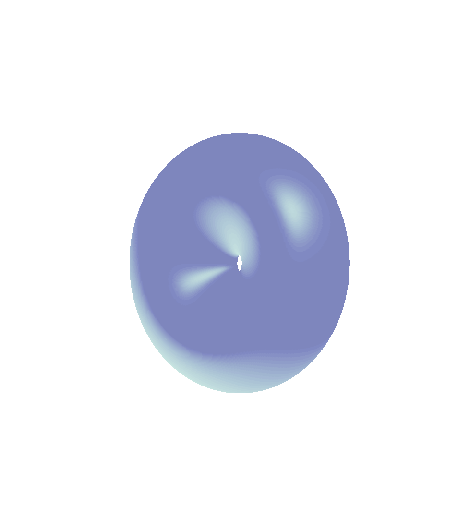}
\end{minipage} \qquad
\begin{minipage}{0.3\textwidth}
\includegraphics[angle=-90,width=0.95\textwidth]{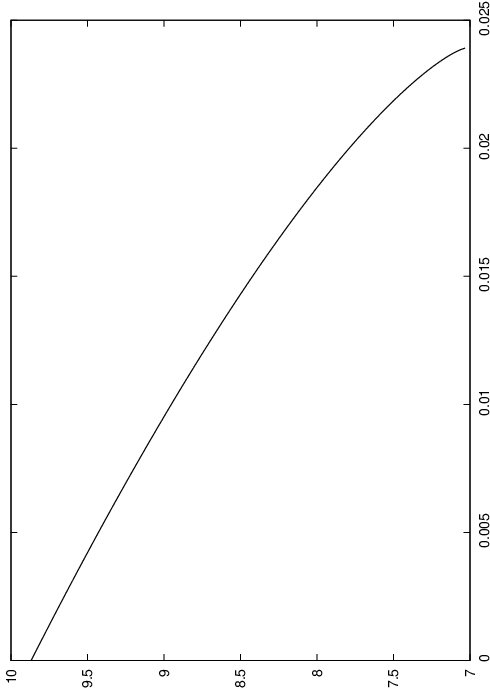}
\includegraphics[angle=-90,width=0.95\textwidth]{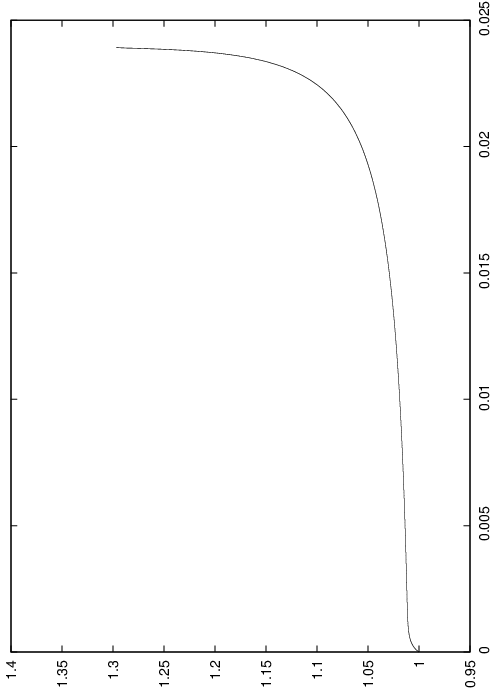}
\includegraphics[angle=-90,width=0.95\textwidth]{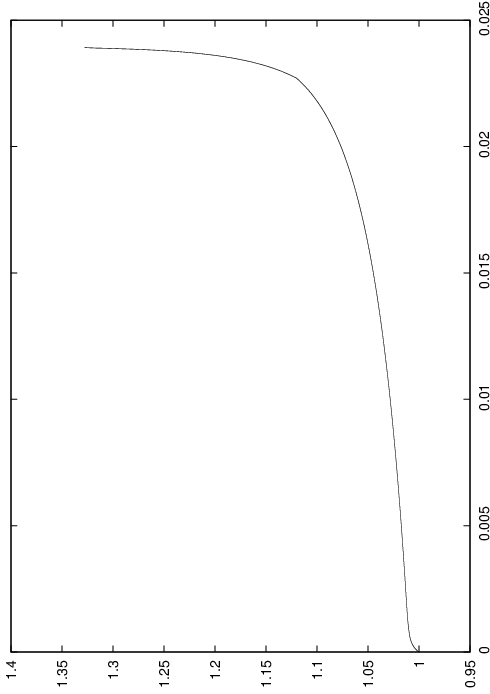}
\end{minipage}
\caption{
$(\BGNsd_m)^h$
Evolution for a torus with $R=1$ and $r=0.25$. Plots are at times
$t=0,0.002,\ldots,0.022,0.02392$. 
We also visualize the axisymmetric surface $\mathcal{S}^m$ generated by
$\Gamma^m$ at times $t=0$ (above) and $t=0.02392$ (below).
On the right are plots of the discrete energy and the ratio $\ratio^m$ and,
as a comparison, a plot of the ratio $\ratio^m$ for the scheme 
$(\BGNsdstab_{m,\star})$.}
\label{fig:sdtorusR1r025}
\end{figure}%
A detailed view of the vertex distribution at the final time, $t=0.02392$,
for the schemes $(\BGNsd_m)^h$ and $(\BGNsdstab_{m,\star})$ is given in
Figure~\ref{fig:sdtorusTM}. Here we note that the element ratios $\ratio^m$ 
at this time are $1.30$ and $1.33$, respectively. Hence the proximity of the 
$x_2$--axis has no dramatic effect on the vertex distribution.
\begin{figure}
\center
\includegraphics[angle=-90,width=0.25\textwidth]{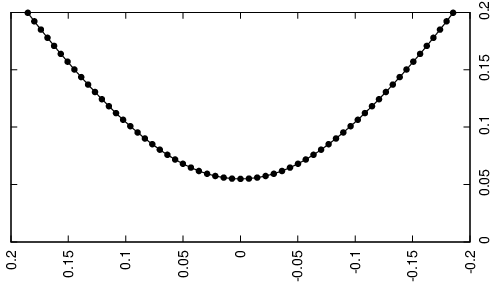}
\qquad
\includegraphics[angle=-90,width=0.25\textwidth]{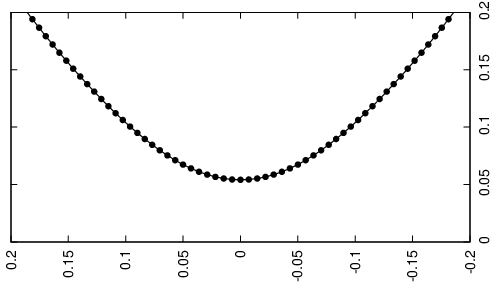}
\caption{
Detail of the vertex distribution at time $t=0.02392$ for the experiment in
Figure~\ref{fig:sdtorusR1r025} for the schemes
$(\BGNsd_m)^h$ (left) and $(\BGNsdstab_{m,\star})$ (right).
}
\label{fig:sdtorusTM}
\end{figure}%

\subsubsection{Droplet on a substrate}
Here we consider the evolution for a droplet on a substrate, so that e.g.\
$\partial_2 I = \{0\}$ and $\partial_0 I = \{1\}$.
See Figure~\ref{fig:sddrop-12} for a simulation for the choice
$\sliprho^{(0)}=-\frac12$.
Here we use $J=64$ and $\ttau = 10^{-3}$.
The relative volume loss for this experiment is $-0.64\%$ for the
scheme $(\BGNsd_m)^h$ and $-0.61\%$ for the scheme
$(\BGNsdstab_{m,\star})$.
\begin{figure}
\center
\begin{minipage}{0.35\textwidth}
\includegraphics[angle=-90,width=0.65\textwidth]{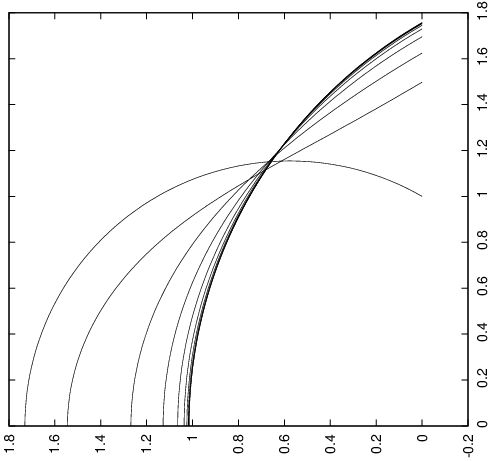} \\[5mm]
\includegraphics[angle=-90,width=0.65\textwidth]{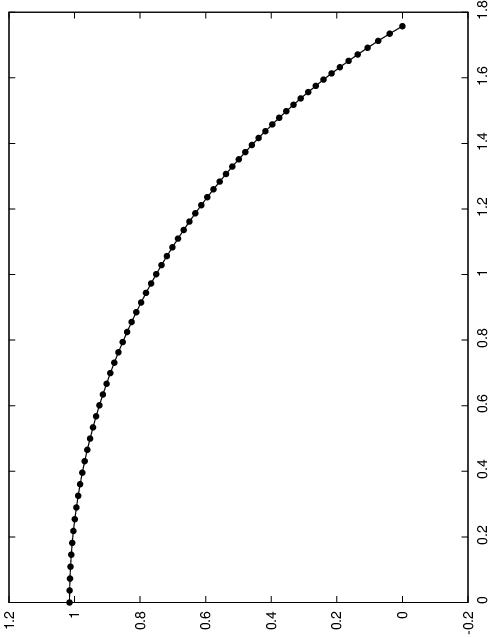} 
\includegraphics[angle=-90,width=0.95\textwidth]{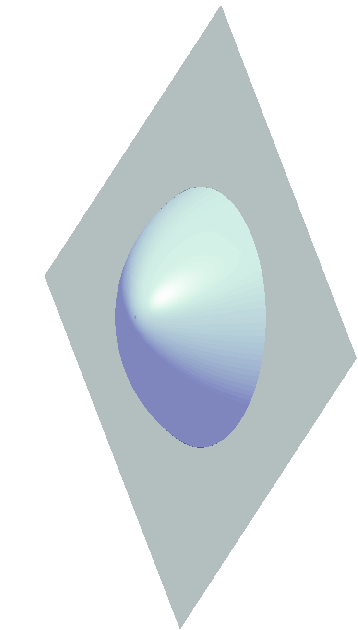} 
\end{minipage} \qquad
\begin{minipage}{0.3\textwidth}
\includegraphics[angle=-90,width=0.95\textwidth]{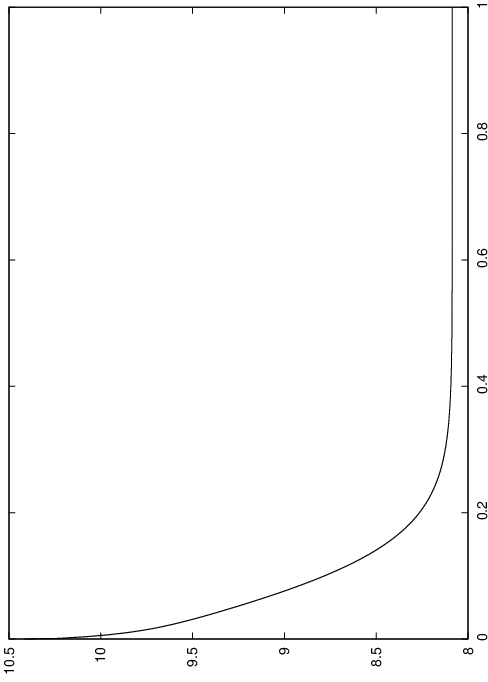} 
\includegraphics[angle=-90,width=0.95\textwidth]{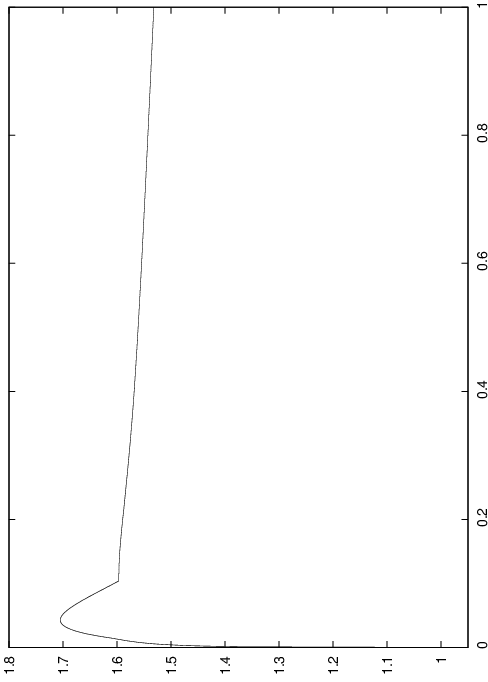} 
\includegraphics[angle=-90,width=0.95\textwidth]{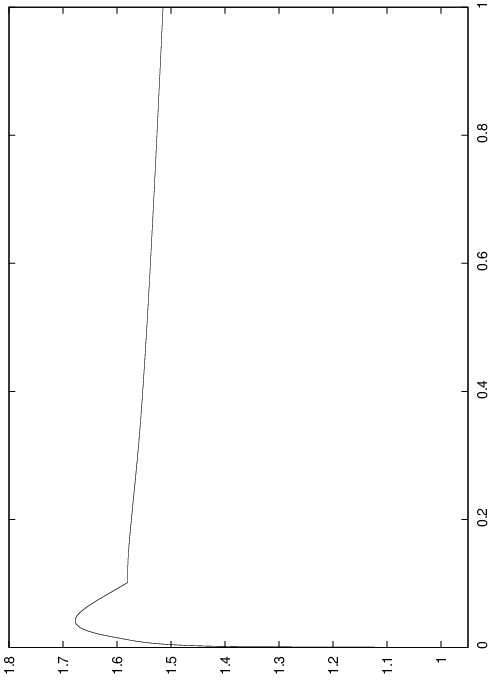} 
\end{minipage}
\caption{
$(\BGNsd_m)^h$ [$\partial_0 I = \{1\}$, $\partial_2 I = \{0\}$, 
$\sliprho^{(0)} =-\frac12$]
Evolution for a droplet attached to $\bR\times \{0\} \times \bR$. 
Solutions at times $t=0,0.1,\ldots,1$ and at time $t=1$.
We also visualize the axisymmetric surface $\mathcal{S}^m$ generated by
$\Gamma^m$ at time $t=1$.
On the right are plots of the discrete energy and the ratio 
(\ref{eq:ratio}) and,
as a comparison, a plot of the ratio $\ratio^m$ for the scheme 
$(\BGNsdstab_{m,\star})$.}
\label{fig:sddrop-12}
\end{figure}%
The same experiment with $\sliprho^{(0)} = 0.9$ can be seen in
Figure~\ref{fig:sddrop09}.
The relative volume loss for this experiment is $-0.13\%$ for the
scheme $(\BGNsd_m)^h$ and $-0.10\%$ for the scheme
$(\BGNsdstab_{m,\star})$.
\begin{figure}
\center
\begin{minipage}{0.35\textwidth}
\includegraphics[angle=-90,width=0.45\textwidth]{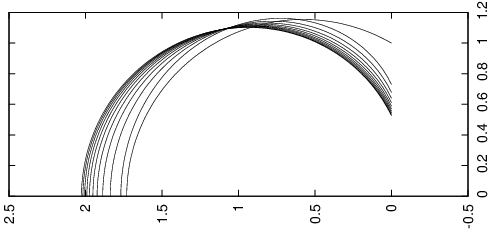}
\includegraphics[angle=-90,width=0.45\textwidth]{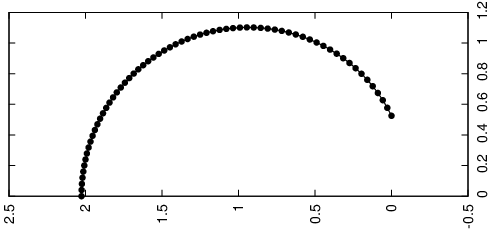} 
\includegraphics[angle=-90,width=0.95\textwidth]{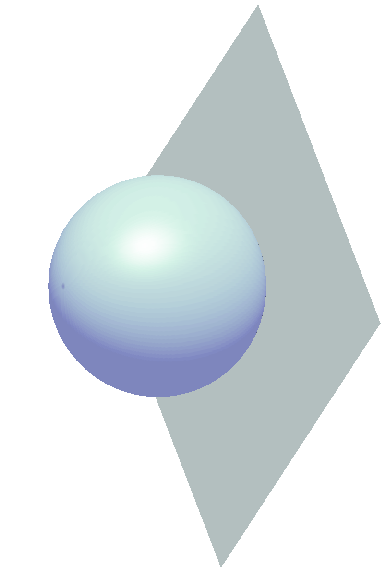} 
\end{minipage} \qquad
\begin{minipage}{0.3\textwidth}
\includegraphics[angle=-90,width=0.95\textwidth]{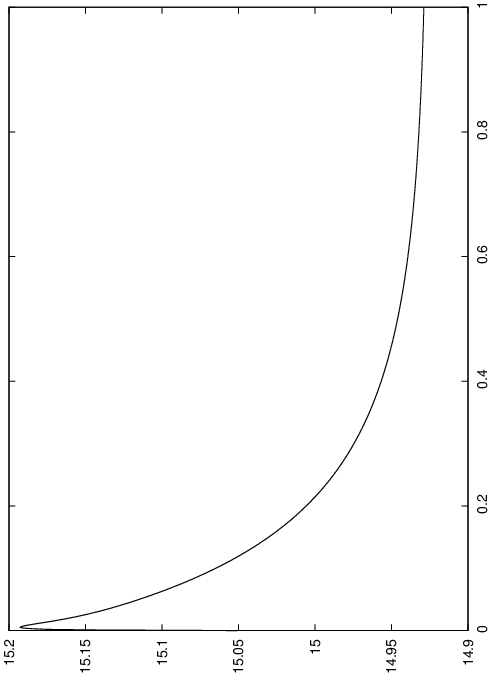} 
\includegraphics[angle=-90,width=0.95\textwidth]{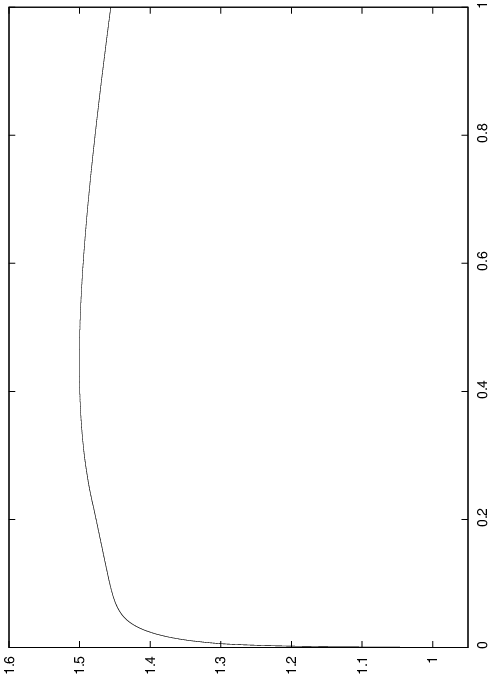} 
\includegraphics[angle=-90,width=0.95\textwidth]{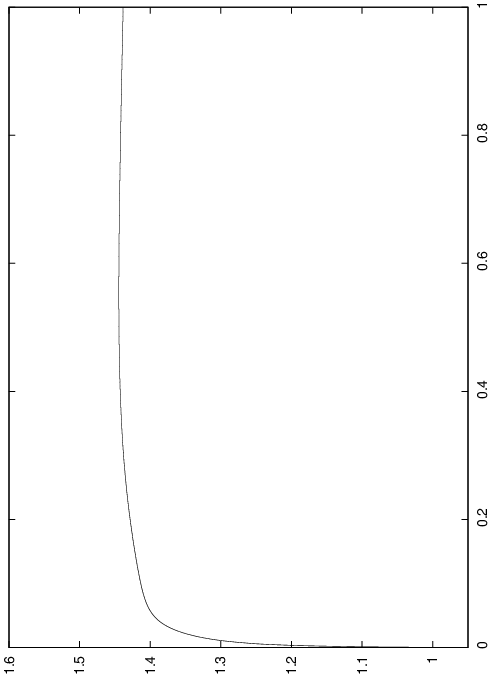} 
\end{minipage}
\caption{
$(\BGNsd_m)^h$ [$\partial_0 I = \{1\}$, $\partial_2 I = \{0\}$, 
$\sliprho^{(0)} = 0.9$]
Evolution for a droplet attached to $\mathcal{B}$.
Solutions at times $t=0,0.1,\ldots,1$ and at time $t=1$.
We also visualize the axisymmetric surface $\mathcal{S}^m$ generated by
$\Gamma^m$ at time $t=1$.
On the right are plots of the discrete energy and the ratio 
(\ref{eq:ratio}) and,
as a comparison, a plot of the ratio $\ratio^m$ for the scheme 
$(\BGNsdstab_{m,\star})$.}
\label{fig:sddrop09}
\end{figure}%

\subsubsection{Cut genus 1 surface on a substrate}

In this section, we show some experiments for the upper half of a genus 1
surface attached to the hyperplane 
$\bR\times\{0\}\times\bR$, so that $\partial_2 I = \partial I = \{0,1\}$.
See Figure~\ref{fig:torusslip} for an experiment with 
$J=129$ and $\ttau = 10^{-4}$.
The relative volume loss for this experiment is $0.47\%$ for the
scheme $(\BGNsd_m)^h$ and $0.43\%$ for the scheme $(\BGNsdstab_{m,\star})$.
\begin{figure}
\center
\begin{minipage}{0.35\textwidth}
\includegraphics[angle=-90,width=0.95\textwidth]{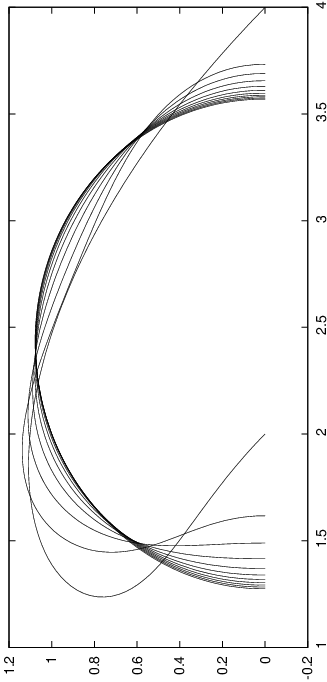}
\\[5mm]
\includegraphics[angle=-90,width=0.95\textwidth]{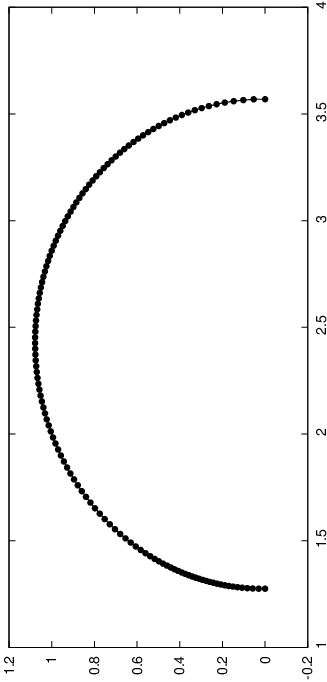} 
\\[5mm]
\includegraphics[angle=-90,width=0.95\textwidth]{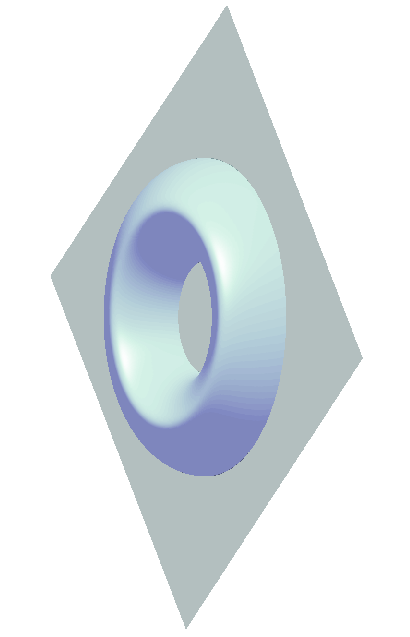} 
\end{minipage} \qquad
\begin{minipage}{0.3\textwidth}
\includegraphics[angle=-90,width=0.95\textwidth]{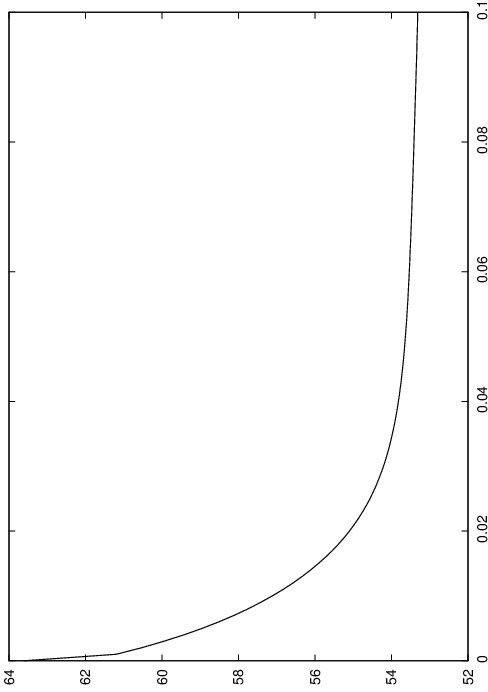} 
\includegraphics[angle=-90,width=0.95\textwidth]{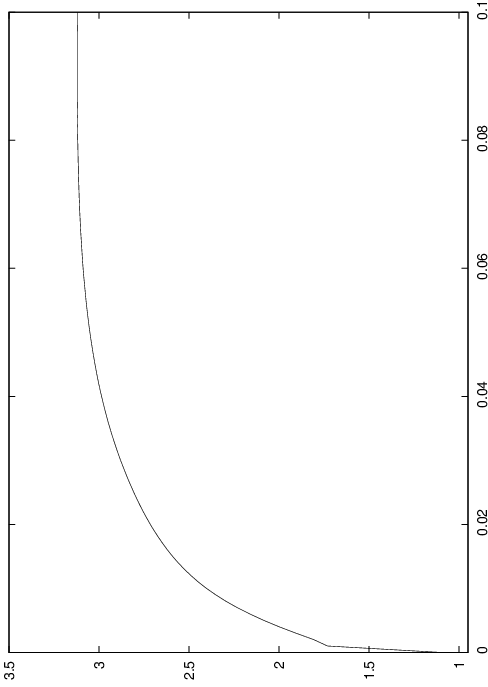} 
\includegraphics[angle=-90,width=0.95\textwidth]{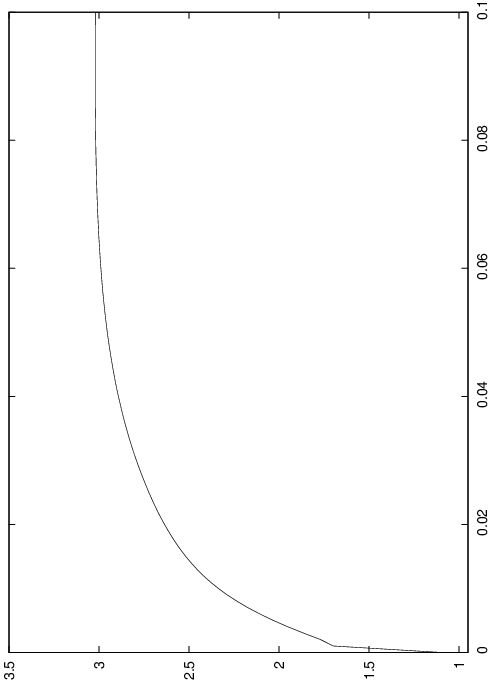} 
\end{minipage}
\caption{
$(\BGNsd_m)^h$ [$\partial_2 I = \partial I = \{0,1\}$,
$\sliprho^{(0)} = \sliprho^{(1)} =0$]
Evolution for the upper half of a genus 1 surface attached to 
$\bR\times\{0\}\times\bR$.
Solutions at times $t=0,0.01,\ldots,0.1$ and at time $t=0.1$.
We also visualize the axisymmetric surface $\mathcal{S}^m$ generated by
$\Gamma^m$ at time $t=0.1$.
On the right are plots of the discrete energy and the ratio 
(\ref{eq:ratio}) and,
as a comparison, a plot of the ratio $\ratio^m$ for the scheme 
$(\BGNsdstab_{m,\star})$.}
\label{fig:torusslip}
\end{figure}%
The same experiment with $\sliprho^{(0)} = -\sliprho^{(1)} = \tfrac12$ 
can be seen in Figure~\ref{fig:torusslip12-12}.
The relative volume loss for this experiment is $-0.25\%$ for both schemes
$(\BGNsd_m)^h$ and $(\BGNsdstab_{m,\star})$.
\begin{figure}
\center
\begin{minipage}{0.4\textwidth}
\includegraphics[angle=-90,width=0.99\textwidth]{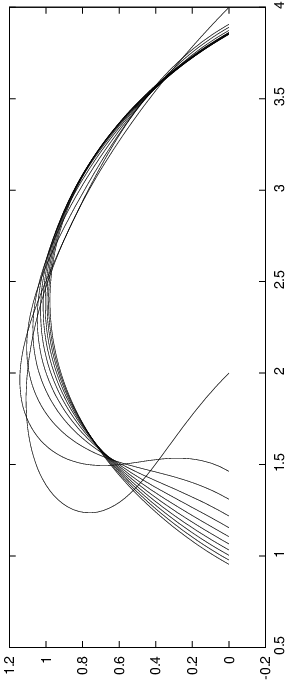}
\\[5mm]
\includegraphics[angle=-90,width=0.99\textwidth]{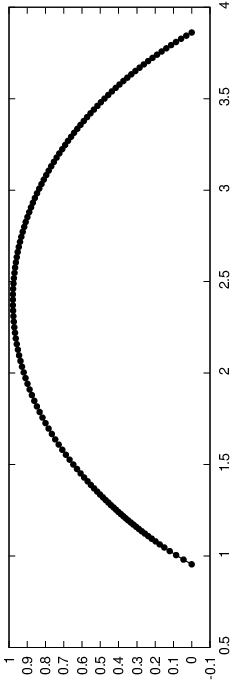} 
\\[5mm]
\includegraphics[angle=-90,width=0.95\textwidth]{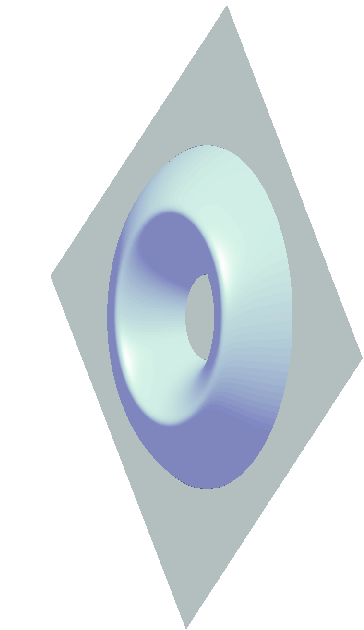} 
\end{minipage} \quad
\begin{minipage}{0.3\textwidth}
\includegraphics[angle=-90,width=0.95\textwidth]{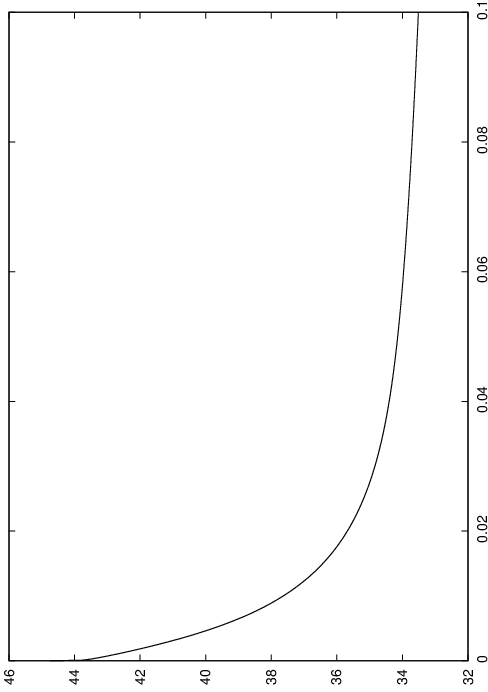} 
\includegraphics[angle=-90,width=0.95\textwidth]{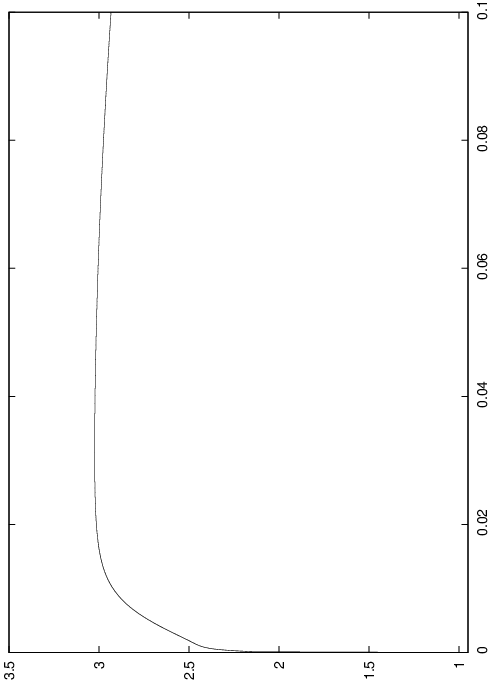} 
\includegraphics[angle=-90,width=0.95\textwidth]{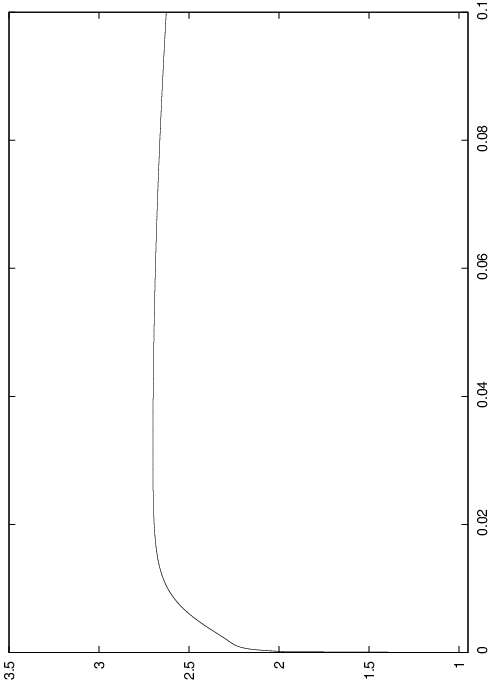} 
\end{minipage}
\caption{
$(\BGNsd_m)^h$ [$\partial_2 I = \partial I = \{0,1\}$,
$\sliprho^{(0)} = -\sliprho^{(1)} = \frac12$]
Evolution for the upper half of a genus 1 surface attached to 
$\bR\times\{0\}\times\bR$.
Solutions at times $t=0,0.01,\ldots,0.1$ and at time $t=0.1$.
We also visualize the axisymmetric surface $\mathcal{S}^m$ generated by
$\Gamma^m$ at time $t=0.1$.
On the right are plots of the discrete energy and the ratio 
(\ref{eq:ratio}) and, as a comparison, a plot of the ratio $\ratio^m$ for the scheme 
$(\BGNsdstab_{m,\star})$.}
\label{fig:torusslip12-12}
\end{figure}%

\subsubsection{Cut cylinder between two hyperplanes}

In this subsection we repeat the computations in 
\cite[Figs.\ 21,\ 22]{clust3d} for two open dumbbell-like cylindrical shapes 
attached to two parallel hyperplanes, see
Figures~\ref{fig:clust3dfig21} and \ref{fig:clust3dfig22},
and so we let $\partial_2 I= \partial I = \{0,1\}$.
In particular, 
in these experiments the two components of the boundary of $\mathcal{S}^m$
are attached to two distinct parallel hyperplanes. That means that 
$\pol X^m(0)$ is attached to the $x_1$--axis, while $\pol X^m(1)$ remains on
the line $\bR \times \{ a \}$, with $a = 4$ in Figure~\ref{fig:clust3dfig21}
and $a=8$ in Figure~\ref{fig:clust3dfig22}.
The initial data are given by
$\Gamma(0) = \{ (1 + \alpha\,\cos(2\,\pi\,\rho), \rho\,a)^T : 
\rho \in [0,1] \}$,
with $\alpha=0.5$ and $\alpha=0.25$, respectively.
For the discretization parameters we choose $J=128$ and $\ttau = 10^{-3}$.
The relative volume losses for these experiments are $-0.02\%$
and $-0.01\%$ for the scheme $(\BGNsd_m)^h$,
and $-0.01\%$ in both cases for the scheme $(\BGNsdstab_{m,\star})$.
\begin{figure}
\center
\begin{minipage}{0.2\textwidth}
\includegraphics[angle=-90,width=0.9\textwidth]{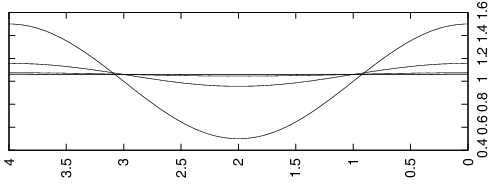} 
\end{minipage} \quad
\begin{minipage}{0.30\textwidth}
\includegraphics[angle=-90,width=0.95\textwidth]{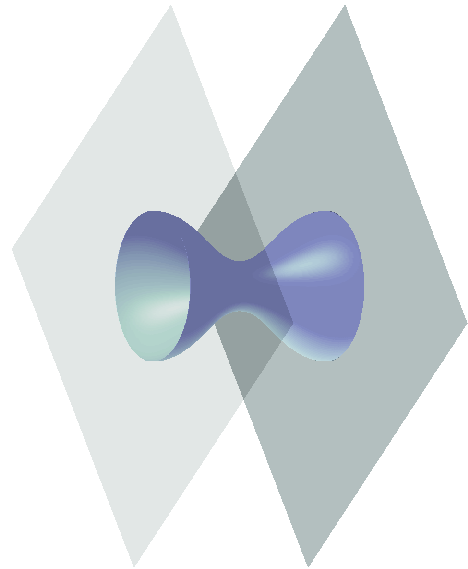} 
\includegraphics[angle=-90,width=0.95\textwidth]{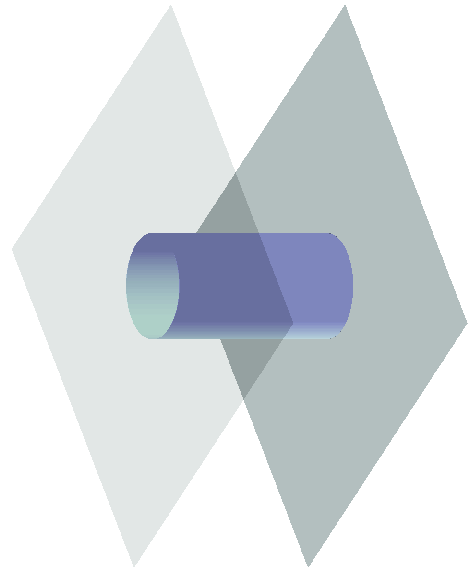} 
\end{minipage} \quad
\begin{minipage}{0.3\textwidth}
\includegraphics[angle=-90,width=0.95\textwidth]{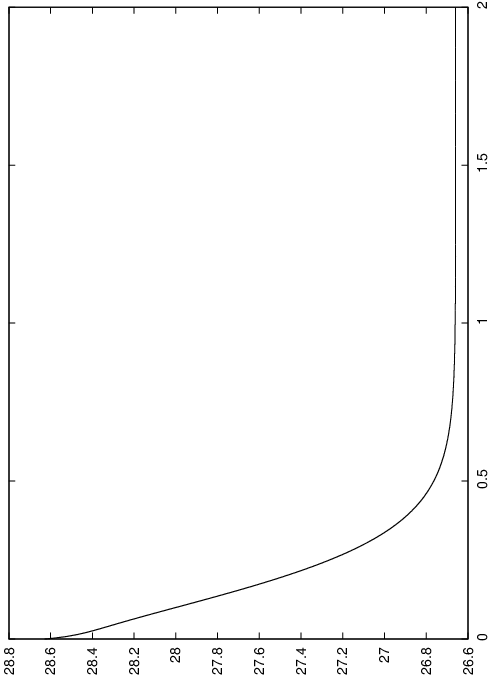} 
\includegraphics[angle=-90,width=0.95\textwidth]{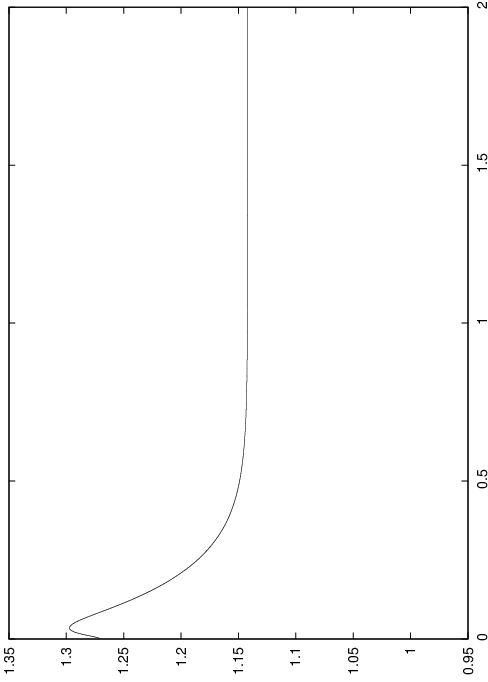}
\includegraphics[angle=-90,width=0.95\textwidth]{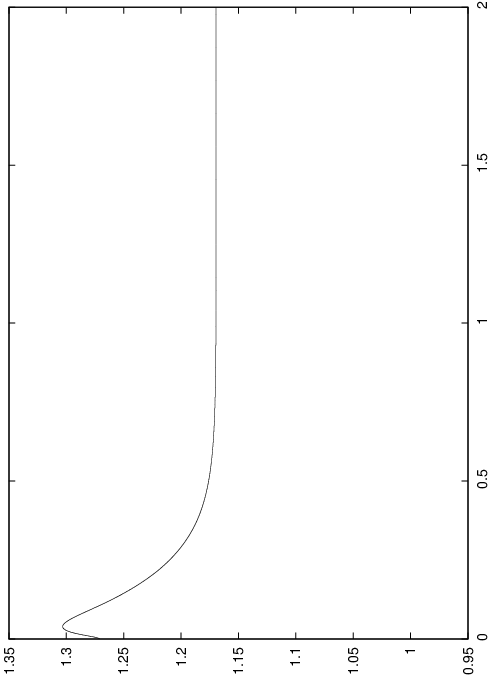}
\end{minipage} \quad
\caption{
$(\BGNsd_m)^h$ [$\partial_2 I = \partial I = \{0,1\}$, 
$\sliprho^{(0)} = \sliprho^{(1)} = 0$]
Evolution for an open dumbbell-like cylindrical shape 
attached to $\bR \times \{ 0 \} \times \bR$
and $\bR \times \{ 4 \} \times \bR$. Solution at times $t=0,0.5,\ldots,2$.
We also visualize the axisymmetric surface $\mathcal{S}^m$ generated by
$\Gamma^m$ at times $t=0$ (above) and $t=2$ (below).
On the right are plots of the discrete energy and the ratio 
(\ref{eq:ratio}) and, as a comparison, a plot of the ratio $\ratio^m$ for the 
scheme $(\BGNsdstab_{m,\star})$.}
\label{fig:clust3dfig21}
\end{figure}%
\begin{figure}
\center
\begin{minipage}{0.25\textwidth}
\mbox{
\includegraphics[angle=-90,width=0.4\textwidth]{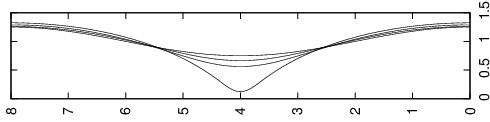} 
\includegraphics[angle=-90,width=0.59\textwidth]{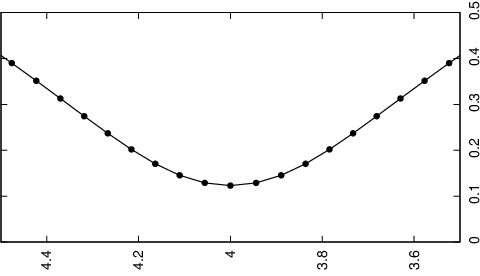} 
}
\end{minipage} \quad
\begin{minipage}{0.28\textwidth}
\includegraphics[angle=-90,width=0.95\textwidth]{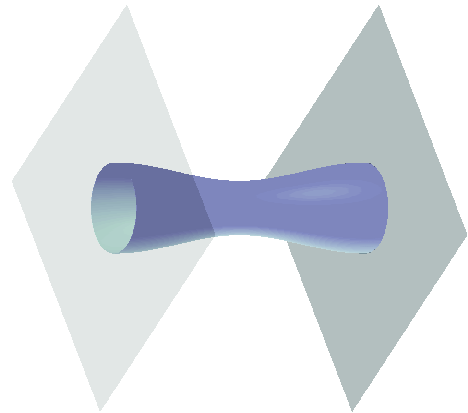} 
\includegraphics[angle=-90,width=0.95\textwidth]{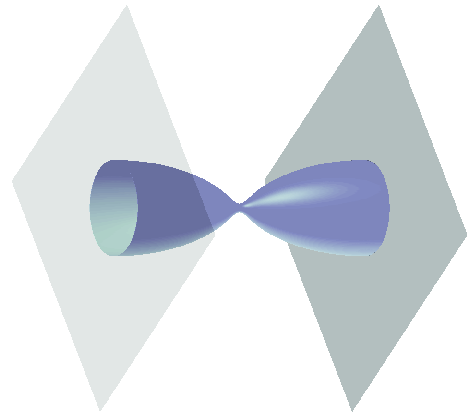} 
\end{minipage} \quad
\begin{minipage}{0.3\textwidth}
\includegraphics[angle=-90,width=0.95\textwidth]{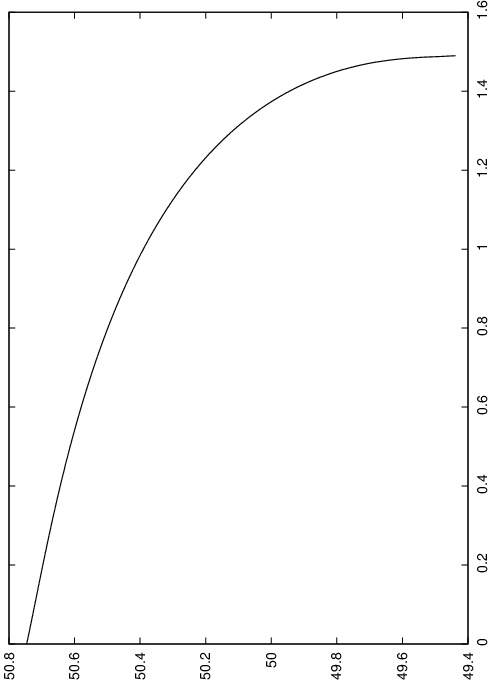} 
\includegraphics[angle=-90,width=0.95\textwidth]{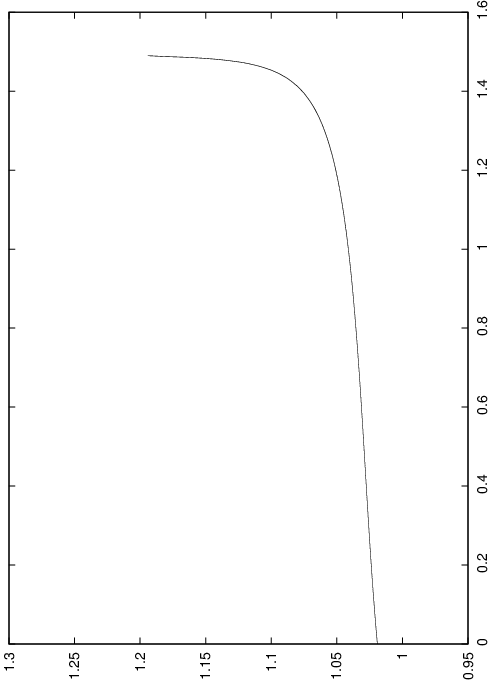}
\includegraphics[angle=-90,width=0.95\textwidth]{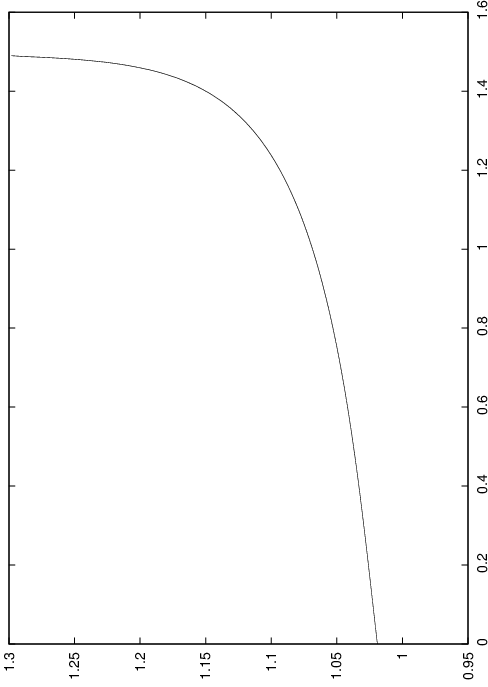}
\end{minipage}
\caption{
$(\BGNsd_m)^h$ [$\partial_2 I = \partial I = \{0,1\}$, 
$\sliprho^{(0)} = \sliprho^{(1)} = 0$]
Evolution for an open dumbbell-like cylindrical shape 
attached to $\bR \times \{ 0 \} \times \bR$
and $\bR \times \{ 8 \} \times \bR$. 
Solution at times $t=0,0.5,1,1.49$, and a detail of the vertex
distribution at time $t=1.49$.
We also visualize the axisymmetric surface $\mathcal{S}^m$ generated by
$\Gamma^m$ at times $t=0$ (above) and $t=1.49$ (below).
On the right are plots of the discrete energy and the ratio 
(\ref{eq:ratio}) and, as a comparison, a plot of the ratio $\ratio^m$ for the 
scheme $(\BGNsdstab_{m,\star})$.}
\label{fig:clust3dfig22}
\end{figure}%
We note that for the smaller aspect ratio of the shape in
Figure~\ref{fig:clust3dfig21}, the evolution reaches a cylinder. For the larger
aspect ratio in Figure~\ref{fig:clust3dfig22} the surface would like to undergo
pinch-off, which represents a singularity in the parametric approach. 
As a
consequence, the element ratio (\ref{eq:ratio}) increases to about $1.19$
for scheme $(\BGNsd_m)^h$, and to about $1.30$ for scheme 
$(\BGNsdstab_{m,\star})$.

\subsection{Numerical results for the intermediate evolution law} 
\label{sec:intnr}

We repeat the experiment in Figure~\ref{fig:sdcigar711} for the scheme
$(\BGNintstab_{m,\star})$ to approximate the flow (\ref{eq:SALK}), rather than
surface diffusion. We choose the values $\xi = \alpha = 1$, so that the flow
interpolates between surface diffusion and conserved mean curvature flow.
The results are shown in Figure~\ref{fig:salkcigar711}, where we note the
slower evolution compared to Figure~\ref{fig:sdcigar711}.
The discretization parameters are $J=128$ and $\ttau = 10^{-4}$.
The relative volume loss for this experiment is $0.00\%$.
\begin{figure}
\center
\begin{minipage}{0.35\textwidth}
\includegraphics[angle=-90,width=0.3\textwidth]{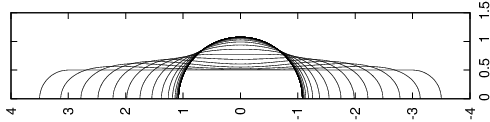} \quad
\includegraphics[angle=-90,width=0.4\textwidth]{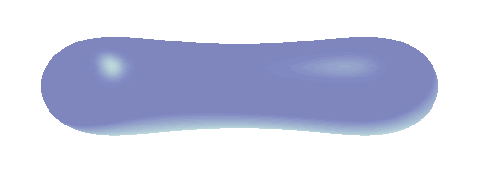}
\end{minipage} \qquad
\begin{minipage}{0.3\textwidth}
\includegraphics[angle=-90,width=0.95\textwidth]{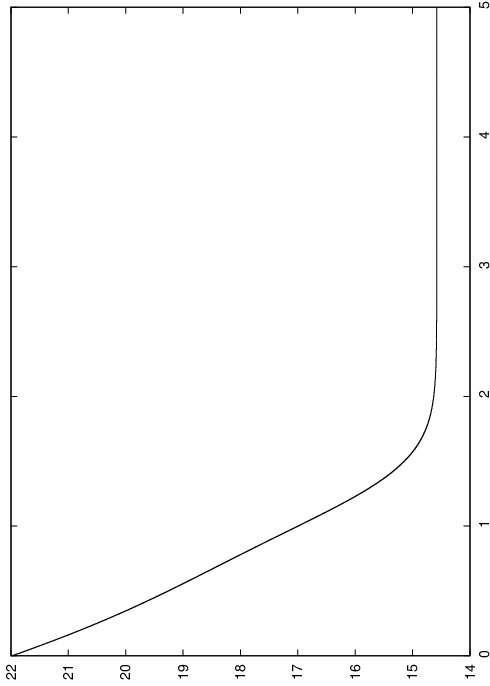} \\
\includegraphics[angle=-90,width=0.95\textwidth]{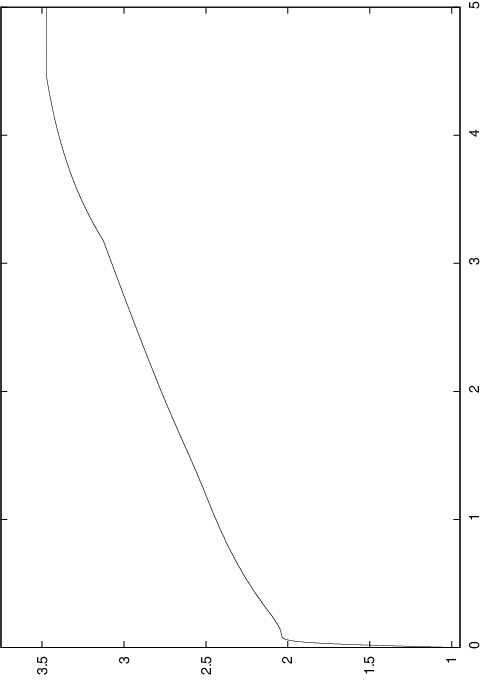}
\end{minipage}
\caption{
$(\BGNintstab_{m,\star})$
Evolution for a rounded cylinder of dimension $1\times7\times1$. 
Plots are at times $t=0,0.2,\ldots,5$. 
We also visualize the axisymmetric surface $\mathcal{S}^m$ generated by
$\Gamma^m$ at time $t=0.4$.
On the right are plots of the discrete energy and the ratio $\ratio^m$.
}
\label{fig:salkcigar711}
\end{figure}%
We mention that for the fully 3d approximation \cite[(2.27a--c)]{gflows3d} 
of the intermediate flow (\ref{eq:SALK}), some transient mesh ringing was
observed for a numerical simulation similar to Figure~\ref{fig:salkcigar711},
see \cite[Fig.\ 17]{gflows3d}.
Of course, in the axisymmetric setting considered in this paper, no such mesh
effects can ever occur.

\subsection{Numerical results for Willmore flow} \label{sec:wfnr}
Here present numerical results for the scheme $(\BGNwf_m)^h$,
recall (\ref{eq:bgnfd}).
As the fully discrete energy, we consider
\begin{equation} \label{eq:bgnWm}
W^h(\pol X^m) = \pi\left( \pol X^m\,.\,\pol\ek_1, 
\left(\kappa^{m+1} - \spont - \doctorkappa^{m}(\kappa^{m+1})\right)^2 
|\pol X^m_\rho| \right)^h\,.
\end{equation}
On recalling (\ref{eq:varkappa}), and given $\Gamma^0 = \pol X^0(\overline I)$,
we define the initial data $\kappa^0 \in V^h$ via
$\kappa^0 = \pi^h\left[\frac{\pol\kappa^0\,.\,\pol\omega^0}{|\pol\omega^0|}
\right]$,
where $\pol\kappa^0\in \Vh$ is such that
\begin{equation*} 
\left( \pol\kappa^{0},\pol\eta\, |\pol X^0_\rho| \right)^h
+ \left( \pol{X}^{0}_\rho , \pol\eta_\rho\,|\pol X^0_\rho|^{-1} \right)
 = 0 \qquad \forall\ \pol\eta \in \Vh\,.
\end{equation*}
Unless otherwise stated, we set $\spont=0$.

\subsubsection{Sphere}

We note that a sphere of radius $r(t)$, where $r(t)$ satisfies
\begin{equation} \label{eq:ODE}
r'(t) = - \tfrac\spont{r(t)}\,(\tfrac2{r(t)} + \spont)\,,
\quad r(0) = r_0 \in \bRplus\,,
\end{equation}
is a solution to (\ref{eq:Willmore_flow}).
The nonlinear ODE (\ref{eq:ODE}), in the case $\spont\not=0$, is solved by 
$r(t) = z(t) - \tfrac2\spont$, where $z(t)$ is such that 
$\tfrac12\,(z^2(t) - z_0^2) - \tfrac4\spont\,(z(t)-z_0) + \tfrac4{\spont^2}\,
\ln \tfrac{z(t)}{z_0} + \spont^2\,t = 0$,
with $z_0 = r_0 + \tfrac2\spont$.

We use the true solution (\ref{eq:ODE}) for a convergence experiment for the 
scheme $(\BGNwf_m)^h$. Here
we start with a nonuniform partitioning of a semicircle of radius 
$r(0)=r_0=1$ and compute the flow for $\spont = -1$ until time $T = 1$. 
In particular, we have $\partial_0 I = \partial I = \{0,1\}$ and we choose
$\pol X^0 \in \Vhpartialzero$ with
\begin{equation*} 
\pol X^0(q_j) = r_0 \begin{pmatrix} 
\cos[(q_j-\tfrac12)\,\pi + 0.1\,\cos((q_j-\tfrac12)\,\pi)] \\
\sin[(q_j-\tfrac12)\,\pi + 0.1\,\cos((q_j-\tfrac12)\,\pi)]
\end{pmatrix}, \quad j = 0,\ldots,J\,,
\end{equation*}
recall (\ref{eq:Jequi}). We compute the error
$\errorXx = \max_{m=1,\ldots,M} \max_{j=0,\ldots,J} 
| |\pol X^m(q_j)| - r(t_m)|$ over the time interval $[0,T]$ between
the true solution and the discrete solutions for the scheme $(\BGNwf_m)^h$.
Here we use the time step size $\ttau=0.1\,h^2_{\Gamma^0}$,
where $h_{\Gamma^0}$ is the maximal edge length of $\Gamma^0$.
The computed errors are reported in Table~\ref{tab:bgnspont-1},
where we observe a convergence rate of $\mathcal{O}(h^2_{\Gamma^0})$.
\begin{table}
\center
\caption{$(\BGNwf_m)^h$
Errors for the convergence test (\ref{eq:ODE}) with $\spont = -1$.}
\begin{tabular}{rrcccc}
\hline
$J$ & $h_{\Gamma^0}$ & $\errorXx$ & EOC  \\ \hline
32  & 1.0792e-01 & 1.9659e-03 & ---      \\
64  & 5.3988e-02 & 5.1262e-04 & 1.940681 \\         
128 & 2.6997e-02 & 1.2980e-04 & 1.981917 \\
256 & 1.3499e-02 & 3.2571e-05 & 1.994737 \\
512 & 6.7495e-03 & 8.1512e-06 & 1.998504 \\
\hline
\end{tabular}
\label{tab:bgnspont-1}
\end{table}%

\subsubsection{Genus 0 surface}
The evolution for Willmore flow for the same
initial data as in Figure~\ref{fig:sdflatcigar}
is shown in Figure~\ref{fig:bgnflatcigar}.
The discretization parameters for the scheme $(\BGNwf_m)^h$ 
are $J=128$ and $\ttau = 10^{-3}$. As expected, the flat disc evolves to a
sphere. At time $t=10$ the discrete Willmore energy (\ref{eq:bgnWm}) is
$25.330$, and continuing the evolution until time $t=100$ yields an energy
of $25.131$. 
This compares well with the value $8\,\pi = 25.133$, 
which is the Willmore energy (\ref{eq:W}), for $\spont=0$, of a sphere.
\begin{figure}
\center
\begin{minipage}{0.65\textwidth}
\includegraphics[angle=-90,width=0.5\textwidth]{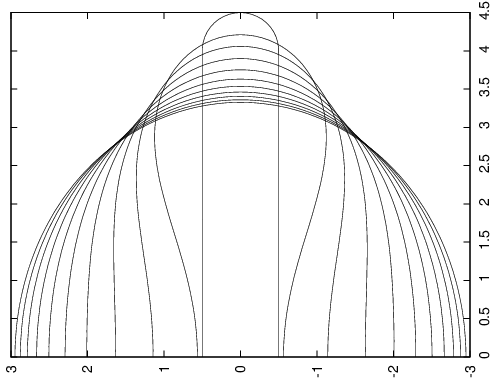} \quad
\includegraphics[angle=-90,width=0.45\textwidth]{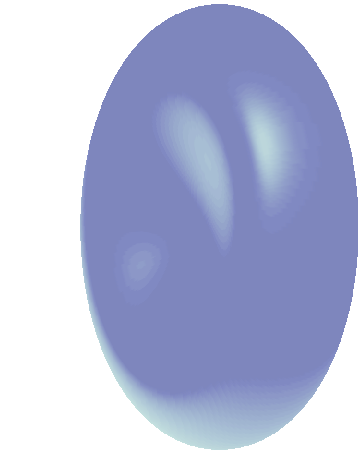} 
\end{minipage}
\begin{minipage}{0.3\textwidth}
\includegraphics[angle=-90,width=0.95\textwidth]{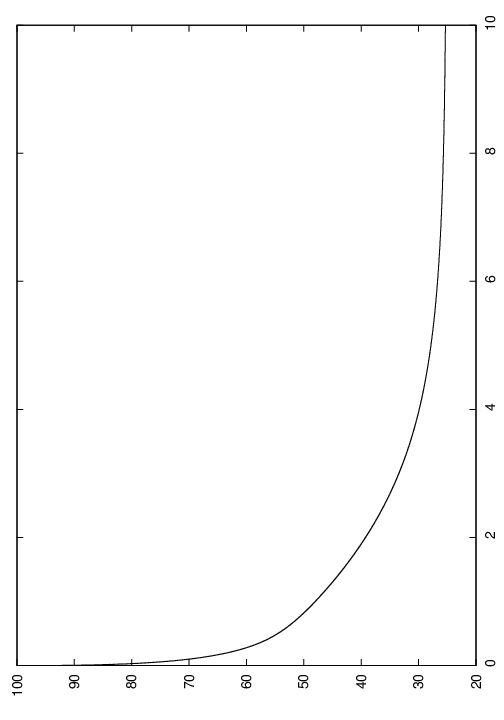}
\includegraphics[angle=-90,width=0.95\textwidth]{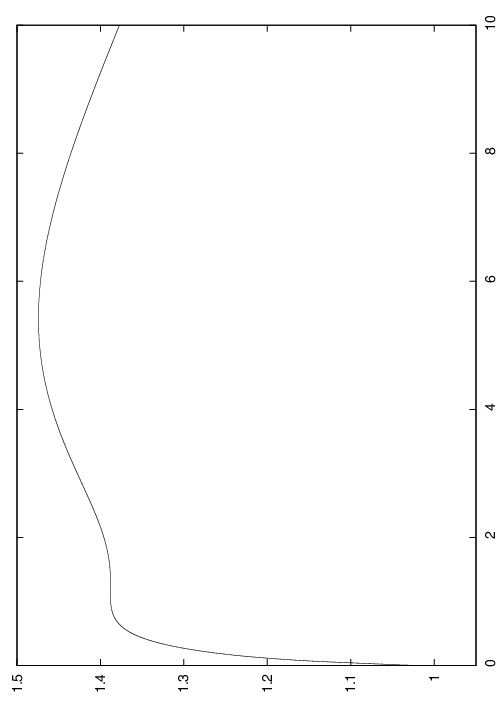}
\end{minipage}
\caption{
$(\BGNwf_m)^h$
Willmore flow for a disc of dimension $9\times1\times9$. 
Solution at times $t=0,1,\ldots,10$.
We also visualize the axisymmetric surface $\mathcal{S}^m$ generated by
$\Gamma^m$ at time $t=1$.
On the right a plot of the discrete energy and of the ratio (\ref{eq:ratio}).} 
\label{fig:bgnflatcigar}
\end{figure}%
Repeating the simulation with $\spont=-2$ yields the results in
Figure~\ref{fig:bgnflatcigarspont-2}, where we observe that the final steady
state now approximates the unit sphere. In fact, the discrete energy
(\ref{eq:bgnWm}) at time $t=3$ is $1.8 \times 10^{-5}$,
which compares with the energy (\ref{eq:W}), for $\spont=-2$, being zero for a
unit sphere.
\begin{figure}
\center
\begin{minipage}{0.65\textwidth}
\includegraphics[angle=-90,width=0.6\textwidth]{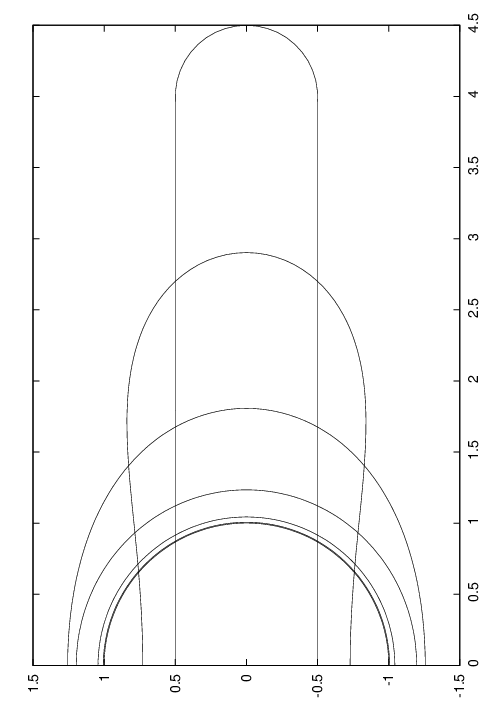} 
\includegraphics[angle=-90,width=0.39\textwidth]{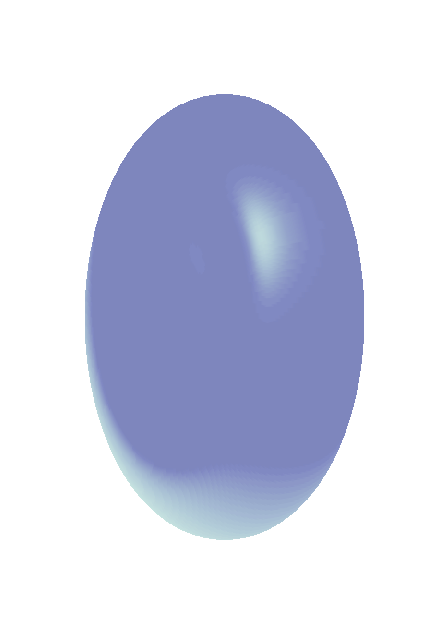} 
\end{minipage}
\begin{minipage}{0.3\textwidth}
\includegraphics[angle=-90,width=0.95\textwidth]{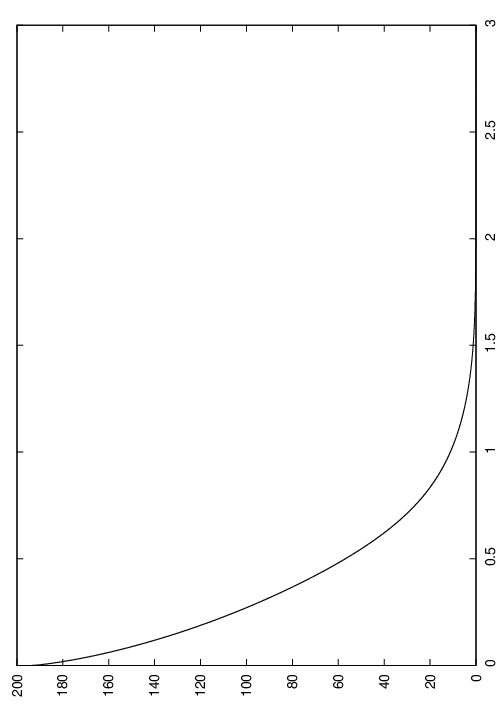}
\includegraphics[angle=-90,width=0.95\textwidth]{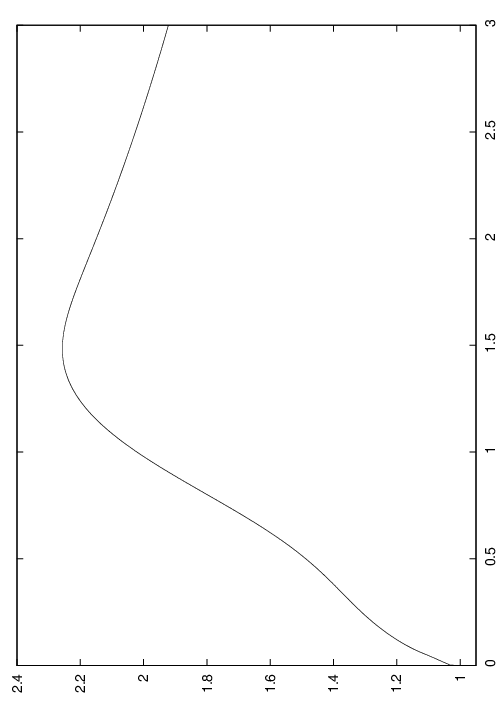}
\end{minipage}
\caption{
$(\BGNwf_m)^h$
Willmore flow with $\spont=-2$ for a disc of dimension $9\times1\times9$. 
Solution at times $t=0,0.5,\ldots,3$.
We also visualize the axisymmetric surface $\mathcal{S}^m$ generated by
$\Gamma^m$ at time $t=0.5$.
On the right a plot of the discrete energy and of the ratio (\ref{eq:ratio}).} 
\label{fig:bgnflatcigarspont-2}
\end{figure}%
We also repeat the computation in \cite[Fig.\ 9]{pwfade} for a rounded
cylinder of total dimension $2\times6\times2$, 
see Figure~\ref{fig:bgnpwfade9}. Here the surface would like
to pinch off into two unit spheres.
The discretization parameters are $J=128$ and $\ttau = 10^{-3}$.
We note that at time $t=1$, the ratio $\ratio^m$ has reached a value of
$1.14$. Hence, despite the proximity to the $x_2$--axis, the vertices are
still nearly equidistributed.
\begin{figure}
\center
\begin{minipage}{0.6\textwidth}
\includegraphics[angle=-90,width=0.2\textwidth]{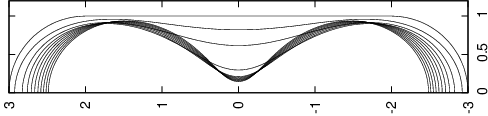} 
\quad
\includegraphics[angle=-90,width=0.35\textwidth]{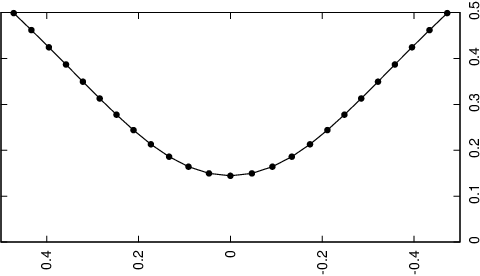} 
\qquad
\includegraphics[angle=-90,width=0.3\textwidth]{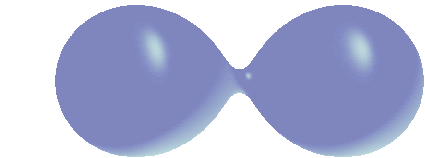} 
\end{minipage} \qquad
\begin{minipage}{0.3\textwidth}
\includegraphics[angle=-90,width=0.9\textwidth]{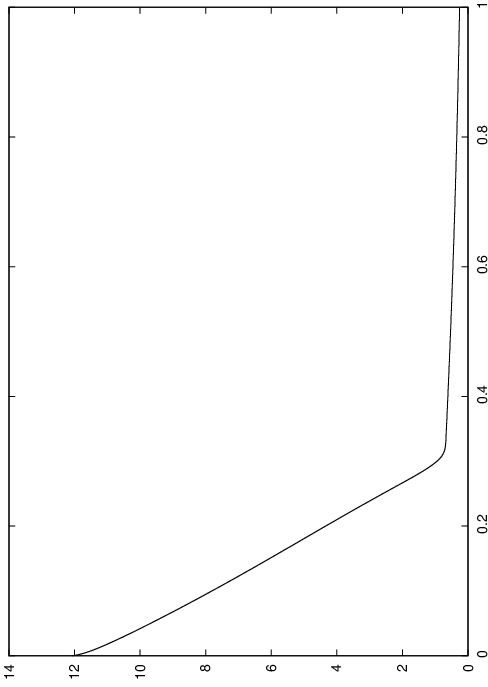}
\includegraphics[angle=-90,width=0.9\textwidth]{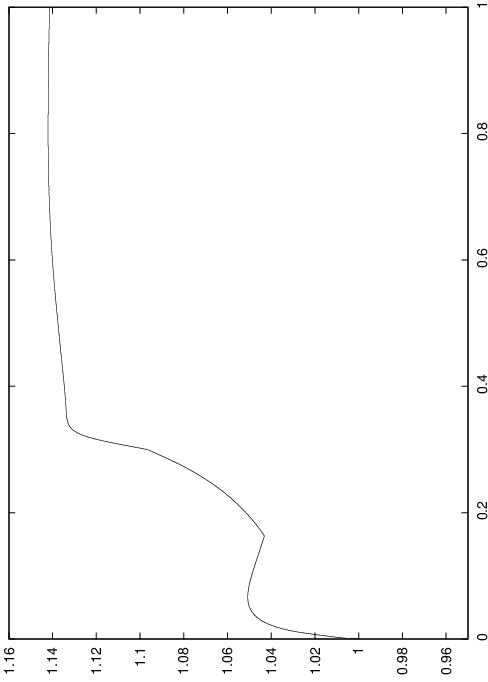}
\end{minipage}
\caption{
$(\BGNwf_m)^h$
Willmore flow with $\spont=-2$ for a rounded
cylinder of dimension $2\times6\times2$. 
Solution at times $t=0,0.1,\ldots,1$, and a detail of the vertex
distribution at time $t=1$.
We also visualize the axisymmetric surface $\mathcal{S}^m$ generated by
$\Gamma^m$ at time $t=1$.
On the right are plots of the discrete energy and of the ratio
(\ref{eq:ratio}).} 
\label{fig:bgnpwfade9}
\end{figure}%

\subsubsection{Genus 1 surface}
Using as initial data for Willmore flow the surface generated by the
curve $\Gamma(0)$ that is given by an elongated cigar-like shape 
of total dimension $4\times1$, with barycentre $(4,0)^T\in\bR^2$, we observe
the numerical evolution shown in Figure~\ref{fig:bgncigar41}.
The discretization parameters are $J=128$ and $\ttau = 10^{-3}$.
The observed final radius of $\Gamma^m$ is $2.11$, with the centre of the
circle at $(3.06,0)$. Hence the ratio of the two radii of the torus is
$R/r = 3.06/2.11 = 1.4488$, which will tend to $\sqrt{2}$ as the evolution
continues further. In fact, continuing the evolution until time $t=10$ yields
a ratio $R/r = 3.03 / 2.15 = 1.4140$ and a discrete energy (\ref{eq:bgnWm}) 
of $39.484$. Here we recall that the ratio $\sqrt{2}$ characterizes the 
Clifford torus, the known minimizer of the Willmore energy (\ref{eq:W}), with
$\spont=0$, among all genus $1$ surfaces, see \cite{MarquesN14}, with Willmore
energy equal to $4\,\pi^2 = 39.478$.
\begin{figure}
\center
\begin{minipage}{0.65\textwidth}
\includegraphics[angle=-90,width=0.60\textwidth]{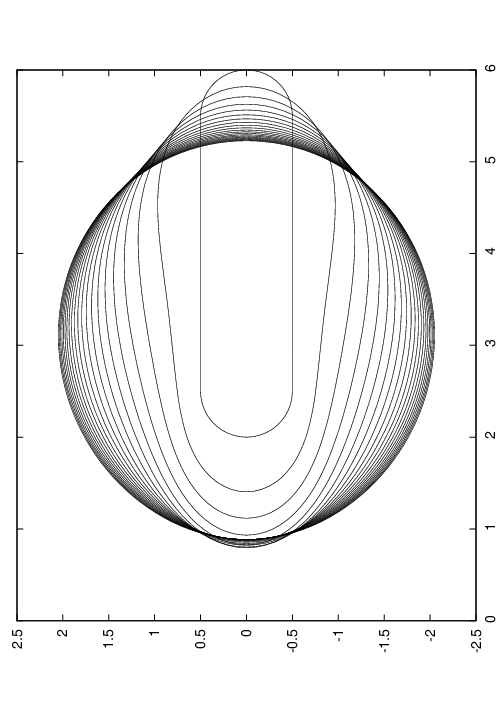}
\includegraphics[angle=-90,width=0.39\textwidth]{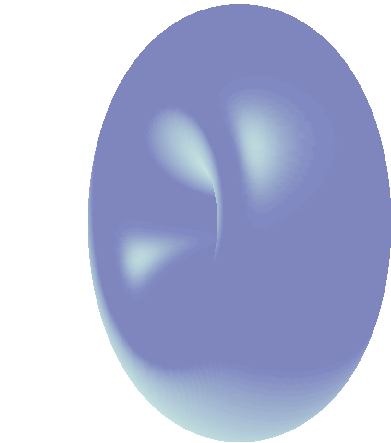}
\end{minipage}
\begin{minipage}{0.3\textwidth}
\includegraphics[angle=-90,width=0.95\textwidth]{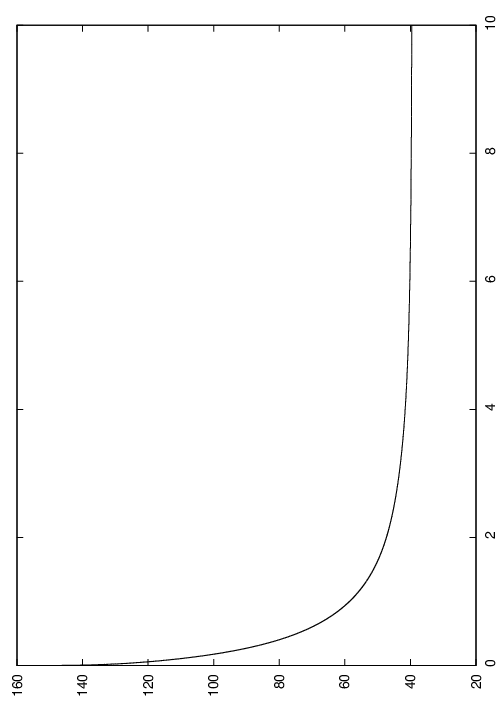}
\includegraphics[angle=-90,width=0.95\textwidth]{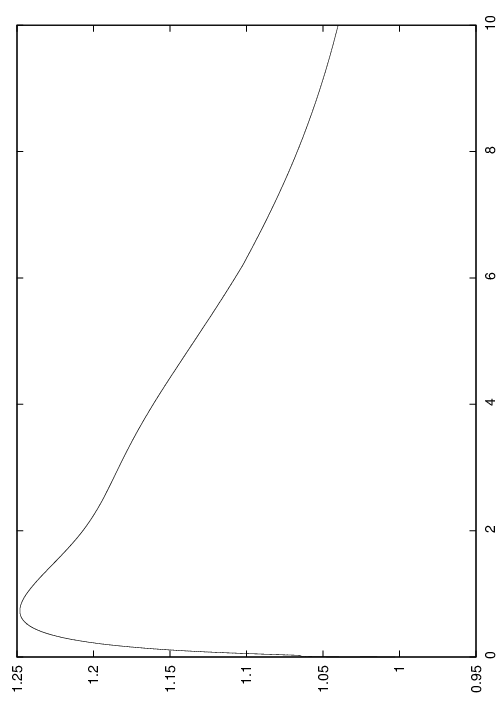}
\end{minipage}
\caption{
$(\BGNwf_m)^h$
Willmore flow towards a Clifford torus. 
Solution at times $t=0,0.5,\ldots,10$.
We also visualize the axisymmetric surface $\mathcal{S}^m$ generated by
$\Gamma^m$ at time $t=10$.
On the right a plot of the discrete energy and of the ratio (\ref{eq:ratio}).
}
\label{fig:bgncigar41}
\end{figure}%
Repeating the simulation in Figure~\ref{fig:bgncigar41} with $\spont=-2$
gives the results in Figure~\ref{fig:bgncigar41spont-2}. 
The observed final radius of $\Gamma^m$ is $0.498$, with the centre of the
circle at $(4.26,0)$. Hence the ratio of the two radii of the torus is now
$R/r = 4.06/0.498 = 8.15$.
\begin{figure}
\center
\includegraphics[angle=-90,width=0.45\textwidth]{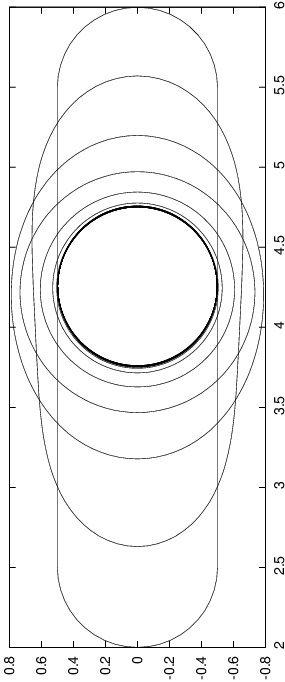}
\qquad
\includegraphics[angle=-90,width=0.25\textwidth]{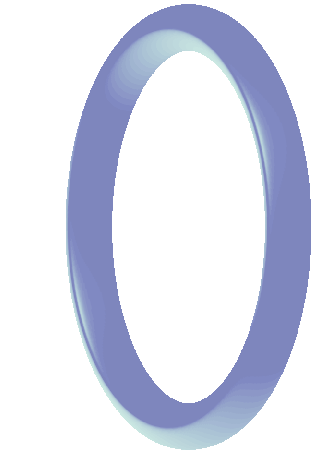} 
\\[2mm]
\includegraphics[angle=-90,width=0.3\textwidth]{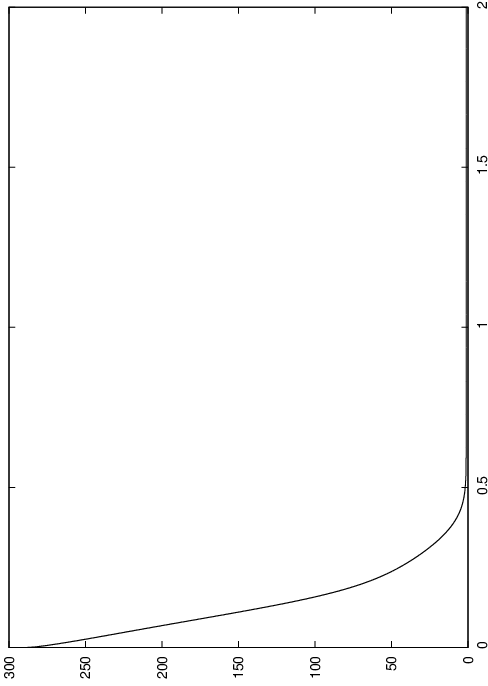}
\includegraphics[angle=-90,width=0.3\textwidth]{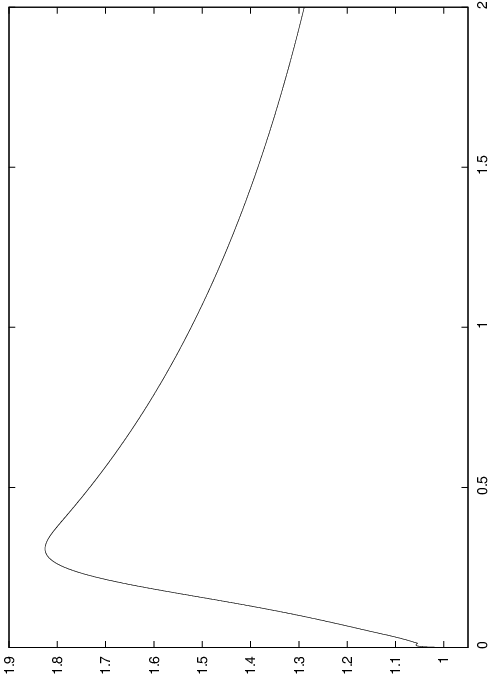}
\caption{
$(\BGNwf_m)^h$
Willmore flow with $\spont=-2$ towards a torus. 
Solution at times $t=0,0.1,\ldots,2$.
We also visualize the axisymmetric surface $\mathcal{S}^m$ generated by
$\Gamma^m$ at time $t=2$.
Below a plot of the discrete energy and of the ratio (\ref{eq:ratio}).} 
\label{fig:bgncigar41spont-2}
\end{figure}%

In order to study the development of a singularity under Willmore flow,
we consider the evolution from \cite[Figs.\ 8, 9]{MayerS02}.
In particular, as initial data for the scheme $(\BGNwf_m)^h$ 
we choose a curve that is the union of a circle of radius $0.5$, and two
quarter circles of radius $2$, see Figure~\ref{fig:limacon}.
The discretization parameters are $J=1024$ and $\ttau = 4\times10^{-5}$.
It can be seen from the numerical results shown in 
Figure~\ref{fig:limacon} that the scheme $(\BGNwf_m)^h$ 
computes an evolution of a shape with a loop with large curvature and two
circular segments that increase in size. We conjecture that as $t\to\infty$,
upon rescaling to a shape of fixed diameter, the surface approaches two
touching spheres. This would resemble a singularity for Willmore flow.
We note that the existence of surfaces that
become singular under Willmore flow was proven in \cite{Blatt09}. 
More precisely, it was shown that either
a finite time singularity occurs, or that a rescaled infinite time solution
becomes singular for large times. It is stated in \cite[p.\ 408]{Blatt09}
that ``either a small quantum of the curvature concentrates or the diameter 
of the surface does not stay bounded under the Willmore flow''. 
Our simulations indicate that the latter can happen and in contrast to
\cite{MayerS02} we did not found any indication for a finite time singularity.
Here we remark that the authors in \cite[Fig.\ 8]{MayerS02},
who also exploit an
additional symmetry and only compute the evolution for half the generating 
curve, appear to have
performed a topological change to yield two touching spheres at a finite time.
Given our numerical results we believe that this heuristical topological change
was not justified, and the simulation should have been continued normally.
Repeating the simulation in Figure~\ref{fig:limacon}
for $J=2048$ and $\ttau=10^{-5}$ until time $t=100$ yields 
very good agreement between the shapes of the curves for our two experiments,
and so we are satisfied that the evolution shown in Figure~\ref{fig:limacon}
approximates Willmore flow of the initial data.
We remark that the discrete energy (\ref{eq:bgnWm}) at time $t=1000$ 
for the run in Figure~\ref{fig:limacon} is
$50.739$, with the Willmore energy, (\ref{eq:W}) for $\spont=0$, for two
touching spheres being equal to $16\,\pi = 50.265$. 
Finally, in order to better understand the long-time behaviour of the
``radius'' of the two approximate expanding spheres, we plot in 
Figure~\ref{fig:limacon_radius} the quantities
$\max_{\overline I} \pol X^m\,.\,\pol\ek_1$ and
$\frac14\,(\max_{\overline I} \pol X^m\,.\,\pol\ek_2 - \min_{\overline I} 
\pol X^m\,.\,\pol\ek_2)$ over time. We fit both curves to a function of the
form $f(t) = a\,t^p$. For the former curve, we obtain a value $p=0.222$,
while for the second curve we obtain the power $p=0.232$.
\begin{figure}
\center
\begin{minipage}{0.4\textwidth}
\includegraphics[angle=-90,width=0.3\textwidth]{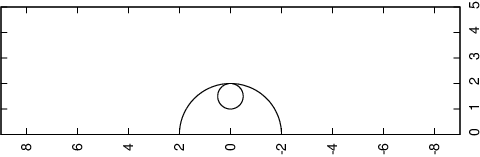}
\includegraphics[angle=-90,width=0.3\textwidth]{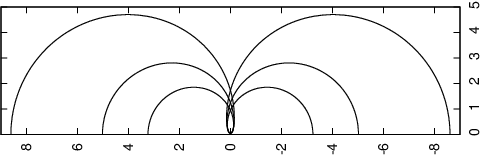} 
\includegraphics[angle=-90,width=0.3\textwidth]{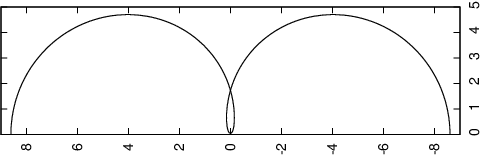}
\end{minipage}
\begin{minipage}{0.5\textwidth}
\includegraphics[angle=-90,width=0.45\textwidth]{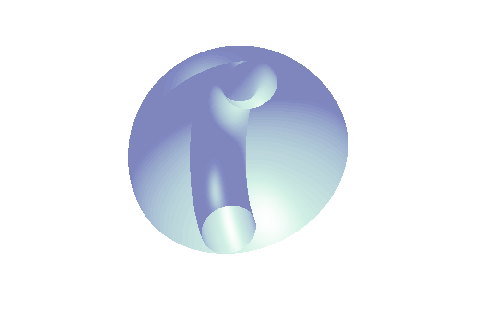}
\includegraphics[angle=-90,width=0.45\textwidth]{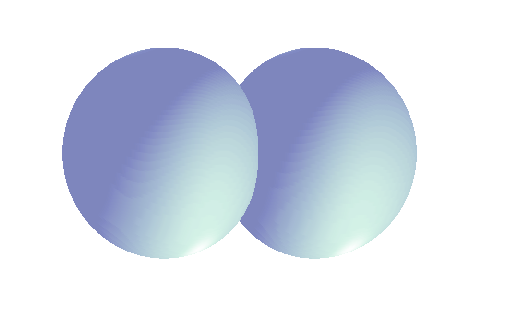} 
\end{minipage}
\includegraphics[angle=-90,width=0.3\textwidth]{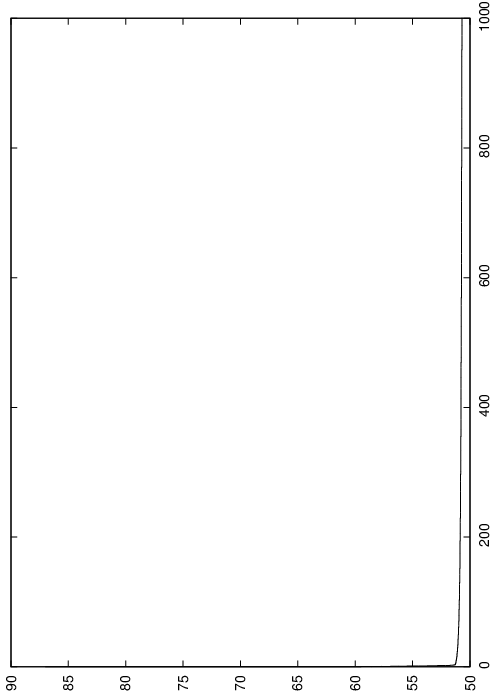}
\includegraphics[angle=-90,width=0.3\textwidth]{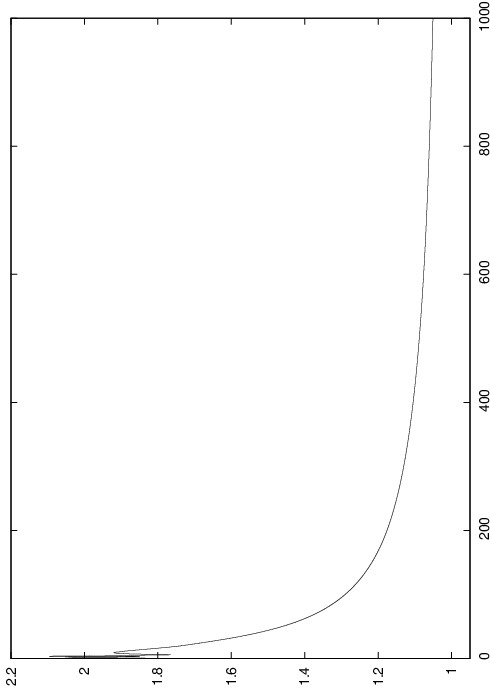}
\caption{
$(\BGNwf_m)^h$
Willmore flow towards two touching spheres. 
In the first three plots we show the initial data, the 
solution at times $t=10,100,1000$, and again at time $t=1000$.
We also visualize parts of the axisymmetric surface $\mathcal{S}^m$ 
generated by $\Gamma^m$ at time $t=0$ and at time $t=10$.
Below a plot of the discrete energy and of the ratio (\ref{eq:ratio}).} 
\label{fig:limacon}
\end{figure}%
\begin{figure}
\center
\includegraphics[angle=-90,width=0.4\textwidth]{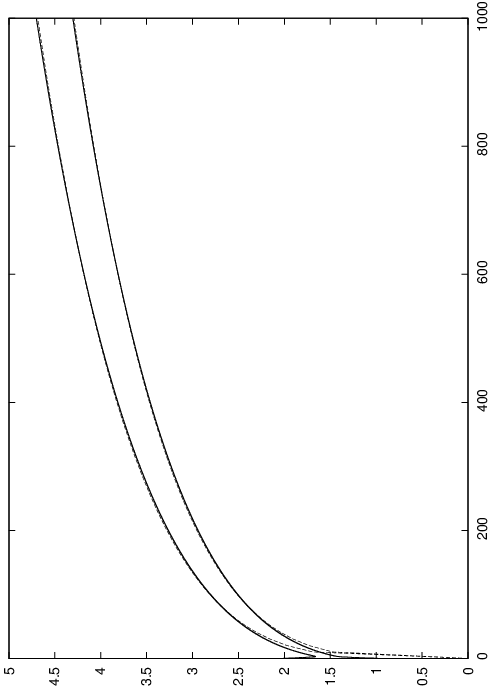}
\caption{
A plot of $\max_{\overline I} \pol X^m\,.\,\pol\ek_1$ (upper graph) and
$\frac14\,(\max_{\overline I} \pol X^m\,.\,\pol\ek_2 - \min_{\overline I} 
\pol X^m\,.\,\pol\ek_2)$ (lower graph) over time, for the simulation in
Figure~\ref{fig:limacon}, together with the functions
$f_i(t) = a_i\,t^{p_i}$, $i = 1,2$, with $(a_1,p_1) = (1.013, 0.222)$
and $(a_2,p_2) = (0.863, 0.232)$.}
\label{fig:limacon_radius}
\end{figure}%

\subsection{Numerical results for Helfrich flow} \label{sec:hfnr}
Here we present some simulations for the scheme
$(\BGNwf_m^{A,V})^h$, recall (\ref{eq:bgnwffd}).

\subsubsection{Genus 0 surface}
We repeat the computation in \cite[Fig.\ 6]{pwfade}
for Helfrich flow with $\spont=0$
of a rounded cylinder of total dimension $1\times4\times1$.
The discretization parameters are $J=128$ and $\ttau = 10^{-3}$.
We observe relative surface area and volume losses of $0.00\%$, 
and we obtain the evolution in Figure~\ref{fig:bgnpwfade6new} towards a mild
dumbbell-like shape.
\begin{figure}
\center
\begin{minipage}{0.35\textwidth}
\includegraphics[angle=-90,width=0.3\textwidth]{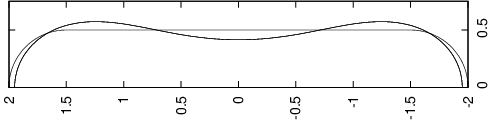}
\qquad
\includegraphics[angle=-90,width=0.35\textwidth]{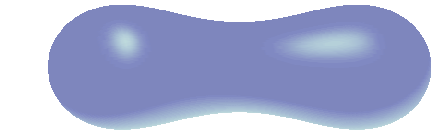} 
\end{minipage}
\begin{minipage}{0.3\textwidth}
\includegraphics[angle=-90,width=0.9\textwidth]{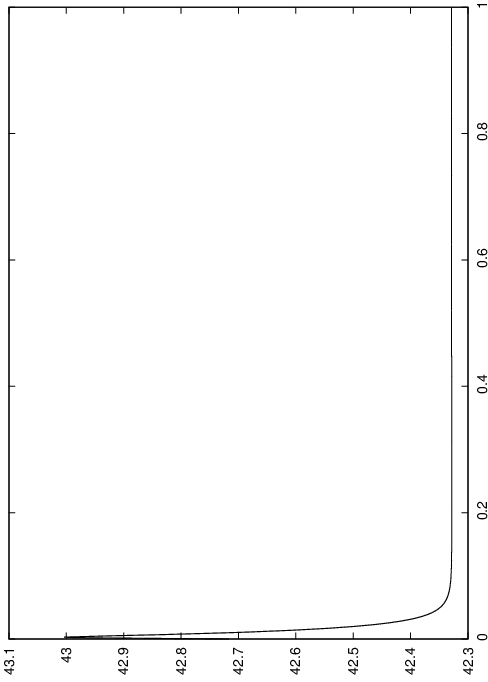}
\includegraphics[angle=-90,width=0.9\textwidth]{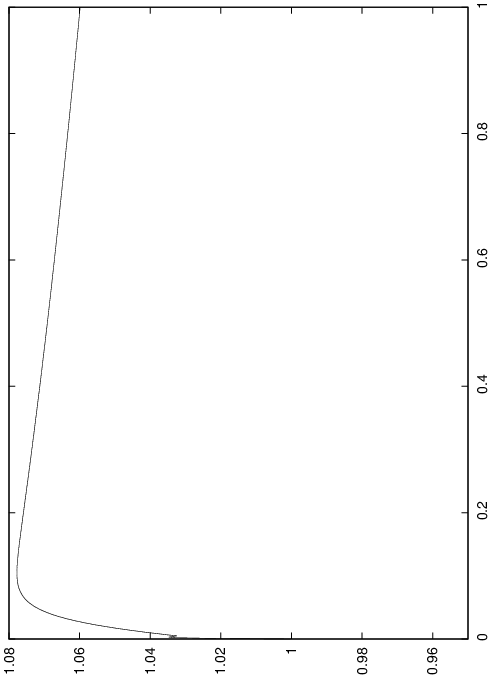}
\end{minipage}
\caption{$(\BGNwf_m^{A,V})^h$
Helfrich flow for $\spont=0$ for a rounded cylinder of dimension
$1\times4\times1$. Solution at times $t=0,0.5,1$.
We also visualize the axisymmetric surface $\mathcal{S}^m$ generated by
$\Gamma^m$ at time $t=1$.
On the right are plots of the discrete energy and of the ratio 
(\ref{eq:ratio}).} 
\label{fig:bgnpwfade6new}
\end{figure}%

\subsubsection{Genus 1 surface}

Repeating the experiment in Figure~\ref{fig:bgncigar41} for Helfrich flow,
until the earlier time of $T = 0.5$, we observe
a relative surface area loss of $0.12\%$ and a relative volume loss of
$0.00\%$. 
The evolution is shown in Figure~\ref{fig:bgncigar41_conv},
where we note that the evolution is very different from the one in
Figure~\ref{fig:bgncigar41}. In particular, the toroidal surface would like to
undergo a change of topology, and close the hole at the origin to become a
genus $0$ surface.
\begin{figure}
\center
\includegraphics[angle=-90,width=0.5\textwidth]{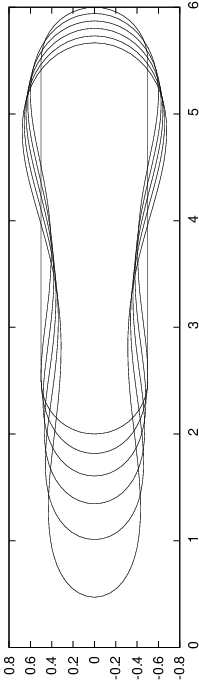} 
\quad
\includegraphics[angle=-90,width=0.25\textwidth]{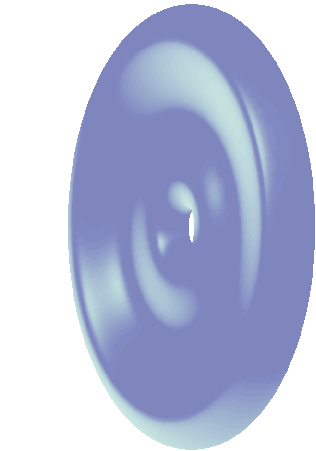} 
\\[2mm]
\includegraphics[angle=-90,width=0.3\textwidth]{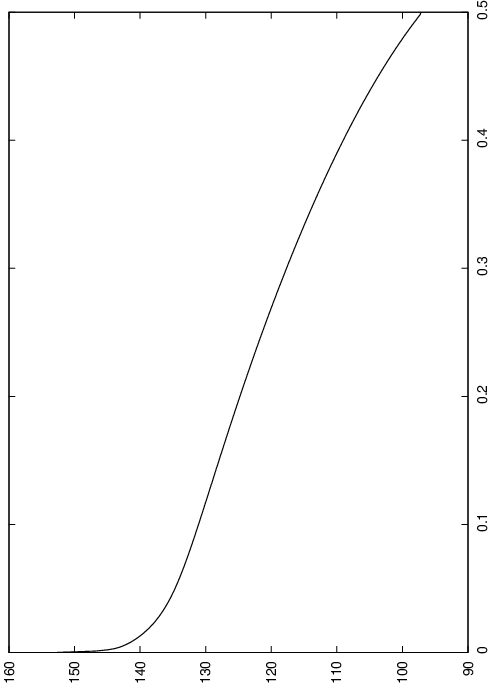}
\includegraphics[angle=-90,width=0.3\textwidth]{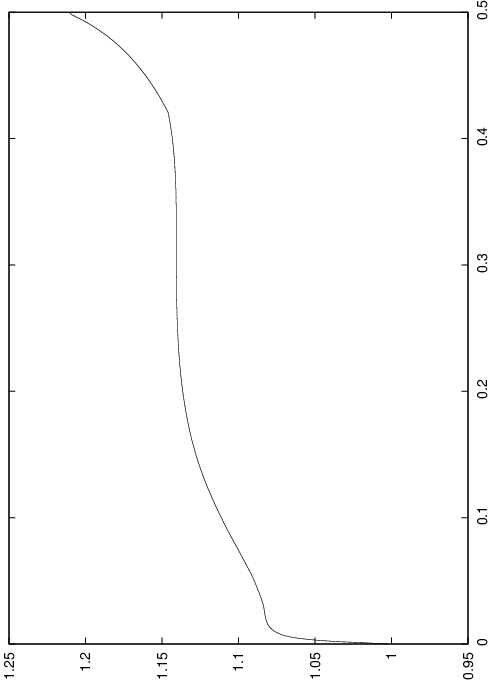}
\caption{$(\BGNwf_m^{A,V})^h$
Helfrich flow for a toroidal surface. Solution at times $t=0,0.1,\ldots,0.5$.
We also visualize the axisymmetric surface $\mathcal{S}^m$ generated by
$\Gamma^m$ at time $t=0.5$.
Below a plot of the discrete energy and of the ratio (\ref{eq:ratio}).} 
\label{fig:bgncigar41_conv}
\end{figure}%
For the smaller time steps $\ttau = 10^{-4}$ and $\ttau = 10^{-5}$,
the relative surface area loss is reduced to $0.01\%$ and $0.00\%$,
respectively, while the relative volume losses remain zero to the displayed
number of digits.

\section*{Conclusions}
We have derived and analysed various numerical schemes for the parametric
approximation of surface diffusion,
an intermediate flow between surface diffusion and conserved mean curvature
flow, Willmore flow and Helfrich flow.

As regards surface diffusion, we propose a choice between two practical and
robust schemes.
A very practical linear scheme is given by
$(\BGNsd_m)^h$. In practice the scheme is stable, and it 
asymptotically distributes the vertices uniformly.
A nonlinear scheme, for which an unconditional stability bound can be shown, is
given by $(\BGNsdstab_{m,\star})$. 
The nonlinearity in $(\BGNsdstab_{m,\star})$ is only very mild,
and so a Newton solver never takes more than three iterations in practice.
Moreover, coalescence of vertices does not occur in practice, and the ratio of
largest element/smallest element appears to asymptotically approach some
value that is significantly larger than $1$, but smaller than $10$.
Similarly to $(\BGNsdstab_{m,\star})$, we presented the scheme
$(\BGNintstab_{m,\star})$ for the approximation of the intermediate flow. Once
again, the scheme is unconditionally stable and can be easily solved for with 
a Newton method

Lastly, for Willmore flow and Helfrich flow we propose the fully practical
linear schemes $(\BGNwf_m)^h$ and $(\BGNwf_m^{A,V})^h$, respectively. Like
the scheme $(\BGNsd_m)^h$, 
they also enjoy an asymptotic equidistribution property.

\renewcommand{\theequation}{A.\arabic{equation}}
\setcounter{equation}{0}
\begin{appendix}

\section{Derivation of (\ref{eq:sdbca}) on $\partial_0 I$} \label{sec:A}
Here we demonstrate that (\ref{eq:sdweaka}) and (\ref{eq:sdstaba}) weakly
impose (\ref{eq:sdbca}) on $\partial_0 I$. 
These proofs are an extension of the proof in \cite[Appendix~A]{aximcf},
where it is shown that (\ref{eq:sdstabb}) weakly
imposes (\ref{eq:bc}). 
First we consider (\ref{eq:sdstaba}) and the case 
$\rho_0 = 0 \in \partial_0 I$.

We assume for almost all $t\in(0,T)$ that
$\pol x(t) \in [C^1(\overline I)]^2$, $\varkappa_S(t) \in C^1(\overline I)$
and $\pol x_t(t) \,.\,\pol \nu(t) \in L^\infty(I)$. 
These assumptions and (\ref{eq:xrho}) imply that 
\begin{equation} \label{eq:x1bound}
C_1\,\rho \leq |\pol x(\rho,t)\,.\,\pol\ek_1 | \leq C_2\,\rho
\qquad \forall\ \rho \in [0,\overline\rho]\,,
\end{equation}
for $\overline\rho$ sufficiently small, and for almost all $t\in(0,T)$.

Let $t \in (0,T)$.
For a fixed $\overline\rho > 0$ and $\epsilon \in (0,\overline\rho)$, we define
\begin{equation*} 
\chi_\epsilon(\rho) = \begin{cases}
(\overline\rho)^{-1}\,\int_\epsilon^{\overline\rho} (\pol
x(z,t)\,.\,\pol\ek_1)^{-1} \;{\rm d}z & 0 \leq \rho < \epsilon\,, \\
(\overline\rho)^{-1}\,\int_\rho^{\overline\rho} (\pol
x(z,t)\,.\,\pol\ek_1)^{-1} \;{\rm d}z & \epsilon \leq \rho < \overline\rho\,, \\
0 & \overline\rho \leq \rho\,.
\end{cases}
\end{equation*}
We observe that (\ref{eq:x1bound}) implies that
$(\pol x\,.\,\pol\ek_1)\,\chi_\epsilon$ is integrable in the limit
$\epsilon \to 0$.
On choosing 
$\chi = \chi_\epsilon \in H^1(I)$ in (\ref{eq:sdstaba}), we obtain in the
limit $\epsilon \to 0$ that
\begin{equation} \label{eq:app1}
(\overline\rho)^{-1}\,
\int_0^{\overline\rho} (\pol x\,.\,\pol\ek_1)\,\pol x_t\,.\,\pol\nu\,
\left(\int_\rho^{\overline\rho} (\pol x\,.\,\pol\ek_1)^{-1} \;{\rm d}z\right)
|\pol x_\rho| \drho
= - (\overline\rho)^{-1}\,\int_0^{\overline\rho} 
(\varkappa_{\mathcal{S}})_\rho\,|\pol x_\rho|^{-1} \drho\,.
\end{equation}
Applying Fubini's theorem and noting (\ref{eq:x1bound}), as
well as the boundedness of $|\pol x_\rho|$ and
$\pol x_t\,.\,\pol\nu$, yields 
that
\begin{equation} \label{eq:fubini}
\left|
(\overline\rho)^{-1}\,
\int_0^{\overline\rho} (\pol x\,.\,\pol\ek_1)\,\pol x_t\,.\,\pol\nu\,
\left(\int_\rho^{\overline\rho} (\pol x\,.\,\pol\ek_1)^{-1} \;{\rm d}z\right)
|\pol x_\rho| \drho \right| 
= \left| (\overline\rho)^{-1}\,
\int_0^{\overline\rho} 
(\pol x\,.\,\pol\ek_1)^{-1}
\left(\int_0^z
(\pol x\,.\,\pol\ek_1)\,\pol x_t\,.\,\pol\nu\,|\pol x_\rho| \drho
\right) {\rm d}z\right| 
\to 0 \quad \text{as }\ \overline\rho\to0\,.
\end{equation}
On the other hand, the right hand side in (\ref{eq:app1}) converges to
$(\varkappa_{\mathcal{S}})_\rho(0,t)\,|\pol x_\rho(0,t)|^{-1}$
as $\overline\rho\to0$, on recalling the smoothness assumptions on
$\varkappa_{\mathcal{S}}$ and $\pol x$. 
Combining this with (\ref{eq:fubini}) and 
(\ref{eq:xrho}) yields the boundary condition 
(\ref{eq:sdbca}) for $\rho = 0 \in \partial_0 I$. The proof for
$\rho = 1 \in \partial_0 I$ is analogous.

The proof for (\ref{eq:sdweaka}) is identical, 
on assuming that $(\varkappa - 
\frac{\pol\nu\,.\,\pol\ek_1}{\pol x\,.\,\pol\ek_1})(t) \in C^1(\overline I)$ 
for almost all $t\in(0,T)$.
Finally we note that the above proof also shows that (\ref{eq:intstaba}) 
weakly imposes (\ref{eq:intbca}) on $\partial_0 I$,
on assuming that $y(t) \in C^1(\overline I)$ for almost all $t\in(0,T)$.

\renewcommand{\theequation}{B.\arabic{equation}}
\setcounter{equation}{0}
\section{Some axisymmetric differential geometry} \label{sec:B}
Let $\pol x : \overline I \to \bR^2$ parameterize $\Gamma$,
the generating curve of a surface $\mathcal{S}$. Then
$\pol y : \overline I \times [0,2\,\pi) \to \bR^3$ parameterizes
$\mathcal{S}$, where
\begin{equation} \label{eq:vecy}
\pol y(\rho,\theta) = (\pol x(\rho)\,.\,\pol\ek_1\,\cos\theta, 
\pol x(\rho)\,.\,\pol\ek_2, \pol x(\rho)\,.\,\pol\ek_1\,\sin\theta)^T 
\,.
\end{equation}
On recalling that $\partial_s = |\pol{x}_\rho|^{-1}\,\partial_\rho$,
we note that 
\begin{equation} \label{eq:vecys}
|\pol y_s|^2 = 1\,,\quad 
|\pol y_\theta|^2 = (\pol x\,.\,\pol\ek_1)^2\,,\quad
\pol y_s\,.\,\pol y_\theta = 0\,.
\end{equation}
In what follows, we often identify a function $f$ defined on
$\overline I \times [0,2\,\pi)$ with the function 
$f \circ \pol y^{-1}$, defined on $\mathcal{S}$. For example, it 
follows from (\ref{eq:vecys}) that
\begin{equation*} 
\nabS\,f = f_s\,\pol y_s + (\pol x\,.\,\pol\ek_1)^{-2}\,
f_\theta\,\pol y_\theta\,.
\end{equation*}
Similarly, 
\begin{equation*} 
\nabS\,.\,\pol f = \pol f_s\,.\,\pol y_s + (\pol x\,.\,\pol\ek_1)^{-2}\,
\pol f_\theta\,.\,\pol y_\theta\,,
\end{equation*}
and so, on noting $((\pol x\,.\,\pol\ek_1)^{-1}\,\pol y_\theta)_s=\pol 0$ 
and $(\pol y_s)_\theta \,.\, \pol y_\theta = (\pol x\,.\,\pol\ek_1)\, 
\pol x_s\,.\,\pol\ek_1$, we obtain that
\begin{equation*} 
\Delta_{\mathcal S}\,f = \nabS\,.\,(\nabS\, f) = f_{ss} + 
\frac{\pol x_s\,.\,\pol\ek_1}{\pol x\,.\,\pol\ek_1}\,f_s + 
(\pol x\,.\,\pol\ek_1)^{-2}\,f_{\theta\theta}\,.
\end{equation*}
For a radially symmetric function $f$, with $f(\rho,\theta) = f(\rho,0)$ 
for all $(\rho,\theta) \in \overline I \times [0,2\,\pi)$, it follows that
\begin{equation} \label{eq:LBSrad}
\Delta_{\mathcal S}\,f 
= (\pol x\,.\,\pol\ek_1)^{-1} \, ( \pol x\,.\,\pol\ek_1\,f_s)_s\,.
\end{equation}

We remark that a derivation of (\ref{eq:kappaS}),
recall also (\ref{eq:LBid}), is obtained by combining
(\ref{eq:vecy}) and (\ref{eq:LBSrad}) to yield, 
on recalling (\ref{eq:varkappa}), (\ref{eq:tau}) and (\ref{eq:nuS}), that 
\begin{align}
\Delta_{\mathcal{S}}\,\pol y & = 
\pol y_{ss} + 
\frac{\pol x_s\,.\,\pol\ek_1}{\pol x\,.\,\pol\ek_1}\,\pol y_s + 
(\pol x\,.\,\pol\ek_1)^{-2}\,\pol y_{\theta\theta}
= 
\begin{pmatrix}
\pol x_{ss}\,.\,\pol\ek_1\,\cos\theta \\
\pol x_{ss}\,.\,\pol\ek_2 \\
\pol x_{ss}\,.\,\pol\ek_1\,\sin\theta 
\end{pmatrix}
+ (\pol x\,.\,\pol\ek_1)^{-1}
\begin{pmatrix}
(\pol x_s\,.\,\pol\ek_1)^2\,\cos\theta \\
(\pol x_s\,.\,\pol\ek_1)\,\pol x_s\,.\,\pol\ek_2 \\
(\pol x_s\,.\,\pol\ek_1)^2\,\sin\theta 
\end{pmatrix}
- (\pol x\,.\,\pol\ek_1)^{-1}
\begin{pmatrix} \cos\theta \\ 0 \\ \sin\theta \end{pmatrix}
\nonumber \\ & 
= \varkappa\,
\begin{pmatrix}
\pol\nu\,.\,\pol\ek_1\,\cos\theta \\
\pol\nu\,.\,\pol\ek_2 \\
\pol\nu\,.\,\pol\ek_1\,\sin\theta 
\end{pmatrix}
- (\pol x\,.\,\pol\ek_1)^{-1}
\begin{pmatrix}
(\pol x_s\,.\,\pol\ek_2)^2\,\cos\theta \\
- (\pol x_s\,.\,\pol\ek_1)\,\pol x_s\,.\,\pol\ek_2 \\
(\pol x_s\,.\,\pol\ek_2)^2\,\sin\theta 
\end{pmatrix} 
=\varkappa\,
\begin{pmatrix}
(\pol\nu\,.\,\pol\ek_1)\,\cos\theta \\
\pol\nu\,.\,\pol\ek_2 \\
(\pol\nu\,.\,\pol\ek_1)\,\sin\theta 
\end{pmatrix}
- \frac{\pol x_s\,.\,\pol\ek_2}{\pol x\,.\,\pol\ek_1}
\begin{pmatrix}
\pol x_s\,.\,\pol\ek_2\,\cos\theta \\
- \pol x_s\,.\,\pol\ek_1 \\
\pol x_s\,.\,\pol\ek_2\,\sin\theta 
\end{pmatrix} \nonumber \\ & 
= \left(\varkappa - \frac{\pol\nu\,.\,\pol\ek_1}{\pol x\,.\,\pol\ek_1}
\right)
\begin{pmatrix}
\pol\nu\,.\,\pol\ek_1\,\cos\theta \\
\pol\nu\,.\,\pol\ek_2 \\
\pol\nu\,.\,\pol\ek_1\,\sin\theta 
\end{pmatrix} 
= \left(\varkappa - \frac{\pol\nu\,.\,\pol\ek_1}{\pol x\,.\,\pol\ek_1}
\right) \pol\nu_{\mathcal{S}} \,.
\label{eq:LBSvecy}
\end{align}

\end{appendix}

\noindent
{\large\bf Acknowledgements}\\
The authors gratefully acknowledge the support 
of the Regensburger Universit\"atsstiftung Hans Vielberth.


\providecommand\noopsort[1]{}\def\soft#1{\leavevmode\setbox0=\hbox{h}\dimen7=\ht0\advance
  \dimen7 by-1ex\relax\if t#1\relax\rlap{\raise.6\dimen7
  \hbox{\kern.3ex\char'47}}#1\relax\else\if T#1\relax
  \rlap{\raise.5\dimen7\hbox{\kern1.3ex\char'47}}#1\relax \else\if
  d#1\relax\rlap{\raise.5\dimen7\hbox{\kern.9ex \char'47}}#1\relax\else\if
  D#1\relax\rlap{\raise.5\dimen7 \hbox{\kern1.4ex\char'47}}#1\relax\else\if
  l#1\relax \rlap{\raise.5\dimen7\hbox{\kern.4ex\char'47}}#1\relax \else\if
  L#1\relax\rlap{\raise.5\dimen7\hbox{\kern.7ex
  \char'47}}#1\relax\else\message{accent \string\soft \space #1 not
  defined!}#1\relax\fi\fi\fi\fi\fi\fi}

\end{document}